\documentclass[11pt,reqno,twoside]{article}

\usepackage{microtype}
\usepackage{mysty}
\usepackage{graphicx}
\usepackage{subfig}
\usepackage{booktabs}
\usepackage[margin=1.0in]{geometry}

\usepackage{mathptmx}
\usepackage[T1]{fontenc}

\usepackage{authblk}

\allowdisplaybreaks

\usepackage{hyperref}

\allowdisplaybreaks

\makeatletter
\newcommand{\printfnsymbol}[1]{%
  \textsuperscript{\@fnsymbol{#1}}%
}
\makeatother

\usepackage{titlesec}
\titlespacing*{\section}{0pt}{1.25\baselineskip}{0.25\baselineskip}
\titlespacing*{\subsection}{0pt}{0.75\baselineskip}{0.125\baselineskip}
\titlespacing*{\subsection}{0pt}{0.5\baselineskip}{0.125\baselineskip}
\titlespacing*{\paragraph}{0pt}{0.25\baselineskip}{0.25\baselineskip}

\allowdisplaybreaks

\begin{document}

\title{SPRING: A fast stochastic proximal alternating method for non-smooth non-convex optimization}

\date{}
\author[1]{Derek Driggs
\thanks{Contributed Equally}}
\author[2]{Junqi Tang\printfnsymbol{1}}
\author[3]{Jingwei Liang}
\author[2]{Mike Davies}
\author[1]{Carola-Bibiane Sch\"onlieb}
\affil[1]{Department of Applied Mathematics and Theoretical Physics, University of Cambridge}
\affil[2]{School of Engineering, University of Edinburgh}
\affil[3]{School of Mathematical Sciences, Queen Mary University of London}

\renewcommand\Authands{ and }

\vspace{-5mm}

    \maketitle
    \begin{abstract}
    We introduce SPRING, a novel stochastic proximal alternating linearized minimization algorithm for solving a class of non-smooth and non-convex optimization problems. Large-scale imaging problems are becoming increasingly prevalent due to advances in data acquisition and computational capabilities. Motivated by the success of stochastic optimization methods, we propose a stochastic variant of proximal alternating linearized minimization (PALM) algorithm \cite{bolte2014proximal}. We provide global convergence guarantees, demonstrating that our proposed method with variance-reduced stochastic gradient estimators, such as SAGA \cite{SAGA} and SARAH \cite{sarah}, achieves state-of-the-art oracle complexities. We also demonstrate the efficacy of our algorithm via several numerical examples including sparse non-negative matrix factorization, sparse principal component analysis, and blind image deconvolution.
    \end{abstract}

\section{Introduction}\label{sec:introduction}

With the advent of large-scale machine learning, developing efficient and reliable algorithms for (empirical) risk minimization has become an intense focus of the optimization community. These tasks involve minimizing a loss function measuring the fit between observed data, $x$, and a model's predicted result, $b$: $\min_{x \in \R^{m_1}} ~ \frac{1}{n} \sum_{i=1}^n \mathcal{L}(x_i,b_i) $ where $n$ denotes the number of samples and $\mathcal{L}$ is the loss function. The two defining qualities of these problems are their large scale (in many applications, $n$ is on the order of billions), and finite-sum structure.

When the value of $n$ above is very large, computing the gradient of the loss function is often prohibitively expensive, rendering most traditional deterministic first-order optimization algorithms ineffective. Over the years, randomized optimization algorithms \cite{bottou2010large,sgd} have become increasingly popular due to their efficiency and simplicity. For these algorithms, the full gradient is replaced by a stochastic approximation that is cheap to compute, so that their per-iteration complexity grows slowly with $n$. For objectives with a finite-sum structure, many works have shown that certain randomized algorithms achieve convergence rates similar to those of full-gradient methods, even though their per-iteration complexity is often a factor of $n$ smaller \cite{SAGA,svrg,proxSVRG}.

Outside machine learning, objectives with a finite-sum structure also arise in problems from image processing and computer vision. Recently, randomized optimization algorithms have been explored for image processing tasks including PET reconstruction, deblurring and tomography \cite{spdhg,tang2019}. As stochastic methods expand into new applications, they move further from smooth, strongly convex finite-sum objectives where they are well-understood theoretically. In this work, we aim to provide a better understanding of stochastic algorithms for problems that are neither smooth nor convex.

\subsection{Non-smooth, non-convex optimization}

Our goal is to minimize composite objectives of the following form:
\begin{equation}\label{eq:J-F-R}
\begin{aligned}
\min_{x\in\bbR^{m_1}, y \in \bbR^{m_2}}
\Ba{\Phi(x, y) & \eqdef J(x) + F(x, y) + R(y) },
\end{aligned}
\end{equation}
where $F(x,y) \defeq \frac{1}{n} \sum_{i=1}^n F_i(x, y)$ has a finite-sum structure.
In general, functions $J$ and $R$ are non-smooth regularizations that promote structures, such as sparsity or non-negativity, in the solutions. The blocks $x$ and $y$ represent differently structured elements of the solution that are coupled through the loss term, $F(x,y)$. Throughout this work, we impose the following assumptions:
\begin{enumerate}[label= ({\textbf{A.\arabic{*}})},ref= \textbf{A.\arabic{*}}, leftmargin=3.75em]
\item\label{item:A-JR}
$J: \bbR^{m_1} \to \bbR\cup\ba{\pinf}$ and $R: \bbR^{m_2} \to \bbR\cup\ba{\pinf}$ are proper lower semi-continuous (lsc) functions that are bounded from below;
\item\label{item:A-F1}
$F_i: \bbR^{m_1}\times \bbR^{m_2} \to \bbR$ are finite-valued, differentiable, and their gradients $\nabla F_i$ are $M$-Lipschitz continuous on bounded sets of $\bbR^{m_1} \times \bbR^{m_2}$ for all $i \in \{1, \cdots, n\}$;
\item\label{item:A-F2}
The partial gradients $\nablax F_i$ are Lipschitz continuous with modulus $L_1(y)$, and $\nablay F_i$ are Lipschitz continuous with modulus $L_2(x)$ for all $i \in \{1, \cdots, n\}$;
\item \label{item:A-Flast}
The function $\Phi$ is bounded from below.
\end{enumerate}
No convexity is imposed on any of the functions involved.
Problem \eqref{eq:J-F-R} departs from the sum-of-convex-objectives models that populate the majority of the optimization literature. Many models in machine learning, statistics and image processing require the full generality of \eqref{eq:J-F-R}. Archetypal examples include non-negative or sparse matrix factorization \cite{hoyer2004non}, Sparse PCA \cite{d2005direct,zou2006sparse}, Robust PCA \cite{candes2011robust}, trimmed least-squares \cite{aravkin2019trimmed} and blind image deconvolution \cite{campisi2016blind}. Despite the prevalence of these problems, few numerical methods can solve the general problem \eqref{eq:J-F-R}, and none that realize match the efficiency that randomized algorithms provide. We outline some existing options below.

\paragraph{Proximal alternating minimization}

One approach to solve \eqref{eq:J-F-R} is the Proximal Alternating Minimization (PAM) method \cite{attouch2010proximal}, whose iterations take the following form:
\begin{equation}
\begin{aligned} \label{eq:pam}
\xkp &\in  \Argmin_{x \in \R^{m_1}} \Ba{\Phi(x, \yk) + \tfrac{1}{2\gmaxk} \norm{x - \xk}^2} ,  \\
\ykp &\in  \Argmin_{y \in \R^{m_2}} \Ba{\Phi(\xkp, y) + \tfrac{1}{2\gmayk} \norm{y - \yk}^2}  ,
\end{aligned}
\eeq
where $\gmaxk, \gmayk > 0$ are step-sizes. A significant limitation of PAM is that the subproblems in \eqref{eq:pam} do not have closed-form solutions in general. As a consequence, each subproblem requires its own set of inner iterations, which makes PAM inefficient in practice.

\paragraph{Proximal alternating linearized minimization \cite{bolte2014proximal}}

To circumvent this limitation of PAM, Proximal Alternating Linearized Minimization (PALM) \cite{bolte2014proximal} replaces PAM's two subproblems with their proximal linearizations. PALM's iterations take the form
\begin{equation}
\begin{aligned}
\label{eq:palm}
\xkp &\in \prox_{\gamma_{x,k} J} \Pa{\xk - \gamma_{x,k} \nablax F (\xk, \yk)} , \\
\ykp &\in \prox_{\gmayk R} \Pa{\yk - \gmayk \nablay F (\xkp, \yk)}  ,
\end{aligned}
\end{equation}
where $\nablax F$ and $\nablay F$ are partial derivatives, and $\prox_{\gamma_{x,k} J}$ is called ``proximal operator'' of $J$ and defined by
$$
\prox_{\gamma J} (\cdot) \defeq \Argmin_x \gamma J(x) + \tfrac{1}{2} \norm{x - \cdot}^2 .
$$
The proximal mapping is set-valued and becomes single-valued if $J$ is convex.

In contrast to PAM, each subproblem of PALM can be efficiently computed if the proximal maps of $J$ and $R$ are easy to calculate, which is true in many applications.
PALM also has the same convergence guarantees as PAM, so linearizing $F$ in each proximal step is a clear improvement over PAM. PALM with momentum is considered in \cite{pock2016inertial}, where the authors show that inertia allows PALM to converge to critical points with lower objective values, although accelerated rates might not be obtained.

\subsection{Stochastic PALM}

In this work, we introduce SPRING, a randomized version of PALM where the partial gradients $\nablax F(x_k,y_k)$ and $\nablay F(x_{k+1},y_k)$ in \eqref{eq:palm} are replaced by random estimates, $\tnablax (x_k,y_k)$ and $\tnablay (x_{k+1},y_k)$, formed using the gradients of only a few indices $\nablax F_j(x_k, y_k)$ and $\nablay F_j(x_{k+1}, y_k)$ for $j \in B_k \subset \{1,2,\cdots,n\}$. The mini-batch $B_k$ is chosen uniformly at random from all subsets of $\{1,2,\cdots,n\}$ with cardinality $b$. We describe SPRING in Algorithm \ref{alg:spring} below.

\begin{algorithm}
\caption{SPRING: \textbf{S}tochastic \textbf{Pr}oximal Alternat\textbf{ing} Linearized Minimization}
\label{alg:spring}
\begin{algorithmic}
\STATE{{\noindent{\bf{Initialize}}}: $x_{0} \in \bbR^{m_1}, y_{0} \in \bbR^{m_2}$.}
\FOR{$k = 1,2, \cdots, T-1$}
\STATE{$\xkp \in \prox_{\gamma_{x,k} J} \Pa{\xk - \gmaxk \tnablax (\xk, \yk)}$}
\STATE{$\ykp \in \prox_{\gamma_{y,k} R} \Pa{\yk - \gmayk \tnablay (\xkp, \yk)}$}
\ENDFOR
\RETURN $(x_T, y_T)$
\end{algorithmic}
\end{algorithm}

Many different gradient estimators in the literature can be used for SPRING. The simplest one is the stochastic gradient descent (SGD) estimator \cite{robbins1971},
\begin{equation*}
    \tnablaxsgd(\xk, \yk) = \tfrac{1}{b} \msum_{j \in B_k} \nablax F_j (\xk, \yk) ,
\end{equation*}
which uses the gradient of a randomly sampled batch to represent the full gradient.
Another popular choice is SAGA gradient estimator \cite{SAGA}, which incorporates the gradient history:
\begin{equation*}
\begin{aligned}
    \tnablaxsaga(\xk, \yk) & = \tfrac{1}{b} \ssum_{j \in B_k} \Pa{ \nablax F_j (\xk, \yk) - g_{k,j} } + \tfrac{1}{n} \msum_{i=1}^n g_{k,i}, \\
    g_{k+1,i} & = \left\{ \begin{aligned}  &\nablax F_i(\xk, \yk) && \textrm{if~~} i \in B_k  ,  \\ &g_{k,i} && \textrm{o.w.} \end{aligned} \right .
\end{aligned}
\end{equation*}
Both SGD and SAGA estimators are {\it unbiased}. The last gradient estimator we specifically consider in this work is the (loopless) SARAH estimator \cite{l2s,sarah}, $\tnablaxsarah(\xk, \yk)$, which is {\it biased}.
\begin{equation*}
\begin{aligned}
	\begin{cases}
	   \nablax F(x_k,y_k)  & \textrm{w.p. } \tfrac{1}{p} \\
	    \tfrac{1}{b} \ssum_{j \in B_k}  \Pa{ \nablax F_j(x_k, y_k) - \nablax F_j(x_{k-1},y_{k-1}) } + \tnablaxsarah(x_{k-1},y_{k-1}) & \textrm{o.w.}
    \end{cases}
\end{aligned}
\end{equation*}
Here, $p$ is a tuning parameter that is generally set to $\mathcal{O}(n)$. Other popular estimators that can be used in SPRING but that we do not specifically consider include SAG \cite{sag} and SVRG \cite{svrg}.

Computing the full gradient is generally $n$-times more expensive than computing $\nablax F_i$, so when $n$ is large and $b \ll n$, each step of SPRING with any of these estimators is significantly less expensive than that of PALM.

\begin{remark}
    Although we consider only two variable blocks in \eqref{eq:J-F-R}, the results of this paper easily extend to an arbitrary number of blocks to solve problems of the form
    \begin{equation*}
        \min_{ x_1, \cdots, x_\ell } \bBa{ \tfrac{1}{n} \msum_{i=1}^n F_i(x_1, \cdots, x_\ell ) + \msum_{t=1}^\ell R_t(x_t)},
    \end{equation*}
    where each $R_t$ is a (possibly non-smooth) regularizer.
\end{remark}

\subsection{Contributions}

In this work, we combine PALM with popular stochastic gradient estimators and show that the resulting algorithm matches the convergence rates of PALM given that the gradient estimators $\tnablax$ and $\tnablay$ satisfy a \emph{variance-reduced} property (see Definition \ref{def:varred}). We prove convergence guarantees of two types.

\paragraph{Convergence rate of generalized gradient map}
Given $z_k = (x_k, y_k)$, the \emph{generalized gradient map} is defined as
\begin{equation}\label{eq:ggmap}
    \mathcal{G}_{\gamma_1,\gamma_2}(z_k) \defeq
    \begin{pmatrix}
    1 / \gamma_1 \big( x_k - \prox_{\gamma_1 J} (x_k - \gamma_1 \nablax F(x_k,y_k) ) \big) \\
    1 / \gamma_2 \big( y_k - \prox_{\gamma_2 R} (y_k - \gamma_2 \nablay F(x_{k+1},y_k) \big) )
    \end{pmatrix},
\end{equation}
where $\gamma_1, \gamma_2 > 0$ are parameters (not necessarily equal to the algorithm's step-sizes), and a point $z = (x,y)$ an \emph{$\epsilon$-approximate critical point} if it satisfies $\EE \dist(0,\mathcal{G}_{\gamma_1,\gamma_2}(z)) \le \epsilon$ \emph{for some} $\gamma_1, \gamma_2 > 0$. In Section \ref{sec:conv}, we show that
\begin{equation*}
\label{eq:gengrad}
    \EE [\dist(0,\mathcal{G}_{\frac{\gamma_{x,\alpha}}{2},\frac{\gamma_{y,\alpha}}{2}} (z_\alpha))^2] \le \mathcal{O} \big( \tfrac{1}{k} \big),
\end{equation*}
where $\alpha$ is chosen uniformly at random from the set $\{1, 2, \cdots, k\}$.
If $\Phi$ satisfies a certain error bound (see Eq. \eqref{eq:errb}), then SPRING converges linearly to the global optimum.
These results generalize almost all existing results for stochastic gradient methods on non-convex, non-smooth objectives \cite{aravkin2019trimmed,spider,reddi,spiderboost,spiderm}.

Specializing these convergence guarantees to specific gradient estimators, the constants appearing in these rates scale with the mean-squared error (MSE, see Definition \ref{def:varred}) of the gradient estimators.
\begin{itemize}
    \item For SAGA estimator with $b \le \mathcal{O}(n^{2/3})$, the iterates of SPRING satisfy
\begin{equation*}
    \EE [\dist(0, \mathcal{G}_{\frac{\gamma_{x,\alpha}}{2},\frac{\gamma_{y,\alpha}}{2}}(z_\alpha))^2]
    \le \mathcal{O} \big( \tfrac{n^2 L}{b^3 k} \big) .
\end{equation*}
\item For SARAH gradient estimator with any batch size, we have
\begin{equation*}
    \EE [ \dist(0, \mathcal{G}_{\frac{\gamma_{x,\alpha}}{2},\frac{\gamma_{y,\alpha}}{2}}(z_\alpha) )^2] \le \mathcal{O} \big( \tfrac{\sqrt{n} L}{k} \big).
\end{equation*}
\end{itemize}
These convergence rates imply complexity bounds with respect to a \emph{stochastic first-order oracle} (SFO) which returns the partial gradient of a single component $F_i$ (for example, $\nablax F_i(x_k, y_k)$). To find an $\epsilon$-approximate critical point, SAGA with a mini-batch of size $n^{2/3}$ requires no more than $\mathcal{O}( n^{2/3} L / \epsilon^2)$ SFO calls, and SARAH requires no more than $\mathcal{O}( \sqrt{n} L / \epsilon^2 )$. The improved dependence on $n$ when using SARAH gradient estimator exists in all of our convergence rates for SPRING. Because most existing works on stochastic optimization for non-smooth, non-convex problems use models that are special cases of \eqref{eq:J-F-R}, our results for SPRING capture most existing work as special cases. In particular, in the case $R \equiv J \equiv 0$, our results recover recent results showing that SARAH achieves the \emph{oracle complexity lower-bound} for non-convex problems with a finite-sum structure \cite{spider,proxsarah,spiderboost,zhou2019lower,spiderm}.

\paragraph{Convergence under the \KL~property}
We also provide convergence guarantees under the \KL~property (see Definition \ref{defn:KLp}).
First, we prove the global convergence of the generated sequence under the assumption that the objective function $\Phi(x,y)$ of \eqref{eq:J-F-R} has the \KL property.
Then, under the assumption that $\Phi$ is semi-algebraic with KL-exponent $\theta$ (see Section \ref{sec:prelim}), we show that the sequence $z_k = (x_k,y_k)$ generated by SPRING converges in expectation to a critical point $z^\star$ of problem \eqref{eq:J-F-R} at the following rates:
\begin{itemize}
    \item If $\theta = 0$, then $\{\EE \Phi(z_k)\}_{k \in \N}$ converges to $\EE \Phi(z^\star)$ in a finite number of steps.
    \item If $\theta \in (0,1/2]$, then $\EE \|z_k - z^\star\| \le \mathcal{O}(\tau^k)$ for some $\tau \in (0, 1)$.
    \item If $\theta \in (1/2, 1)$, then $\EE \|z_k - z^\star\| \le \mathcal{O} ( k^{- \frac{1 - \theta}{2 \theta - 1}})$.
\end{itemize}
These rates match the rates of the original PALM algorithm.

\subsection{Prior Art}

SPRING offers several advantages over existing stochastic algorithms for non-smooth non-convex optimization. Reddi \emph{et al.\!} investigate proximal SAGA and SVRG for solving problems of the form \eqref{eq:J-F-R} when $y$ is constant and $J$ is convex \cite{reddi}. Using mini-batches of size $b = n^{2/3}$, SAGA and SVRG require $\mathcal{O}(n^{2/3} L / \epsilon^2)$ stochastic gradient evaluations to converge to an $\epsilon$-approximate critical point. Similarly, Aravkin and Davis introduce TSVRG, a stochastic algorithm based on SVRG gradient estimator, for solving another special case of \eqref{eq:J-F-R} \cite{aravkin2019trimmed}. Our work generalizes their results and improves them in many cases. Most importantly, we show that using SARAH gradient estimator allows SPRING to achieve a complexity of $\mathcal{O}(\sqrt{n} L / \epsilon^2)$ even when the mini-batch size is equal to one. Our results for semi-algebraic objectives offer even sharper convergence rates.

The block stochastic gradient method \cite{xu2015} is closely related to SPRING using the (non-variance-reduced) SGD gradient estimator. In a similar work, Davis \emph{et al.\!} introduce SAPALM, an asynchronous version of PALM that allows stochastic noise in the gradients \cite{davis2016sound}. The authors prove convergence rates that scale with the variance of the noise in the gradients, with their best complexity bound for finding an $\epsilon$-approximate critical point equal to $\mathcal{O}(n L / \epsilon^2)$. While significant in their own right, these results are not directly related to ours, as these works require an explicit bound on the variance of the noise in the gradients, and the gradient estimators we consider do not admit such a bound~\cite{davis2016sound}.

\section{Preliminaries}
\label{sec:prelim}

We use the following definitions and notation throughout the manuscript.

\paragraph{Variance Reduction}
In our analysis, we mainly focus on stochastic gradient estimators that are variance reduced. We use a general definition of variance-reduced gradient estimator that includes all existing estimators, for example, SAGA and SARAH, as special cases.

\begin{definition}[Variance-reduced gradient estimator]
\label{def:varred}
    A gradient estimator $\tnabla$ is \emph{variance-reduced} with constants $V_1, V_2, V_\Upsilon \ge 0$, and $\rho \in (0,1]$ if it satisfies the following conditions:
    \begin{enumerate}
    \item {\bf (MSE Bound)} There exists a sequence of random variables $\{\Upsilon_k\}_{k \ge 1}$ of the form $\Upsilon_k = \sum_{i=1}^s \|v_k^i\|^2$ for some random vectors $v_k^i$ such that
    \begin{equation}
        \begin{aligned}
        \label{eq:varreduc}
        & \E [ \|\tnablax (x_k,y_k) - \nablax F(x_k,y_k) \|^2 + \|\tnablay (x_{k+1},y_k) - \nablay F(x_{k+1},y_k) \|^2] \\ %
        &\le \Upsilon_k + V_1 (\E \|z_{k+1} - z_k\|^2 + \|z_k - z_{k-1}\|^2),
        \end{aligned}
    \end{equation}
    and, with $\Gamma_k = \sum_{i=1}^s \|v_k^i\|$,
    \begin{equation*}
        \begin{aligned}
        & \E [\|\tnablax (x_k,y_k) - \nablax F(x_k,y_k) \| + \|\tnablay (x_{k+1},y_k) - \nablay F(x_{k+1},y_k) \|] \\ %
        &\le \Gamma_k + V_2 (\E \|z_{k+1} - z_k\| + \|z_k - z_{k-1}\|).
        \end{aligned}
    \end{equation*}
    \item {\bf (Geometric Decay)} The sequence $\{\Upsilon_k\}_{k \ge 1}$ decays geometrically:
    \begin{equation}
    \label{eq:geo}
        \E \Upsilon_{k+1} \le (1 - \rho) \Upsilon_k + V_\Upsilon (\E \|z_{k+1} - z_k\|^2 + \|z_k - z_{k-1}\|^2).
    \end{equation}
    \item {\bf (Convergence of Estimator)} For all sequences $\{z_k\}_{k \ge 1}$ satisfying $\lim_{k \to \infty} \EE \|z_k - z_{k-1}\|^2 = 0$, it follows that $\EE \Upsilon_k \to 0$ and $\EE \Gamma_k \to 0$.
    \end{enumerate}
\end{definition}

\begin{proposition}
\label{prop:const}
    SAGA gradient estimator is variance-reduced with parameters $V_1 = 6 M^2/b$, $V_2 = \sqrt{6} M / \sqrt{b}$, $V_\Upsilon = \frac{134 n L^2}{b^2}$, and $\rho = \frac{b}{2 n}$. SARAH estimator is variance-reduced with parameters $V_1 = V_\Upsilon = 2 L^2$, $V_2 = 2 L$, and $\rho = 1/p$.
\end{proposition}

Proposition \ref{prop:const} is a generalization of existing variance bounds for these estimators. For a derivation of the constants appearing in Proposition \ref{prop:const}, we refer to Appendix \ref{app:sagaconst} for the SAGA estimator and Appendix \ref{app:sarahconst} for the SARAH estimator.

\begin{remark}
    Our results allow Algorithm \ref{alg:spring} to use any variance-reduced gradient estimator, even different estimators for $\nablax$ and $\nablay$. In particular, it is possible to use different mini-batch sizes when approximating the two partial gradients.
\end{remark}

\paragraph{\KL property}
Let $H : \bbR^{m_1} \to \bbR \cup \ba{\pinf}$ be a proper lower semicontinuous function. For $\epsilon_1, \epsilon_2$ satisfying $\ninf < \epsilon_1 < \epsilon_2 < \pinf$, define the set $[ \epsilon_1< H < \epsilon_2 ] \eqdef \ba{ x\in \bbR^{m_1} : \epsilon_1 < H(x) < \epsilon_2}$.

\begin{definition}[\KL\!]\label{defn:KLp}
A function $H$ is said to have the Kurdyka-\L ojasiewicz property at $\xbar \in \dom(H)$ if there exists $\epsilon \in (0, \pinf]$, a neighborhood $U$ of $\xbar$ and a continuous concave function $\vphi:[0, \epsilon) \to \bbR_{+}$ such that
\begin{enumerate}[label={\rm (\roman{*})}]
\item
$\vphi(0) = 0$, $\vphi$ is $C^1$ on $(0, \epsilon)$, and for all $r \in (0, \epsilon)$, $\vphi'(r) > 0$;
\item
for all $x \in U\cap [H(\xbar)<H<H(\xbar) + \epsilon]$, the \KL inequality holds:
\begin{equation}\label{eq:KLi}
\vphi'\Pa{H(x)-H(\xbar)}\dist\Pa{0, \partial H(x)} \geq 1 .
\eeq
\end{enumerate}
Proper  functions which satisfy KL property at each point of $\dom(\partial H)$ are called KL functions.

\end{definition}

Roughly speaking, KL functions become sharp up to reparameterization via $\varphi$, a \emph{desingularizing function} for $H$. Typical KL functions include the class of semi-algebraic functions \cite{Bolte07,bolte2010characterizations}. For instance, the $\ell_0$ pseudo-norm and the rank function are KL. Semi-algebraic functions admit desingularizing functions of the form $\varphi(r) = a r^{1-\theta}$ for $a>0$, and $\theta \in [0,1)$ is known as the \textit{KL exponent} of the function \cite{Bolte07,bolte2014proximal}. For these functions, the KL inequality reads
\begin{equation}
    \label{eq:kl}
    \big( H(x) - H(\overline{x}) \big)^\theta \le C \|\zeta\| \qquad \forall \zeta \in \partial H(x),
\end{equation}
for some $C > 0$. In the case $H(x) = H(\overline{x})$, we use the convention $0^0 \eqdef 0$.

\paragraph{Notation}
We denote $\ba{x_k,y_k}_{k\in\bbN}$ the sequence generated by SPRING. Denote $L_x \eqdef \max_{k \in \mathbb{N}} L_1$ $(y_k)$, and define $L_y$ analogously. We set $\bar{L} \eqdef \max\{ L_x, L_y \}$, $\overline{\gamma}_k \eqdef \max\{ \gmaxk, \gmayk\}$, $\underline{\gamma}_k \eqdef \min\{ \gmaxk, $ $\gmayk \}$, and $\underline{\Phi} \eqdef \inf_{(x,y) \in \dom (\Phi)} \Phi(x,y)$. We also use $L$ to denote the maximum of $L_x$, $L_y$, and $M$ over the iterates generated by SPRING, so that $\bar{L}, M \le L$. We use $\E$ to denote the expectation conditional on the first $k$ iterations of SPRING.\footnote{Specifically, $\E \equiv \EE[\cdot | \mathcal{F}_k]$ where $\mathcal{F}_k$ is the $\sigma$-algebra generated by $B_0, \cdots, B_{k-1}$.}

\subsection{Elementary Lemmas}

The following lemmas generalize the sufficient decrease property of proximal gradient descent to the stochastic-gradient setting. They allow us to show that, if the MSE of the stochastic gradient estimator is small enough, then iteratively applying the proximal gradient operator decreases the suboptimality of each iterate in expectation.

\begin{lemma}
\label{lem:noncon}
    Let $F : \R^m \to \R$ be a function with $L$-Lipschitz continuous gradient, $R : \R^m \to \R$ a proper lower semicontinuous function that is bounded from below, and $z \in \prox_{\eta R} ( x - \eta d )$ for some $\eta > 0$ and $d \in \R^m$. Then
    \begin{equation}
        \textstyle 0 \le F(y) + R(y) - F(z) - R(z) + \langle \nabla F(x) - d, z - y \rangle + ( \frac{L}{2} - \frac{1}{2 \eta} ) \|x - z\|^2 + ( \frac{L}{2} + \frac{1}{2 \eta} ) \|x - y\|^2.
    \end{equation}
\end{lemma}

\begin{proof}
    By the Lipschitz continuity of $\nabla F$, we have the inequalities
    \begin{equation*}
    \begin{aligned}
        F(x) - F(y) &\le \langle \nabla F(x), x - y \rangle + \tfrac{L}{2} \|x - y\|^2, \\
        F(z) - F(x) &\le \langle \nabla F(x), z - x \rangle + \tfrac{L}{2} \|z - x\|^2.
    \end{aligned}
    \end{equation*}
    Furthermore, by the definition of $z$,
    \begin{equation*}
        z \in \Argmin_{v \in \R^m} \big\{ \langle d, v - x \rangle + \tfrac{1}{2 \eta} \|v - x\|^2 + R (v) \big\}.
    \end{equation*}
    Taking $v = y$, we obtain
    \begin{equation*}
        0 \le \textstyle R(y) - R(z) + \langle d, y - z \rangle + \frac{1}{2 \eta} \big( \|x - y\|^2 - \|x - z\|^2 \big).
    \end{equation*}
    Adding these three inequalities completes the proof.
\end{proof}

If the full gradient estimator is used, Lemma \ref{lem:noncon} implies the well-known sufficient decrease property of proximal gradient descent. Using a gradient estimator, this decrease is offset by the estimator's MSE. The following lemma quantifies this relationship.

\begin{lemma}[Sufficient Decrease Property]
\label{lem:descent}
     Let $F,R,$ and $z$ be defined as in Lemma \ref{lem:noncon}. The following inequality holds for any $\lambda > 0$:
    \begin{equation}
        0 \le F(x) + R(x) - F(z) - R(z) + \tfrac{1}{2 L \lambda} \|d - \nabla F(x)\|^2 + \big(\tfrac{L (\lambda + 1)}{2} - \tfrac{1}{2 \eta} \big) \|x - z\|^2 .
    \end{equation}
\end{lemma}

\begin{proof}
    From Lemma \ref{lem:noncon} with $x = y$, we have
    \begin{equation*}
        \textstyle 0 \le F(x) + R(x) - F(z) - R(z) + \langle \nabla F(x) - d, z - x \rangle + ( \frac{L}{2} - \frac{1}{2 \eta} ) \|x - z\|^2.
    \end{equation*}
    Using Young's inequality
    \begin{equation*}
        \langle \nabla F(x) - d, z - x \rangle
        \le \tfrac{1}{2 L \lambda} \|d - \nabla F(x)\|^2 + \tfrac{L \lambda}{2} \|x - z\|^2,
    \end{equation*}
    we obtain the desired result.
\end{proof}

As in a related work \cite{davisfull}, we use the \textit{supermartingale convergence theorem} to obtain almost sure convergence of sequences generated by SPRING. Below, we present a version of this result adapted to our context. We refer to \cite[Theorem 4.2]{davisfull} and \cite[Theorem 1]{robbins1971} for more general presentations.

\begin{lemma}[Supermartingale Convergence]
    Let $\E$ denote the expectation conditioned on the first $k$ iterations of SPRING. Let $\{X_k\}_{k=0}^\infty$ and $\{Y_k\}_{k=0}^\infty$ be sequences of bounded non-negative random variables such that $X_k$ and $Y_k$ depend only on the first $k$ iterations of SPRING. If
    \begin{equation}
        \E X_{k+1} + Y_k \le X_k,
    \end{equation}
    then $\sum_{k=0}^\infty Y_k < \infty$ a.s. and $X_k$ converges a.s.
\end{lemma}

\section{Convergence rates of the generalized gradient map}
\label{sec:conv}

To begin, we present our analysis of the convergence rate of the generalized gradient map defined in \eqref{eq:ggmap}.
Recall that $\bar{L} \eqdef \max\{ L_x, L_y \}$, $\overline{\gamma}_k \eqdef \max\{ \gmaxk, \gmayk\}$, $\underline{\gamma}_k \eqdef \min\{ \gmaxk, $ $\gmayk \}$, and $\underline{\Phi} \eqdef \inf_{(x,y) \in \dom (\Phi)} \Phi(x,y)$.

\begin{theorem}\label{thm:convrate}
For the SPRING algorithm, suppose that assumptions \eqref{item:A-JR} to \eqref{item:A-Flast} hold.
    Let $\tnabla_x$ and $\tnabla_y$ be variance-reduced gradient estimators following Definition \ref{def:varred}.
    \begin{itemize}
    \item Suppose $\overline{\gamma}_k$ is non-increasing, and for all $k$,
    \begin{equation*}
    \begin{aligned}
        & \overline{\gamma}_k \le \tfrac{1}{16} \sqrt{\tfrac{\bar{L}^2}{(V_1 + V_\Upsilon / \rho)^2} + \tfrac{16}{(V_1 + V_\Upsilon / \rho)}} - \tfrac{\bar{L}}{16 (V_1 + V_\Upsilon / \rho)},~
        0 < \beta \le \underline{\gamma}_k,~
         \gmaxk < \tfrac{1}{4 L_x} ~ \textnormal{and} ~ \gmayk < \tfrac{1}{4 L_y}.
    \end{aligned}
    \end{equation*}
    With $\alpha$ chosen uniformly at random from $\{ 0, 1, \cdots, T - 1 \}$,
    \begin{equation*}
         \mathbb{E} [ \dist(0,\mathcal{G}_{\frac{\gamma_{x,\alpha}}{2},\frac{\gamma_{y,\alpha}}{2}}(z_\alpha))^2 ]
         \le \tfrac{4 (\Phi(x_0, y_0) + \frac{2 \overline{\gamma}_0}{\rho} \Upsilon_0)}{T \nu \beta^2}.
    \end{equation*}

    \item
    If, moreover, $\Phi$ satisfies the following error bound
    \begin{equation}
\label{eq:errb}
    \Phi(x,y) - \underline{\Phi} \le \mu \dist\big(0,\mathcal{G}_{\frac{\gamma_{x,k}}{2},\frac{\gamma_{y,k}}{2}}(x,y)\big)^2,
\end{equation}
and $\overline{\gamma}_k$ is such that
    \begin{equation*}
    \begin{aligned}
        \overline{\gamma}_k &\le \tfrac{1}{20} \sqrt{\tfrac{\bar{L}^2}{(V_1 + V_\Upsilon / \rho)^2} + \tfrac{20}{(V_1 + V_\Upsilon / \rho)}} - \tfrac{\bar{L}}{20 (V_1 + V_\Upsilon / \rho)},
    \end{aligned}
    \end{equation*}
    then after $T$ iterations of Algorithm \ref{alg:spring},
    \begin{equation*}
    \mathbb{E} [ \Phi(x_T,y_T) - \underline{\Phi} ]
    \le (1 - \Theta)^T ( \Phi(x_0,y_0) - \underline{\Phi} + \tfrac{4 \overline{\gamma}_0}{\rho} \Upsilon_0),
    \end{equation*}
    where $\Theta \eqdef \min\{\mu \nu \beta^2 / 4, \rho / 2 \}$ and $\nu \eqdef \min\{ \tfrac{1}{4 \gamma_{x,0}} - L_x , \tfrac{1}{4 \gamma_{y,0}} - L_y \}$.
    \end{itemize}
\end{theorem}

\begin{remark}
The error bound condition is closely related to the \KL inequality and is also investigated in related works \cite{aravkin2019trimmed}. These two results generalize many existing convergence guarantees for stochastic gradient methods on non-convex, non-smooth objectives \cite{aravkin2019trimmed,spider,reddi,spiderboost,spiderm}.
\end{remark}

\begin{proof}[Proof of Theorem \ref{thm:convrate}, Part 1]
Let $\hat{x}_{k+1} \in \prox_{\tfrac{\gamma_{x,k}}{2} J} \pa{ x_k - \tfrac{\gamma_{x,k}}{2} \nablax F(x_k,y_k) }$, and let $\hat{y}_{k+1} \in \prox_{\tfrac{\gamma_{y,k}}{2} R} ( y_k - \tfrac{\gamma_{y,k}}{2} $ $\nablay F(x_{k+1},y_k) )$.
Applying Lemma \ref{lem:noncon} with $z = \hat{x}_{k+1}$, $y = x = x_k$ and $d = \nablax F(x_k,y_k)$, we have
\begin{equation*}
    \textstyle F(\hat{x}_{k+1}, y_k) + J(\hat{x}_{k+1}) \le F(x_k,y_k) + J(x_k) + ( \frac{L_x}{2} - \frac{1}{\gamma_{x,k}} ) \|\hat{x}_{k+1} - x_k\|^2.
\end{equation*}
Again, applying Lemma \ref{lem:noncon} with $z = x_{k+1}$, $y = \hat{x}_{k+1}$, $x = x_k$, and $d = \tnablax (x_k,y_k)$, we obtain
\begin{equation*}
\begin{aligned}
F(x_{k+1},y_k) + J(x_{k+1}) &\le \textstyle  F(\hat{x}_{k+1},y_k) + J(\hat{x}_{k+1}) + \langle \nablax F(x_k,y_k) - \tnablax (x_k, y_k), x_{k+1} - \hat{x}_{k+1} \rangle \\
&\textstyle \qquad+ ( \frac{L_x}{2} - \frac{1}{2 \gamma_{x,k}} ) \|x_{k+1} - x_k\|^2 + (\frac{L_x}{2} + \frac{1}{2 \gamma_{x,k}} ) \|\hat{x}_{k+1} - x_k\|^2.
\end{aligned}
\end{equation*}
Adding these two inequalities gives
\begin{equation}\label{eq:x}
\begin{aligned}
& F(x_{k+1},y_k) + J(x_{k+1}) \\
&\le F(x_k, y_k) + J(x_k) + \pa{ L_x - \tfrac{1}{2 \gamma_{x,k}} } \|\hat{x}_{k+1} - x_k\|^2 + ( \tfrac{L_x}{2} - \tfrac{1}{2 \gamma_{x,k}} ) \|x_{k+1} - x_k\|^2 \\
&\qquad + \langle \nablax F(x_k, y_k) - \tnablax (x_k, y_k), x_{k+1} - \hat{x}_{k+1} \rangle \\
&\symnum{1}{\le} F(x_k, y_k) + J(x_k) + \pa{ L_x - \tfrac{1}{2 \gamma_{x,k}} } \|\hat{x}_{k+1} - x_k\|^2 + \pa{ \tfrac{L_x}{2} - \tfrac{1}{2 \gamma_{x,k}} } \|x_{k+1} - x_k\|^2 \\
& \qquad + 2 \gamma_{x,k} \| \nablax F(x_k,y_k) - \tnablax (x_k, y_k) \|^2 + \tfrac{1}{8 \gamma_{x,k}} \|\hat{x}_{k+1} - x_{k+1}\|^2 \\
&\symnum{2}{\le} F(x_k, y_k) + J(x_k) + \pa{ L_x - \tfrac{1}{4 \gamma_{x,k}} } \|\hat{x}_{k+1} - x_k\|^2 + \pa{ \tfrac{L_x}{2} - \tfrac{1}{4 \gamma_{x,k}} } \|x_{k+1} - x_k\|^2 \\
& \qquad + 2 \gamma_{x,k} \| \nablax F(x_k,y_k) - \tnablax (x_k, y_k) \|^2.
\end{aligned}
\eeq
Inequality \numcirc{1} is Young's, and \numcirc{2} is the standard inequality $\|a - c\|^2 \le 2 \|a - b\|^2 + 2 \|b - c\|^2$. Performing the same procedure for the updates in $y_k$ gives
\begin{equation}\label{eq:y}
\begin{aligned}
& F(x_{k+1},y_{k+1}) + R(y_{k+1}) \\
&\le F(x_{k+1}, y_k) + R(y_k) + \pa{ L_x - \tfrac{1}{4 \gamma_{y,k}} } \|\hat{y}_{k+1} - y_k\|^2 + \pa{ \tfrac{L_y}{2} - \tfrac{1}{4 \gamma_{y,k}} } \|y_{k+1} - y_k\|^2 \\
& \qquad + 2 \gamma_{y,k} \| \nablay F(x_{k+1},y_k) - \tnablay (x_{k+1}, y_k) \|^2.
\end{aligned}
\eeq
Adding inequality \eqref{eq:x} and inequality \eqref{eq:y}, we have
\begin{equation}\label{eq:z}
\begin{aligned}
\Phi(x_{k+1},y_{k+1}) & \le \Phi(x_k,y_k) + \pa{ L_x - \tfrac{1}{4 \gamma_{x,k}} } \|\hat{x}_{k+1} - x_k\|^2 + \pa{ L_y - \tfrac{1}{4 \gamma_{y,k}} } \|\hat{y}_{k+1} - y_k\|^2 \\
&\qquad + \pa{ \tfrac{L_x}{2} - \tfrac{1}{4 \gamma_{x,k}} } \|x_{k+1} - x_k\|^2 + \pa{ \tfrac{L_y}{2} - \tfrac{1}{4 \gamma_{y,k}} } \|y_{k+1} - y_k\|^2 \\
&\qquad + 2 \overline{\gamma}_k \big( \| \nablax F(x_k,y_k) - \tnablax (x_k, y_k) \|^2 + \| \nablay F(x_{k+1},y_k) - \tnablay (x_{k+1}, y_k) \|^2 \big),
\end{aligned}
\eeq
where $\overline{\gamma}_k = \max\{\gmaxk,\gmayk\}$. We apply the conditional expectation operator $\E$ and bound the MSE terms using \eqref{eq:varreduc}. This gives
\begin{equation}\label{eq:linbegin}
\begin{aligned}
& \E [\Phi(x_{k+1},y_{k+1}) + \pa{ - \tfrac{L_x}{2} - 2 V_1 \overline{\gamma}_k + \tfrac{1}{4 \gamma_{x,k}} } \|x_{k+1} - x_k\|^2 \\
& \qquad + \pa{ -\tfrac{L_y}{2} - 2 V_1 \overline{\gamma}_k + \tfrac{1}{4 \gamma_{y,k}} } \|y_{k+1} - y_k\|^2 ] \\
&\le \Phi(x_k,y_k) + \pa{ L_x - \tfrac{1}{4 \gamma_{x,k}} } \|\hat{x}_{k+1} - x_k\|^2 + \pa{ L_y - \tfrac{1}{4 \gamma_{y,k}} } \|\hat{y}_{k+1} - y_k\|^2 + 2 \overline{\gamma}_k \Upsilon_k \\
& \qquad + 2 V_1 \overline{\gamma}_k \|z_k - z_{k-1}\|^2.
\end{aligned}
\eeq
Next, we use \eqref{eq:geo} to say
\begin{equation*}
    2 \overline{\gamma}_k \Upsilon_k \le \tfrac{2 \overline{\gamma}_k}{\rho} \big(-\E \Upsilon_{k+1} + \Upsilon_k + V_{\Upsilon} (\E \|z_{k+1} - z_k\|^2 + \|z_k - z_{k-1}\|^2) \big).
\end{equation*}
Adding the previous two inequalities, we have
\begin{equation*}
\begin{aligned}
& \E [\Phi(x_{k+1},y_{k+1}) + \pa{ -\tfrac{L_x}{2} - 2 V_1 \overline{\gamma}_k - \tfrac{2 V_\Upsilon \overline{\gamma}_k}{\rho} + \tfrac{1}{4 \gamma_{x,k}} } \|x_{k+1} - x_k\|^2 \\
& \qquad + \pa{ -\tfrac{L_y}{2} - 2 V_1 \overline{\gamma}_k - \tfrac{2 V_\Upsilon \overline{\gamma}_k}{\rho} + \tfrac{1}{4 \gamma_{y,k}} } \|y_{k+1} - y_k\|^2 + \tfrac{2 \overline{\gamma}_k}{\rho} \Upsilon_{k+1}] \\
&\le \Phi(x_k,y_k) + \pa{ L_x - \tfrac{1}{4 \gamma_{x,k}} } \|\hat{x}_{k+1} - x_k\|^2 + \pa{ L_y - \tfrac{1}{4 \gamma_{y,k}} } \|\hat{y}_{k+1} - y_k\|^2 + \tfrac{2 \overline{\gamma}_k}{\rho} \Upsilon_k \\
&\qquad + 2 \overline{\gamma}_k (V_1 + \tfrac{V_\Upsilon}{\rho}) \|z_k - z_{k-1}\|^2.
\end{aligned}
\eeqn
Let $\bar{L} = \max\{L_x,L_y\}$. To ensure that the coefficients of $\|x_{k+1} - x_k\|^2$ and $\|y_{k+1} - y_k\|^2$ are non-negative, we set
\begin{equation*}
    \overline{\gamma}_k \le \tfrac{1}{16} \sqrt{\tfrac{\bar{L}^2}{(V_1 + V_\Upsilon / \rho)^2} + \tfrac{16}{(V_1 + V_\Upsilon / \rho)}} - \tfrac{\bar{L}}{16 (V_1 + V_\Upsilon / \rho)},
\eeqn
for all $k \ge 0$. To ensure that the coefficients of $\|\hat{x}_{k+1} - x_k\|^2$ and $\|\hat{y}_{k+1} - y_k\|^2$ are non-positive, we set $\gmaxk < \tfrac{1}{4 L_x}$ and $\gmayk < \tfrac{1}{4 L_y}$, which yields
\begin{equation*}
\begin{aligned}
& \E [\Phi(x_{k+1},y_{k+1}) + 2 \overline{\gamma}_k (V_1 + V_\Upsilon / \rho) \|z_{k+1} - z_k\|^2 + \tfrac{2 \overline{\gamma}_k }{\rho} \Upsilon_{k+1} ] \\
&\le \Phi(x_k,y_k) + \pa{ L_x - \tfrac{1}{4 \gamma_{x,k}} } \|\hat{x}_{k+1} - x_k\|^2 + \pa{ L_y - \tfrac{1}{4 \gamma_{y,k}} } \|\hat{y}_{k+1} - y_k\|^2 \\
& \qquad + 2 \overline{\gamma}_k (V_1 + V_\Upsilon / \rho) \|z_k - z_{k-1}\|^2 + \tfrac{2 \overline{\gamma}_k }{\rho} \Upsilon_k.
\end{aligned}
\eeqn
Because $\overline{\gamma}_k$ is non-increasing,
\begin{equation*}
\begin{aligned}
& \E [\Phi(x_{k+1},y_{k+1}) + 2 \overline{\gamma}_{k+1} (V_1 + V_\Upsilon / \rho) \|z_{k+1} - z_k\|^2 + \tfrac{2 \overline{\gamma}_{k+1}}{\rho} \Upsilon_{k+1} ] \\
&\le \Phi(x_k,y_k) - \nu \|\hat{z}_{k+1} - z_k\|^2 + 2 \overline{\gamma}_k (V_1 + V_\Upsilon / \rho) \|z_k - z_{k-1}\|^2 + \tfrac{2 \overline{\gamma}_k }{\rho} \Upsilon_k,
\end{aligned}
\eeqn
where $\nu = \min\{ \tfrac{1}{4 \gamma_{x,0}} - L_x , \tfrac{1}{4 \gamma_{y,0}} - L_y \}$ Applying the full expectation operator and summing from $k = 0$ to $k = T-1$ gives
\begin{equation*}
    \tfrac{2 \overline{\gamma}_T}{\rho} \Upsilon_T + 2\overline{\gamma}_T (V_1 + V_\Upsilon / \rho ) \|z_T - z_{T-1}\|^2 + \nu \sum\nolimits_{k=0}^{T-1} \mathbb{E} \|\hat{z}_{k+1} - z_k\|^2 \le \Phi(x_0, y_0) + \tfrac{2 \overline{\gamma}_0 }{\rho} \Upsilon_0.
\eeqn
We drop the first two terms on the left from the inequality as they are non-negative. Let $\alpha$ be drawn uniformly at random from the set $\{0,1,\cdots,T-1\}$, and recall $\underline{\gamma}_k \ge \beta$. Using the fact that $\|\hat{z}_{k+1} - z_k\|^2 \ge \tfrac{\beta^2}{4} \dist(0,\mathcal{G}_{\tfrac{\gamma_{x,k}}{2}, \tfrac{\gamma_{y,k}}{2}} (z_k))^2$,
\begin{equation*}
    \mathbb{E} \dist(0,\mathcal{G}_{\tfrac{\gamma_{x,\alpha}}{2}, \tfrac{\gamma_{y,\alpha}}{2}}(z_\alpha))^2 \le \tfrac{4 (\Phi(x_0, y_0) + \frac{2 \overline{\gamma}_0}{\rho} \Upsilon_0)}{T \nu \beta^2},
\eeqn
which completes the proof of the first claim.
\end{proof}

Combining the same argument with the error bound \eqref{eq:errb}, we obtain a linear convergence rate to the global optimum.

\begin{proof}[Proof of Theorem \ref{thm:convrate}, Part 2]
    We begin with equation \eqref{eq:linbegin}:
    \begin{equation*}
        \begin{aligned}
        & \E [\Phi(x_{k+1},y_{k+1}) + \pa{ -\tfrac{L_x}{2} - 2 V_1 \gamma_{x,k} + \tfrac{1}{4 \gamma_{x,k}} } \|x_{k+1} - x_k\|^2  + \pa{ -\tfrac{L_y}{2} - 2 V_1 \gamma_{y,k} + \tfrac{1}{4 \gamma_{y,k}} } \|y_{k+1} - y_k\|^2 ] \\
        &\le \Phi(x_k,y_k) - \nu \|\hat{z}_{k+1} - z_k\|^2 + 2 \overline{\gamma}_k \Upsilon_k + 2 V_1 \overline{\gamma}_k \|z_k - z_{k-1}\|^2.
        \end{aligned}
    \eeqn
    Using \eqref{eq:geo}, we can say for any $c > 0$,
    \begin{equation*}
        0 \le \tfrac{2 c \overline{\gamma}_k}{\rho} \big(- \E \Upsilon_{k+1} + (1 - \rho) \Upsilon_k + V_{\Upsilon} (\|z_{k+1} - z_k\|^2 + \|z_k - z_{k-1}\|^2) \big).
    \end{equation*}
    Adding the previous two inequalities, we have
    \begin{equation*}
    \begin{aligned}
    & \E [\Phi(x_{k+1},y_{k+1}) + \pa{ -\tfrac{L_x}{2} - 2 V_1 \gamma_{x,k} - \tfrac{2 c V_\Upsilon \overline{\gamma}_k}{\rho} + \tfrac{1}{4 \gamma_{x,k}} } \|x_{k+1} - x_k\|^2 \\
    & \qquad + \pa{ -\tfrac{L_y}{2} - 2 V_1 \gamma_{y,k} - \tfrac{2 c V_\Upsilon \overline{\gamma}_k}{\rho} + \tfrac{1}{4 \gamma_{y,k}} } \|y_{k+1} - y_k\|^2 + \tfrac{2 c \overline{\gamma}_k}{\rho} \Upsilon_{k+1} ] \\
    &\le \Phi(x_k,y_k) - \nu \|\hat{z}_{k+1} - z_k\|^2 + 2 \overline{\gamma}_k (V_1 + \tfrac{c V_\Upsilon}{\rho} \|z_k - z_{k-1}\|^2 + \tfrac{2 c \overline{\gamma}_k}{\rho} (1 + \tfrac{\rho}{c} - \rho) \Upsilon_k.
    \end{aligned}
    \eeqn
    Because $\gmaxk < \tfrac{1}{4 L_x}$ and $\gmayk < \tfrac{1}{4 L_y}$, we can apply the error bound assumption \eqref{eq:errb} to say
    \begin{equation*}
        - \nu \|\hat{z}_{k+1} - z_k\|^2 \le - \tfrac{\nu \underline{\gamma}_k^2}{4} \dist(0,\mathcal{G}_{\tfrac{\gamma_{x,k}}{2},\tfrac{\gamma_{y,k}}{2}}(z_k))^2 \le - \tfrac{\mu \nu \underline{\gamma}_k^2}{4} ( \Phi(x_k, y_k) - \underline{\Phi} ).
    \end{equation*}
    In total, we have
    \begin{equation*}
    \begin{aligned}
    & \E [\Phi(x_{k+1},y_{k+1}) - \underline{\Phi} + \pa{ -\tfrac{L_x}{2} - 2 V_1 \gamma_{x,k} - \tfrac{2 c V_\Upsilon \overline{\gamma}_k}{\rho} + \tfrac{1}{4 \gamma_{x,k}} } \|x_{k+1} - x_k\|^2 \\
    & \qquad + \pa{ -\tfrac{L_y}{2} - 2 V_1 \gamma_{y,k} - \tfrac{2 c V_\Upsilon \overline{\gamma}_k}{\rho} + \tfrac{1}{4 \gamma_{y,k}} } \|y_{k+1} - y_k\|^2 + \tfrac{2 c \overline{\gamma}_k}{\rho} \Upsilon_{k+1} ] \\
    &\le ( 1 - \tfrac{\mu \nu \underline{\gamma}_k^2}{4} ) ( \Phi(x_k,y_k) - \underline{\Phi} ) + 2 \overline{\gamma}_k (V_1 + \tfrac{c V_\Upsilon}{\rho} \|z_k - z_{k-1}\|^2 + \tfrac{2 c \overline{\gamma}_k}{\rho} (1 + \tfrac{\rho}{c} - \rho) \Upsilon_k.
    \end{aligned}
    \eeqn
    Choosing $c = 2$, setting the step-sizes so that they satisfy, for all $k$,
    \begin{equation*}
        \overline{\gamma}_k \le \tfrac{1}{20} \sqrt{\tfrac{\bar{L}^2}{(V_1 + 2 V_\Upsilon / \rho)^2} + \tfrac{20}{(V_1 + 2 V_\Upsilon / \rho)}} - \tfrac{\bar{L}}{20 (V_1 + 2 V_\Upsilon / \rho)}, \ \gmaxk < \tfrac{1}{4 L_x}, \ \gmayk < \tfrac{1}{4 L_y}, \ 0 < \beta \le \underline{\gamma}_k,
    \eeqn
    and letting $\Theta = \min\{\mu \nu \beta^2 / 4, \rho / 2 \}$, we have
    \begin{equation*}
    \begin{aligned}
    & \E [ \Phi(x_{k+1},y_{k+1}) - \underline{\Phi} + 2 \overline{\gamma}_k (V_1 + \tfrac{2 V_\Upsilon}{\rho} \|z_{k+1} - z_k\|^2 + \tfrac{4 \overline{\gamma}_k}{\rho} \Upsilon_{k+1} ] \\
    &\le (1 - \Theta) ( \Phi(x_k,y_k) - \underline{\Phi} + 2 \overline{\gamma}_k (V_1 + \tfrac{2 V_\Upsilon}{\rho} \|z_k - z_{k-1}\|^2 + \tfrac{4 \overline{\gamma}_k}{\rho} \Upsilon_k ) .
    \end{aligned}
    \eeqn
    Because $\overline{\gamma}_k$ is non-increasing,
    \begin{equation*}
    \begin{aligned}
    & \E [ \Phi(x_{k+1},y_{k+1}) - \underline{\Phi} + 2 \overline{\gamma}_{k+1} (V_1 + \tfrac{2 V_\Upsilon}{\rho} \|z_{k+1} - z_k\|^2 + \tfrac{4 \overline{\gamma}_{k+1}}{\rho} \Upsilon_{k+1} ] \\
    &\le (1 - \Theta) ( \Phi(x_k,y_k) - \underline{\Phi} + 2 \overline{\gamma}_k (V_1 + \tfrac{2 V_\Upsilon}{\rho} \|z_k - z_{k-1}\|^2 + \tfrac{4 \overline{\gamma}_k}{\rho} \Upsilon_k ) .
    \end{aligned}
    \eeqn
    Applying the full expectation operator and chaining this inequality over the iterations $k = 0$ to $k = T-1$, 
    \begin{equation*}
    \mathbb{E} [ \Phi(x_T,y_T) - \underline{\Phi} ] \le (1 - \Theta)^T \big( \Phi(x_0,y_0) - \underline{\Phi} + \tfrac{4 \overline{\gamma}_0}{\rho} \Upsilon_0 \big),
    \eeqn
    which completes the proof.
\end{proof}

Because SAGA and SARAH gradient estimators are variance-reduced, Theorem \ref{thm:convrate} implies specific convergence rates for Algorithm \ref{alg:spring} when using these estimators.

\begin{corollary}
\label{cor:spec}
    To compute an $\epsilon$-approximate critical point in expectation, Algorithm \ref{alg:spring} using
    \begin{itemize}
        \item SARAH gradient estimator with $p = n$, $\overline{\gamma}_k \le \tfrac{1}{2 L \sqrt{30 n}}$ and any batch size requires no more than $\mathcal{O}\big( L \sqrt{n} / \epsilon^2 \big)$ SFO calls;
        \item SAGA gradient estimator with $b = n^{2/3}$ and $\overline{\gamma}_k \le \frac{1}{2 \sqrt{2710} L}$ requires no more than $\mathcal{O}(L n^{2/3} / \epsilon^2)$ SFO calls.\footnote{For ease of exposition, we do not optimize over constants, so these step-sizes (particularly for SAGA estimator) are not optimal. In general, we find the step-sizes suggested by theory to be conservative in practice (see Section \ref{sec:ex} for details regarding practical step-sizes).}
    \end{itemize}
    If $\Phi$ satisfies the error bound condition \eqref{eq:errb}, then to compute an $\epsilon$-suboptimal point in expectation, Algorithm \ref{alg:spring} using
    \begin{itemize}
        \item the SARAH gradient estimator requires no more than $\mathcal{O}( (n + L \sqrt{n}/\mu ) \log \big(1/\epsilon\big))$ SFO calls;
        \item the SAGA gradient estimator requires no more than $\mathcal{O}( (n + L n^{2/3}/\mu) \log(1/\epsilon))$ SFO calls.
    \end{itemize}
\end{corollary}

\begin{remark}
The improved dependence on $n$ when using SARAH gradient estimator exists in all of our convergence rates for SPRING. Because most existing works on stochastic optimization for non-smooth, non-convex problems use models that are special cases of \eqref{eq:J-F-R}, our results for SPRING capture most existing work as special cases. In particular, in the case $R \equiv J \equiv 0$, our results recover recent results showing that SARAH achieves the \emph{oracle complexity lower-bound} for non-convex problems with a finite-sum structure \cite{spider,proxsarah,spiderboost,zhou2019lower,spiderm}.
\end{remark}

\section{Convergence Rate under the KL Property}
\label{sec:kl}

The results from previous section require only assumptions \eqref{item:A-JR} to \eqref{item:A-Flast}. To prove convergence of the sequence of the algorithm, and to obtain convergence rates depending on the KL exponent of the objective, two extra conditions are required. In this section, under the assumption that the objective function $\Phi$ is KL and the sequence generated by SPRING is bounded, we prove convergence of the sequence and extend the convergence rates of PALM to SPRING.
To derive these results, we need some preparatory results which generalize the claims of PALM \cite{bolte2014proximal} to stochastic setting. Define the quantity
\begin{equation} \label{eq:psi}
    \Psi_k \defeq \Phi(z_k) + \tfrac{1}{2 \rho \sqrt{2 (V_1 + V_\Upsilon / \rho)}} \Upsilon_k + \tfrac{\sqrt{V_1 + V_\Upsilon / \rho}}{\sqrt{2}} \|z_k - z_{k-1}\|^2  .
\end{equation}
Our first result guarantees that $\Psi_k$ is decreasing in expectation.

\begin{lemma}[$\ell_2$ summability]
\label{lem:seqdec}
    Let $\{ z_k \}_{k=0}^\infty$ be the sequence generated by SPRING with $\overline{\gamma}_k$ non-increasing and satisfying $\overline{\gamma}_k < \tfrac{\sqrt{2}}{5 (\sqrt{V_1 + V_\Upsilon / \rho} + \bar{L})},\, \forall k$,
    then $\Psi_{k} $ satisfies
    \begin{equation}
    \label{eq:psidec}
        \E \Psi_{k+1}
        \le \Psi_k + \big( \tfrac{\bar{L}}{2} + \tfrac{3}{2} \sqrt{ 2 (V_1 + V_\Upsilon / \rho) } - \tfrac{1}{2 \overline{\gamma}_k} \big) \E \|z_{k+1} - z_k\|^2 - \tfrac{\sqrt{V_1 + V_\Upsilon / \rho}}{2 \sqrt{2}} \|z_k - z_{k-1}\|^2,
    \end{equation}
    and the expectation of the squared distance between the iterates is summable:
    \begin{equation*}
        \sum\nolimits_{k=0}^\infty \mathbb{E} \left[ \|x_{k+1} - x_k\|^2 + \|y_{k+1} - y_k\|^2 \right] = \sum\nolimits_{k=0}^\infty \mathbb{E} \|z_{k+1} - z_k\|^2 < \infty.
    \end{equation*}
\end{lemma}

\begin{proof}
    Applying Lemma \ref{lem:descent} twice, once for the update in $x_k$ and once for the update in $y_k$, we have
    \begin{equation*}
    \begin{aligned}
        F(x_{k+1},y_k) + J(x_{k+1}) &\le F(x_k,y_k) + J(x_k) + \tfrac{1}{2 \bar{L} \lambda} \|\tnablax (x_k,y_k) - \nablax F(x_k,y_k) \|^2 \\
        &\qquad + \big( \tfrac{\bar{L} (\lambda + 1)}{2} - \tfrac{1}{2 \gamma_{x,k}} \big) \|x_{k+1} - x_k\|^2,
    \end{aligned}
    \end{equation*}
    as well as
    \begin{equation*}
    \begin{aligned}
        F(x_{k+1},y_{k+1}) + R(y_{k+1}) &\le F(x_{k+1},y_k) + R(y_k) + \big( \tfrac{\bar{L} (\lambda + 1)}{2} - \tfrac{1}{2 \gamma_{y,k}} \big) \|y_{k+1} - y_k\|^2  \\& \qquad+ \tfrac{1}{2 \bar{L} \lambda} \|\tnablay (x_{k+1},y_k) - \nablay F(x_{k+1},y_k) \|^2  .
    \end{aligned}
    \end{equation*}
    Adding these inequalities together,
    \begin{equation*}
    \begin{aligned}
        \Phi (x_{k+1}, y_{k+1})
        &\le \Phi(x_k,y_k) + \tfrac{1}{2 \bar{L} \lambda} \|\tnablax (x_k,y_k) - \nablax F(x_k,y_k) \|^2\\
        &\qquad + \tfrac{1}{2 \bar{L} \lambda} \|\tnablay (x_{k+1},y_k) - \nablay F(x_{k+1},y_k) \|^2  + \big( \tfrac{\bar{L} (\lambda + 1)}{2} - \tfrac{1}{2 \overline{\gamma}_k} \big) \|z_{k+1} - z_k\|^2.
    \end{aligned}
    \end{equation*}
    Applying the conditional expectation operator $\E$, we can bound the MSE terms using \eqref{eq:varreduc}. This gives
    \begin{equation}
    \label{eq:lyap}
         \E \big[ \Phi (z_{k+1}) + \big( - \tfrac{\bar{L} (\lambda + 1)}{2} - \tfrac{V_1}{2 \bar{L} \lambda} + \tfrac{1}{2 \overline{\gamma}_k} \big) \|z_{k+1} - z_k\|^2 \big]
         \le \Phi(z_k) + \tfrac{1}{2 \bar{L} \lambda} \Upsilon_k + \tfrac{V_1}{2 \bar{L} \lambda} \|z_k - z_{k-1}\|^2.
    \end{equation}
    Next, we use \eqref{eq:geo} to say
    \begin{equation*}
        \tfrac{1}{2 \bar{L} \lambda} \Upsilon_k \le \tfrac{1}{2 \bar{L} \lambda \rho} \big(-\E \Upsilon_{k+1} + \Upsilon_k + V_{\Upsilon} (\E \|z_{k+1} - z_k\|^2 + \|z_k - z_{k-1}\|^2) \big).
    \end{equation*}
    Combining these inequalities, we have
    \begin{equation*}
    \begin{aligned}
        & \E \Big[ \Phi (z_{k+1}) + \tfrac{1}{2 \bar{L} \lambda \rho} \Upsilon_{k+1} + \big( - \tfrac{\bar{L} (\lambda + 1)}{2} - \tfrac{V_1 + V_\Upsilon / \rho}{2 \bar{L} \lambda} + \tfrac{1}{2 \overline{\gamma}_k} \big) \|z_{k+1} - z_k\|^2 \Big] \\
        &\le \Phi(z_k) + \tfrac{1}{2 \bar{L} \lambda \rho} \Upsilon_k + \tfrac{V_1 + V_\Upsilon / \rho}{2 \bar{L} \lambda} \|z_k - z_{k-1}\|^2.
    \end{aligned}
    \end{equation*}
    This is equivalent to
    \begin{equation*}
    \begin{aligned}
        & \E \Big[ \Phi (z_{k+1}) + \tfrac{1}{2 \bar{L} \lambda \rho} \Upsilon_{k+1} + \big( \tfrac{V_1 + V_\Upsilon / \rho}{2 \bar{L} \lambda} + Z \big) \|z_{k+1} - z_k\|^2 + \big( - \tfrac{\bar{L} (\lambda + 1)}{2} - \tfrac{V_1 + V_\Upsilon / \rho}{\bar{L} \lambda} - Z + \tfrac{1}{2 \overline{\gamma}_k} \big) \|z_{k+1} - z_k\|^2 \Big] \\
        &\le \Phi(z_k) + \tfrac{1}{2 \bar{L} \lambda \rho} \Upsilon_k + \big( \tfrac{V_1 + V_\Upsilon / \rho}{2 \bar{L} \lambda} + Z \big) \|z_k - z_{k-1}\|^2 - Z \|z_k - z_{k-1}\|^2,
    \end{aligned}
    \end{equation*}
    for some constant $Z \ge 0$. Setting $\overline{\gamma}_k \le (2 (\frac{\bar{L} (\lambda + 1)}{2} + \frac{V_1 + V_\Upsilon / \rho}{\bar{L} \lambda} + Z) )^{-1}$ and using the fact that $\overline{\gamma}_k$ is non-increasing, we have
    \begin{equation*}
    \label{eq:tele}
        \E \Psi_{k+1} \le \Psi_k + \big( \tfrac{\bar{L} (\lambda + 1)}{2} + \tfrac{V_1 + V_\Upsilon / \rho}{\bar{L} \lambda} + Z - \tfrac{1}{2 \overline{\gamma}_k} \big) \E \|z_{k+1} - z_k\|^2 - Z \|z_k - z_{k-1}\|^2.
    \end{equation*}
    proving the first claim that $\Psi_k$ is decreasing in expectation. To approximately maximize our bound on $\overline{\gamma}_k$, we set $\lambda = \frac{\sqrt{2 (V_1 + V_\Upsilon / \rho)}}{\bar{L}}$.

    To prove the second claim, we apply the full expectation operator to \eqref{eq:tele} and sum the resulting inequality from $k = 0$ to $k = T-1$,
    \begin{equation*}
        \mathbb{E} \Psi_T \le \Psi_0 + \big( \tfrac{\bar{L} (\lambda + 1)}{2} + \tfrac{V_1 + V_\Upsilon / \rho}{\bar{L} \lambda} + Z - \tfrac{1}{2 \overline{\gamma}_k} \big) \sum\nolimits_{k=0}^{T-1} \mathbb{E} \|z_{k+1} - z_k\|^2  - Z \mathbb{E} \|z_k - z_{k-1}\|^2.
    \end{equation*}
    Rearranging and using the facts that $\underline{\Phi} \le \Psi_T$ and $\overline{\gamma}_k$ is non-increasing,
    \begin{equation}
    \label{eq:2}
    \begin{aligned}
        & \big( \tfrac{1}{2 \overline{\gamma}_k} - \tfrac{\bar{L} (\lambda + 1)}{2} - \tfrac{V_1 + V_\Upsilon / \rho}{\bar{L} \lambda} - Z \big) \sum\nolimits_{k=0}^{T-1} \mathbb{E} \|z_{k+1} - z_k\|^2 + Z \mathbb{E} \|z_k - z_{k-1}\|^2 \le \Psi_0 - \underline{\Phi}.
    \end{aligned}
    \end{equation}
    Taking the limit $T \to + \infty$ proves that the sequence $\EE \|z_{k+1} - z_k\|^2$ is summable.
    Relations \eqref{eq:tele} and \eqref{eq:2} hold for any $Z \ge 0$; we set $Z = \frac{\sqrt{V_1 + V_\Upsilon / \rho}}{2 \sqrt{2}}$ to simplify later arguments.
\end{proof}

The next lemma establishes a bound on the norm of subgradients of $\Phi(z_k)$.

\begin{lemma}[Subgradient Bound]
\label{lem:subbound}
    Let $\{z_k\}_{k\in \mathbb{N}}$ be the sequence generated by SPRING with step-sizes satisfying $0 < \beta \le \underline{\gamma}_k$. Define
    \begin{equation*}
    \begin{aligned}
        A_x^k &\defeq 1 / \gmaxk (x_{k-1} - x_k) + \nablax F(x_k,y_k) - \tnablax (x_{k-1},y_{k-1}) \quad \textrm{and} \\
        A_y^k &\defeq 1 / \gmayk (y_{k-1} - y_k) + \nablay F(x_k,y_k) - \tnablay (x_k,y_{k-1}).
        \end{aligned}
    \end{equation*}
    Then $(A_x^k, A_y^k) \in \partial \Phi(x_k,y_k)$ and, with $p = 1 / \beta + M + L_y + V_2$,
    \begin{equation}
    \label{eq:subbound}
        \mathbb{E}_{k-1} \|(A_x^k, A_y^k)\| \le p (\mathbb{E}_{k-1} \left\| z_k - z_{k-1} \right\| + \|z_{k-1} - z_{k-2}\|) + \Gamma_{k-1}.
    \end{equation}
\end{lemma}

\begin{proof}
    The fact that $(A_x^k, A_y^k) \in \partial \Phi(x_k,y_k)$ is clear from the definition of the proximal operator:
    \begin{equation*}
    \begin{aligned}
        \tfrac{1}{\gmaxk} (x_{k-1} - x_k) - \tnablax (x_{k-1},y_{k-1}) &\in \partial J(x_k), \\
        \tfrac{1}{\gmayk} (y_{k-1} - y_k) - \tnablay (x_k,y_{k-1}) &\in \partial R(y_k).
    \end{aligned}
    \end{equation*}
    Combining this with the fact that $\partial \Phi(x_k,y_k) = (\nablax F(x_k, y_k) + \partial J(x_k), \nablay F(x_k, y_k) + \partial R(y_k))$ makes it clear that $(A_x^k,A_y^k) \in \partial \Phi(x_k,y_k)$. All that remains is to bound the norms of $A_x^k$ and $A_y^k$. Because $\nabla F$ is $M$-Lipschitz continuous on bounded sets,
    \begin{equation}
    \begin{aligned}
    \label{eq:Ax}
        \mathbb{E}_{k-1} \|A_x^k\|
        & \le \tfrac{1}{\gmaxk} \mathbb{E}_{k-1} \|x_{k-1} - x_k\| + \mathbb{E}_{k-1} \| \nablax F(x_k,y_k) - \tnablax (x_{k-1},y_{k-1}) \| \\
        & \le \tfrac{1}{\gmaxk} \mathbb{E}_{k-1} \|x_{k-1} - x_k\| + \mathbb{E}_{k-1} \| \nablax F(x_k,y_k) - \nablax F(x_{k-1}, y_{k-1}) \| \\
        & \qquad + \mathbb{E}_{k-1} \| \nablax F(x_{k-1}, y_{k-1}) - \tnablax (x_{k-1},y_{k-1}) \| \\
        & \le \big( \tfrac{1}{\gmaxk} + M \big) \mathbb{E}_{k-1} \|x_{k-1} - x_k\| + M \mathbb{E}_{k-1} \| y_k - y_{k-1} \| \\
        & \qquad + \mathbb{E}_{k-1} \| \nablax F(x_{k-1}, y_{k-1}) - \tnablax (x_{k-1},y_{k-1}) \|.
    \end{aligned}
    \end{equation}
    A similar argument holds for $\|A_y^k\|$.
    \begin{equation*}
    \begin{aligned}
        \mathbb{E}_{k-1} \|A_y^k\|
        & \le \tfrac{1}{\gmayk} \mathbb{E}_{k-1} \|y_{k-1} - y_k\| + \mathbb{E}_{k-1} \| \nablay F(x_k,y_k) - \tnablay (x_k,y_{k-1}) \| \\
        & \le \tfrac{1}{\gmayk} \mathbb{E}_{k-1} \|y_{k-1} - y_k\| + \mathbb{E}_{k-1} \| \nablay F(x_k,y_k) - \nablay F(x_k, y_{k-1}) \| \\
        & \qquad + \mathbb{E}_{k-1} \| \nablay F(x_k, y_{k-1}) - \tnablay (x_k,y_{k-1}) \| \\
        & \le \big(\tfrac{1}{\gmayk} + L_y \big) \mathbb{E}_{k-1} \|y_{k-1} - y_k\| + \mathbb{E}_{k-1} \| \nablay F(x_k, y_{k-1}) - \tnablay (x_k,y_{k-1}) \|.
    \end{aligned}
    \end{equation*}
    Adding these two inequalities together and using equation \eqref{eq:varreduc} to bound the MSE terms, we get
    \begin{equation*}
        \mathbb{E}_{k-1} \|(A_x^k, A_y^k)\| \le  \mathbb{E}_{k-1} \left[ \|A_x^k\| + \|A_y^k\| \right] \le p ( \mathbb{E}_{k-1} \| z_k - z_{k-1} \| + \|z_{k-1} - z_{k-2}\|) + \Gamma_{k-1},
    \end{equation*}
    where $p = 1 / \beta + M + L_y + V_2$.
\end{proof}

The following lemma describes the limiting behavior of $\{z_k\}_{k=0}^\infty$. The set of limit points of $\{z_k\}_{k=0}^\infty$ is defined as
\begin{equation*}
    \omega(z_0) \defeq \{ z : \exists \textnormal{ an increasing sequence of integers } \{k_\ell\}_{\ell \in \mathbb{N}} \textnormal{ such that } z_{k_\ell} \to z \textnormal{ as } \ell \to +\infty \}.
\end{equation*}

\begin{lemma}[Limit points of $\{z_k\}_{k=0}^\infty$]
\label{lem:4case}
Suppose assumptions \eqref{item:A-JR}-\eqref{item:A-Flast} hold, that the sequence $z_k=(x_k, y_k)$ is \emph{bounded}, and the step-sizes of Algorithm \ref{alg:spring} satisfy the following conditions:
    \begin{equation*}
        \gmaxk, \gmayk \in \big[ \beta, \tfrac{\sqrt{2}}{5 (\sqrt{V_1 + V_\Upsilon / \rho} + \bar{L})} \big) \qquad \forall k,
    \end{equation*}
    and $\overline{\gamma}_k$ is non-increasing. Then
    \begin{enumerate}
        \item $\sum_{k = 1}^\infty \|z_k - z_{k-1}\|^2 < \infty$ a.s., and $\|z_k - z_{k-1}\| \to 0$ a.s.;
        \item $\EE \Phi(z_k) \to \Phi^\star$, where $\Phi^\star \in [\underline{\Phi}, \infty)$;
        \item $\EE \dist(0,\partial \Phi(z_k)) \to 0$;
        \item The set $\omega(z_0)$ is non-empty, and for all $z^\star \in \omega(z_0), \ \EE \dist(0,\partial \Phi(z^\star)) = 0$;
        \item $\dist(z_k,\omega(z_0)) \to 0$ a.s.;
        \item $\omega(z_0)$ is a.s. compact and connected;
        \item $\EE \Phi(z^\star) = \Phi^\star$ for all $z^\star \in \omega(z_0)$.
    \end{enumerate}
\end{lemma}

\begin{remark}
The boundedness of $z_k$ is also imposed in the original PALM \cite{bolte2014proximal} and asynchronous PALM \cite{davisfull}, it is satisfied automatically if, for instance, each regularizer has bounded domain.
\end{remark}

\begin{proof}
    By Lemma \ref{lem:seqdec}, we have
    \begin{equation*}
        \E \Psi_{k+1} + \mathcal{O}\big( \|z_k - z_{k-1}\|^2  \big) \le \Psi_k.
    \end{equation*}
    The supermartingale convergence theorem implies that $\sum_{k = 1}^\infty \|z_k - z_{k-1}\|^2 < +\infty$ a.s., and it follows that $\|z_k - z_{k-1}\| \to 0$ a.s. This proves Claim 1.

    The supermartingale convergence theorem also ensures $\Psi_k$ converges a.s.\! to a finite, positive random variable. Because $\|z_k - z_{k-1}\| \to 0$ a.s. and $\tnabla$ is variance-reduced so $\EE \Upsilon_k \to 0$, we can say $\lim_{k\to \infty} \EE \Psi_k = \lim_{k\to \infty} \EE \Phi(z_k) \in [\underline{\Phi}, \infty)$, implying Claim 2.

    Claim 3 holds because, by Lemma \ref{lem:subbound},
    \begin{equation*}
        \EE \|(A_x^k, A_y^k)\| \le p \EE [\| z_k - z_{k-1} \| + \|z_{k-1} - z_{k-2}\|] + \EE \Gamma_{k-1}.
    \end{equation*}
    We have that $\EE \|z_k - z_{k-1}\| \to 0$ and $\EE \Gamma_k \to 0$. This ensures that $\EE \|(A_x^k, A_y^k)\| \to 0$.

    To prove Claim 4, suppose $z^\star = (x^\star,y^\star)$ is a limit point of the sequence $\{z_k\}_{k=0}^\infty$ (a limit point must exist because we suppose the sequence $\{z_k\}_{k=0}^\infty$ is bounded). This means there exists a subsequence $z_{k_q}$ satisfying $\lim_{q \to \infty} z_{k_q} \to z^\star$. Because $R$ and $J$ are lower semicontinuous,
    \begin{equation}
    \label{eq:liminf}
        \liminf_{q \to \infty} R(x_{k_q}) \ge R(x^\star) \qquad \textnormal{and} \qquad \liminf_{q \to \infty} J(x_{k_q}) \ge J(x^\star).
    \end{equation}
    By the update rule for $x_{k+1}$,
    \begin{equation*}
        x_{k+1} \in \argmin_x \big\{ \langle x - x_k, \tnablax (x_k,y_k) \rangle + \tfrac{1}{2 \gmaxk} \|x - x_k\|^2 + R(x) \big\}.
    \end{equation*}
    Letting $x = x^\star$,
    \begin{equation*}
        \begin{aligned}
            & \langle x_{k+1} - x_k, \tnablax (x_k,y_k) \rangle + \tfrac{1}{2 \gmaxk} \|x_{k+1} - x_k\|^2 + R(x_{k+1}) \\
            &\le \langle x^\star - x_k, \nablax F(x_k,y_k) \rangle + \langle x^\star - x_k, \tnablax (x_k,y_k) - \nablax F(x_k,y_k) \rangle + \tfrac{1}{2 \gmaxk} \|x^\star - x_k\|^2 + R(x^\star) .
        \end{aligned}
    \end{equation*}
    Setting $k = k_q$ and taking the limit $q \to \infty$,
    \begin{equation*}
    \begin{aligned}
        \limsup_{q \to \infty} R(x_{k_q + 1})
        &\le \limsup_{q \to \infty} \langle x^\star - x_{k_q}, \nablax F(x_{k_q},y_{k_q}) \rangle \\
        & \qquad  + \langle x^\star - x_{k_q}, \tnablax (x_{k_q},y_{k_q}) - \nablax F(x_{k_q},y_{k_q}) \rangle + \tfrac{1}{2 \gmaxk} \|x^\star - x_{k_q}\|^2 + R(x^\star).
    \end{aligned}
    \end{equation*}
    Because $x_{k_q} \to x^\star$, we can say $\limsup_{q \to \infty} R(x_{k_q + 1}) \le R(x^\star)$, which, together with equation \eqref{eq:liminf}, implies $R(x_{k_q + 1}) \to R(x^\star)$. The same argument holds for $J$ and $y_k$, and it follows that
    \begin{equation*}
        \lim_{q \to \infty} \Phi(x_{k_q}, y_{k_q}) = \Phi(x^\star,y^\star).
    \end{equation*}
    Claim 3 ensures that $(x^\star,y^\star)$ is a critical point of $\Phi$ because $\EE \dist(0,\partial \Phi(z^\star)) \to 0$ as $k \to \infty$ and $\partial \Phi(x^\star,y^\star)$ is closed. Claims 5 and 6 hold for any sequence satisfying $\|z_k - z_{k-1}\| \to 0$ a.s. (this fact is used in the same context in \cite[Remark 5]{bolte2014proximal} and \cite[Remark 4.1]{davisfull}). Finally, we must show that $\Phi$ has constant expectation over $\omega(z_0)$. From Claim 2, we have  $\EE \Phi(z_k) \to \Phi^\star$ which implies $\EE \Phi(z_{k_q}) \to \Phi^\star$ for every subsequence $\{z_{k_q}\}_{q=0}^\infty$ converging to some $z^\star \in \omega(z_0)$. In the proof of Claim 4, we show that $\Phi(z_{k_q}) \to \Phi(z^\star)$, so $\EE \Phi(z^\star) = \Phi^\star$ for all $z^\star \in \omega(z_0)$.
\end{proof}

The following lemma is analogous to the Uniformized \KL Property \cite{bolte2014proximal}. It is a slight generalization of the \KL property showing that $z_k$ eventually enters a region of $\overline{z}$ for some $\overline{z}$ satisfying $\Phi(\overline{z}) = \Phi(z^*)$, and in this region, the \KL inequality holds.

\begin{lemma}
\label{lem:exkl}
    Assume the conditions of Lemma \ref{lem:4case} hold and that $z_k$ is not a critical point of $\Phi$ after a finite number of iterations.
    Let $\Phi$ be a semi-algebraic function satisfying KL property with exponent $\theta$. Then there exists an index $m$ and a desingularizing function $\phi$ so that the following bound holds:
    \begin{equation*}
        \phi'(\EE [\Phi(z_k) - \Phi_k^\star] ) \EE \dist \big(0, \partial \Phi(z_k) \big) \ge 1 \qquad \forall k > m,
    \end{equation*}
    where $\Phi_k^\star$ is a non-decreasing sequence converging to $\EE \Phi(z^\star)$ for some $z^\star \in \omega(z_0)$.
\end{lemma}

\begin{proof}
    First, we show that $\EE \Phi(z_k)$ satisfies the KL property. Recall that $b$ is the mini-batch size. Let $\overline{n} = \binom{n}{b}$ be the number of possible gradient estimates in one iteration, and let $\{z_k^i\}_{i=1}^{\overline{n}^k}$ be the set of possible values for $z_k$. It is clear that $\EE \Phi$ is a function of $\{z_k^i\}_{i=1}^{\overline{n}^k}$:
    \begin{equation*}
        \EE \Phi(z_k) = \tfrac{1}{\overline{n}^k} \sum\nolimits_{i=1}^{\overline{n}^k} \Phi(z_k^i).
    \end{equation*}
    Because $\EE \Phi(z_k)$ can be written as $\sum_i f_i(x_i)$ where $f_i$ are KL functions with exponent $\theta$, $\EE \Phi(z_k)$ (as a function of $\{z_k^i\}_{i=1}^{\overline{n}^k}$) is also KL with exponent $\theta$ \cite[Theorem 3.3]{klcalc}.
    Hence, $\EE \Phi$ satisfies the KL inequality at every point in its domain. Therefore, for every point $(z_k^1, \cdots,z_k^{\overline{n}^k})$ in a neighborhood $U_k$ of $(\overline{z}_k^1, \overline{z}_k^2, \cdots,\overline{z}_k^{\overline{n}^k})$ and satisfying
    \begin{equation}
    \label{eq:locmin}
    \tfrac{1}{\overline{n}^k} \sum\nolimits_{i=1}^{\overline{n}^k} \Phi(\overline{z}_k^i) < \tfrac{1}{\overline{n}^k} \sum\nolimits_{i=1}^{\overline{n}^k} \Phi(z_k^i) < \tfrac{1}{\overline{n}^k} \sum\nolimits_{i=1}^{\overline{n}^k} \Phi(\overline{z}_k^i) + \epsilon_k
    \eeq
    for some $\epsilon_k > 0$, the \KL inequality holds:
    \begin{equation*}
        \phi' \Big( \tfrac{1}{\overline{n}^k} \sum\nolimits_{i=1}^{\overline{n}^k} \Phi(z_k^i) - \tfrac{1}{\overline{n}^k} \sum\nolimits_{i=1}^{\overline{n}^k} \Phi(\overline{z}_k^i) \Big) \dist \Big(0, \tfrac{1}{\overline{n}^k} \sum\nolimits_{i=1}^{\overline{n}^k} \partial \Phi(z_k^i) \Big) \ge 1.
    \end{equation*}
    There always exists a choice of $(\overline{z}_k^1, \overline{z}_k^2, \cdots,\overline{z}_k^{\overline{n}^{k}})$ satisfying \eqref{eq:locmin} unless $\EE \Phi(z_k)$ is a local minimum.

    Let $\Phi^\star_k \defeq \tfrac{1}{\overline{n}^k} \sum_{i=1}^{\overline{n}^k} \Phi(\overline{z}_k^i)$. Lemma \ref{lem:4case}, Claim 1 implies that $\|z_{k+1} - z_k\| \to 0$ a.s., and Claim 5 implies $\dist(z_k, \omega(z_0)) \to 0$ a.s. These results show that there exists an index $m$ such that for all $k \ge m$, we can choose $\overline{z}^i_k$ so that $\Phi_k^\star$ is non-decreasing and converging to $\EE \Phi(z^\star)$. Hence, we have
    \begin{equation*}
        \phi'(\EE [\Phi(z_k) - \Phi^\star_k] ) \dist \big(0, \EE \partial \Phi(z_k) \big) \ge 1 \qquad \forall k > m,
    \end{equation*}
    The desired inequality follows from Jensen's inequality and the convexity of $x \mapsto \dist(0,$ $x)$.
\end{proof}

We now show that the iterates of SPRING have finite length in expectation.

\begin{lemma}[Finite Length]
\label{lem:finitelength}
    Suppose $\Phi$ is a semi-algebraic function with KL exponent $\theta \in [0,1)$. Let $\{ z_k \}_{k=0}^\infty$ be a bounded sequence of iterates of SPRING using a variance-reduced gradient estimator and step-sizes satisfying the hypotheses of Lemma \ref{lem:4case}.
    Either $z_k$ is a critical point after a finite number of iterations, or $\{ z_k \}_{k=0}^\infty$ satisfies the finite length property in expectation:
    \begin{equation*}
        \sum\nolimits_{k=0}^\infty \EE \|z_{k+1} - z_k\| < \infty,
    \end{equation*}
    and there exists an iteration $m$ so that for all $i > m$,
    \begin{equation*}
    \begin{aligned}
        \sum\nolimits_{k=m}^i \EE \|z_{k+1} - z_k\| + \EE \|z_k - z_{k-1}\| &\le \sqrt{ \EE \|z_m - z_{m-1} \|^2 } + \sqrt{ \EE \|z_{m-1} - z_{m-2} \|^2 } \\
        &\qquad + \tfrac{2 \sqrt{s}}{K_1 \rho} \sqrt{ \EE \Upsilon_{m-1} } + K_3 \Delta_{m,i+1},
    \end{aligned}
    \end{equation*}
    where
    \begin{equation*}
        K_1 \defeq p + 2 \sqrt{s V_\Upsilon} / \rho, \qquad K_2
        \defeq \tfrac{1}{2 \overline{\gamma}_0} - \tfrac{\bar{L}}{2} - \tfrac{3 \sqrt{2}}{4} \sqrt{V_1 + V_\Upsilon / \rho}, \qquad K_3 \defeq \tfrac{2 K_1 (K_2+ Z)}{K_2 Z},
    \end{equation*}
    $p$ is as in Lemma \ref{lem:subbound}, and $\Delta_{p,q} \defeq \phi(\EE [\Psi_p - \Psi^\star_p]) - \phi(\EE [ \Psi_q - \Psi^\star_q ]) ]$.
\end{lemma}

\begin{remark}
    Our analysis for SPRING requires $\Phi$ to be semi-algebraic for the finite-length property to hold, but in the analysis of PALM, the finite-length property requires only that $\Phi$ is KL \cite[Thm. 1]{bolte2014proximal}. This difference arises because SPRING does not necessarily decrease the objective every iteration (even in expectation), but PALM does \cite[Lem. 3]{bolte2014proximal}. Instead, we prove that the iterates of SPRING decrease $\Psi_k$ in expectation. Related works \cite{davisfull} solve this problem by requiring an analog of $\Psi_k$ to be KL, but this is not a straightforward approach for SPRING because of the complex variance bounds required to analyze variance-reduced gradient estimators.
\end{remark}

\begin{proof}

    We begin with a proof of Claim 1. If $\theta \in (0,1/2)$, then $\Phi$ satisfies the KL property with exponent $1/2$, so we consider only the case $\theta \in [1/2,1)$. By Lemma \ref{lem:exkl}, there exists a function $\phi_0(r) = a r^{1-\theta}$ such that
    \begin{equation*}
        \phi_0'(\EE [\Phi(z_k) - \Phi(z^\star)] ) \EE \dist \big(0, \partial \Phi_k^\star \big) \ge 1 \qquad \forall k > m.
    \end{equation*}
    Lemma \ref{lem:subbound} provides a bound on $\EE \dist (0, \partial \Phi(z_k) )$.
    \begin{equation}
    \begin{aligned}
    \label{eq:subnormbound}
        \EE \dist \big(0, \partial \Phi(z_k) \big)  \le \EE \|(A_x^k,A_y^k)\|
        & \le p \EE [ \| z_k - z_{k-1} \|  + \|z_{k-1} - z_{k-2}\|] + \EE \Gamma_{k-1} \\
        & \le p ( \sqrt{ \EE \| z_k - z_{k-1} \|^2 } + \sqrt{ \EE \|z_{k-1} - z_{k-2}\|^2 }) + \sqrt{s \EE \Upsilon_{k-1} }.
    \end{aligned}
    \end{equation}
    The final inequality is Jensen's. Because $\Gamma_k = \sum_{i=1}^s \|v_k^i\|$ for some vectors $v_k^i$, we can say $\EE \Gamma_k = \EE \sum_{i=1}^s \|v_k^i\| \le \EE \sqrt{s \sum_{i=1}^s \|v_k^i\|^2 } \le \sqrt{s \EE \Upsilon_k}$. We can bound the term $\sqrt{ \EE \Upsilon_k }$ using \eqref{eq:geo}:
    \begin{equation}
    \begin{aligned}
    \label{eq:sqrtgeo}
        \sqrt{ \EE \Upsilon_k } & \le \sqrt{ (1 - \rho) \EE \Upsilon_{k-1} + V_\Upsilon \EE [ \|z_k - z_{k-1}\|^2 + \|z_{k-1} - z_{k-2}\|^2 ] } \\
        & \le \sqrt{ (1 - \rho) } \sqrt{ \EE \Upsilon_{k-1} } +\sqrt{V_\Upsilon} ( \sqrt{ \EE \|z_k - z_{k-1}\|^2 } + \sqrt{ \EE \|z_{k-1} - z_{k-2}\|^2} ) \\
        & \le \big( 1 - \tfrac{\rho}{2} \big) \sqrt{ \EE \Upsilon_{k-1} } +\sqrt{V_\Upsilon} ( \sqrt{ \EE \|z_k - z_{k-1}\|^2 } + \sqrt{ \EE \|z_{k-1} - z_{k-2}\|^2} ).
    \end{aligned}
    \end{equation}
    The final inequality uses the fact that $\sqrt{1-\rho} = 1 - \rho/2-\rho^2/8 -\cdots$. This allows us to say
    \begin{equation}
    \begin{aligned}
    \label{eq:distbound}
        \EE \dist \big(0, \partial \Phi(z_k) \big)
        &\le K_1 \sqrt{\EE \| z_k - z_{k-1} \|^2} + K_1 \sqrt{ \EE \|z_{k-1} - z_{k-2}\|^2 } + \tfrac{2 \sqrt{s}}{\rho} ( \sqrt{ \EE \Upsilon_{k-1} } - \sqrt{ \EE \Upsilon_k } ),
    \end{aligned}
    \end{equation}
    where $K_1 \defeq p + 2 \sqrt{s V_\Upsilon} / \rho$. Define $C_k$ to be the right side of this inequality:
    \begin{equation*}
        C_k \defeq K_1 \sqrt{ \EE \| z_k - z_{k-1} \|^2 } + K_1 \sqrt{ \EE \|z_{k-1} - z_{k-2}\|^2 } + \tfrac{2 \sqrt{s}}{\rho} ( \sqrt{ \EE \Upsilon_{k-1} } - \sqrt{ \EE \Upsilon_k } ).
    \end{equation*}
    We then have
    \begin{equation}
    \label{eq:psisum}
        \phi_0'( \EE [\Phi(z_k) - \Phi_k^\star] ) C_k \ge 1 \qquad \forall k > m.
    \end{equation}
    By the definition of $\phi_0$, this is equivalent to
    \begin{equation}
    \label{eq:desi}
        \frac{a ( 1 - \theta ) C_k}{( \EE [\Phi(z_k) - \Phi_k^\star] )^{\theta}} \ge 1 \qquad \forall k > m.
    \end{equation}
    We would like the inequality above to hold for $\Psi_k$ rather than $\Phi(z_k)$. Replacing $\EE \Phi(z_k)$ with $\EE \Psi_k$ introduces a term of $\mathcal{O}( (\EE [\|z_k - z_{k-1}\|^2 + \Upsilon_k] )^{\theta})$ in the denominator. We show that inequality \eqref{eq:desi} still holds after this adjustment because these terms are small compared to $C_k$.

    The quantity $C_k \ge \mathcal{O}(\sqrt{\EE \|z_k - z_{k-1}\|^2} + \sqrt{\EE \|z_{k-1} - z_{k-2}\|^2} + \sqrt{\EE \Upsilon_{k-1}})$, and because $\EE \|z_k - z_{k-1}\|^2$, $\EE \Upsilon_k \to 0$, and $\theta \ge 1 / 2$, there exists an index $m$ and a constant $c > 0$ such that
    \begin{equation*}
    \begin{aligned}
        & \Big( \EE \Big[ \tfrac{1}{2 \rho \sqrt{2 (V_1 + V_\Upsilon / \rho)}} \Upsilon_k + \tfrac{\sqrt{V_1 + V_\Upsilon / \rho}}{\sqrt{2}} \|z_k - z_{k-1}\|^2 \Big] \Big)^{\theta} \\
        &\le \mathcal{O} \Big( \big( \EE \left[ \Upsilon_{k-1} + \|z_k - z_{k-1}\|^2 + \|z_{k-1} - z_{k-2}\|^2 \right] \big)^{\theta} \Big)
        \le c  C_k \qquad \forall k > m.
    \end{aligned}
    \end{equation*}
    The first inequality uses \eqref{eq:geo}. Because the terms above are small compared to $C_k$, there exists a constant $+\infty > d > c$ such that
    \begin{equation*}
        \tfrac{a d ( 1 - \theta ) C_k }{( \EE[ \Phi(z_k) - \Phi_k^\star] )^{\theta} + \big( \EE [ \frac{1}{2 \rho \sqrt{2 (V_1 + V_\Upsilon / \rho)}} \Upsilon_k + \frac{\sqrt{V_1 + V_\Upsilon / \rho}}{\sqrt{2}} \|z_k - z_{k-1}\|^2 ] \big)^{\theta}} \ge 1,
    \end{equation*}
    for all $k > m$. Using the fact that $(a + b)^\theta \le a^\theta + b^\theta$ for all $a,b \ge 0$ because $\theta \in [1/2,1)$, we have
    \begin{equation*}
    \begin{aligned}
        \tfrac{a d ( 1 - \theta ) C_k}{( \EE [\Psi_k - \Psi^\star] )^{\theta}}
        & = \tfrac{a d ( 1 - \theta ) C_k }{\big( \EE \big[ \Phi(z_k) - \Phi_k^\star + \tfrac{1}{2 \rho \sqrt{2 (V_1 + V_\Upsilon / \rho)}} \Upsilon_k + \tfrac{\sqrt{V_1 + V_\Upsilon / \rho}}{\sqrt{2}} \|z_k - z_{k-1}\|^2 \big] \big)^{\theta}} \\
        & \ge \tfrac{a d ( 1 - \theta ) C_k }{\left(\EE \left[ \Phi(z_k) - \Phi_k^\star \right]\right)^{\theta} + \big( \EE\big[ \tfrac{1}{2 \rho \sqrt{2 (V_1 + V_\Upsilon / \rho)}} \Upsilon_k + \tfrac{\sqrt{V_1 + V_\Upsilon / \rho}}{\sqrt{2}} \|z_k - z_{k-1}\|^2 \big] \big)^{\theta}}
        \ge 1 \qquad \forall k > m.
    \end{aligned}
    \end{equation*}
    Therefore, with $\phi(r) = a d r^{1-\theta}$,
    \begin{equation*}
        \phi'(\EE [\Psi_k - \Phi_k^\star] ) C_k \ge 1 \qquad \forall k > m.
    \end{equation*}
    By the concavity of $\phi$,
    \begin{equation}
    \begin{aligned}
    \label{eq:concav}
        \phi( \EE [ \Psi_k - \Phi_k^\star ] ) - \phi( \EE [\Psi_{k+1} - \Phi_{k+1}^\star] ) & \ge \phi'(\EE [\Psi_k - \Phi_k^\star]) (\EE [\Psi_k - \Phi_k^\star + \Phi_{k+1}^\star - \Psi_{k+1}]) \\
        &\ge \phi'(\EE [\Psi_k - \Phi_k^\star]) (\EE [\Psi_k - \Psi_{k+1}]),
    \end{aligned}
    \end{equation}
    where the last inequality follows from the fact that $\Phi_k^\star$ is non-decreasing. With $\Delta_{p,q} \defeq \phi(\EE [\Psi_p - \Phi^\star_p]) - \phi(\EE [ \Psi_q - \Phi^\star_q ]) ]$, we have shown
    \begin{equation*}
        \Delta_{k,k+1} C_k \ge \EE [\Psi_k - \Psi_{k+1}].
    \end{equation*}
    Using Lemma \ref{lem:seqdec}, we can bound $\EE[ \Psi_k - \Psi_{k+1} ]$ below by both $\EE \|z_{k+1} - z_k\|^2$ and $\EE \|z_k - z_{k-1}\|^2$. Specifically,
    \begin{equation}
    \label{eq:Z}
        \Delta_{k,k+1} C_k \ge Z \EE [\|z_k - z_{k-1}\|^2],
    \end{equation}
    as well as
    \begin{equation}
    \label{eq:K}
        \Delta_{k,k+1} C_k \ge K_2 \EE [\|z_{k+1} - z_k\|^2],
    \end{equation}
    where $K_2 \defeq - \big( \tfrac{\bar{L} (\lambda + 1)}{2} + \tfrac{V_1 + V_\Upsilon / \rho}{\bar{L} \lambda} + Z - \tfrac{1}{2 \overline{\gamma}_0} \big)$
    and $\lambda$ and $Z$ are set as in Lemma \ref{lem:seqdec}. Let us use the first of these inequalities to begin. Applying Young's inequality to \eqref{eq:Z} yields
    \begin{equation}
    \begin{aligned}
    \label{eq:zc}
        2 \sqrt{ \EE \|z_k - z_{k-1}\|^2 } & \le 2 \sqrt{C_k \Delta_{k,k+1} Z^{-1}}
        \le \tfrac{C_k}{2 K_1} + \tfrac{2 K_1 \Delta_{k,k+1}}{Z}
    \end{aligned}
    \end{equation}
    Summing inequality \eqref{eq:zc} from $k = m $ to $k = i$,
    \begin{equation}
        \begin{aligned}
        \label{eq:rewrite}
            2 \sum\nolimits_{k=m}^i \sqrt{ \EE \|z_k - z_{k-1}\|^2 } & \le \sum\nolimits_{k=m}^i \tfrac{C_k}{2 K_1} + \tfrac{2 K_1 \Delta_{m,i+1}}{Z} \\
            & \le \sum\nolimits_{k=m}^i \tfrac{1}{2} \sqrt{\EE \|z_k - z_{k-1}\|^2} + \tfrac{1}{2} \sqrt{ \EE \|z_{k-1} - z_{k-2}\|^2 } \\
            & \qquad - \tfrac{\sqrt{s} }{K_1 \rho} \big( \sqrt{ \EE \Upsilon_i } - \sqrt{ \EE \Upsilon_{m-1} } \big) + \tfrac{2 K_1 \Delta_{m,i+1}}{Z},
        \end{aligned}
    \end{equation}
    Dropping the non-positive term $-\sqrt{ \EE \Upsilon_i }$, this shows that
    \begin{equation*}
        \tfrac{3}{2} \sum\nolimits_{k=m}^i \sqrt{ \EE \|z_k - z_{k-1}\|^2 } \le \tfrac{1}{2} \sqrt{ \EE \|z_{m-1} - z_{m-2} \|^2 } + \tfrac{\sqrt{s} }{K_1 \rho} \sqrt{ \EE \Upsilon_{m-1} } + \tfrac{2 K_1 \Delta_{m,i+1}}{Z}.
    \end{equation*}
    Applying the same argument using inequality \eqref{eq:K} instead of \eqref{eq:Z}, we obtain
    \begin{equation*}
    \begin{aligned}
        & \tfrac{3}{2} \sum\nolimits_{k=m}^i \sqrt{ \EE \|z_{k+1} - z_k\|^2 } \\
        &\le \tfrac{1}{2} \sqrt{ \EE \|z_m - z_{m-1} \|^2 } + \tfrac{1}{2} \sqrt{ \EE \|z_{m-1} - z_{m-2} \|^2 } + \tfrac{\sqrt{s} }{K_1 \rho} \sqrt{ \EE \Upsilon_{m-1} } + \tfrac{2 K_1 \Delta_{m,i+1}}{K_2}.
    \end{aligned}
    \end{equation*}
    Adding these inequalities together, we have
    \begin{equation*}
    \begin{aligned}
        \tfrac{3}{2} \big( \sum\nolimits_{k=m}^i \sqrt{ \EE \|z_{k+1} - z_k\|^2} + \sqrt{\EE\|z_k - z_{k-1}\|^2} \big) &\le \tfrac{1}{2} \sqrt{\EE \|z_m - z_{m-1} \|^2} + \sqrt{ \EE \|z_{m-1} - z_{m-2} \|^2 } \\
        &\qquad + \tfrac{2 \sqrt{s} }{K_1 \rho} \sqrt{ \EE \Upsilon_{m-1} } + \tfrac{2 K_1 (K_2+ Z) \Delta_{m,i+1} }{K_2 Z}.
    \end{aligned}
    \end{equation*}
    For easier analysis, we use a slightly looser inequality:
    \begin{equation}
    \begin{aligned}
    \label{eq:final}
        &\sum\nolimits_{k=m}^i \sqrt{ \EE \|z_{k+1} - z_k\|^2} + \sqrt{\EE\|z_k - z_{k-1}\|^2} \\ &\le \sqrt{ \EE \|z_m - z_{m-1} \|^2 } + \sqrt{ \EE \|z_{m-1} - z_{m-2} \|^2 } + \tfrac{2 \sqrt{s} }{K_1 \rho} \sqrt{ \EE \Upsilon_{m-1} } + \tfrac{2 K_1 (K_2+ Z) \Delta_{m,i+1} }{K_2 Z}.
    \end{aligned}
    \end{equation}
    Applying Jensen's inequality to the terms on the left gives
    \begin{equation*}
    \begin{aligned}
        & \sum\nolimits_{k=m}^i \EE \|z_{k+1} - z_k\| + \EE \|z_k - z_{k-1}\| \\
        &\le \sqrt{ \EE \|z_m - z_{m-1} \|^2 } + \sqrt{ \EE \|z_{m-1} - z_{m-2} \|^2 } + \tfrac{2 \sqrt{s} }{K_1 \rho} \sqrt{ \EE \Upsilon_{m-1} } + \tfrac{2 K_1 (K_2 + Z) \Delta_{m,i+1} }{K_2 Z},
    \end{aligned}
    \end{equation*}
    and letting $i \to \infty$ proves the assertion.

    An immediate consequence of Claim 1 is that the sequence $\EE \|z_{k+1} - z_k\|$ is Cauchy, so the sequence $\{z_k\}_{k=0}^\infty$ converges in expectation to a critical point. This is because, for any $p, q \in \mathbb{N}$ with $p \ge q$, $\EE \|z_p - z_q\| = \EE \|\msum_{k=q}^{p-1} z_{k+1} - z_k\| \le \msum_{k=q}^{p-1} \EE \| z_{k+1} - z_k\|$, and the finite length property implies this final sum converges to zero. This proves Claim 2.
\end{proof}

Finally, we prove convergence rates for SPRING depending on the KL exponent of the objective function, demonstrating that the full convergence theory of PALM extends to SPRING.

\begin{theorem}[Convergence Rates]
    Suppose $\Phi$ is a semi-algebraic function with KL exponent $\theta \in [0,1)$. Let $\{z_k\}_{k=0}^\infty$ be a bounded sequence of iterates of SPRING using a variance-reduced gradient estimator and step-sizes satisfying the hypotheses of Lemma \ref{lem:4case}. The following convergence rates hold:
    \begin{enumerate}
        \item If $\theta = 0$, then there exists an $m \in \mathbb{N}$ such that $\EE \Phi(z_k) = \EE \Phi(z^\star)$ for all $k \ge m$.
        \item If $\theta \in (0,1/2]$, then there exists $d_1 > 0$ and $\tau \in [1 - \rho, 1)$ such that $\EE \|z_k - z^\star\| \le d_1 \tau^k$.
        \item If $\theta \in (1/2,1)$, then there exists a constant $d_2 > 0$ such that $\EE \|z_k - z^\star\| \le d_2 k^{-\frac{1-\theta}{2 \theta - 1}}$.
    \end{enumerate}
\end{theorem}

\begin{proof}
    As in the proof of the previous lemma, if $\theta \in (0,1/2)$, then $\Phi$ satisfies the KL property with exponent $1/2$, so we consider only the case $\theta \in [1/2,1)$.

    Substituting the desingularizing function $\phi(r) = a r^{1 - \theta}$ into \eqref{eq:final},
    \begin{equation*}
    \begin{aligned}
         &\sum\nolimits_{k=m}^\infty \sqrt{ \EE \|z_{k+1} - z_k\|^2} + \sqrt{\EE\|z_k - z_{k-1}\|^2} \\
         &\le \sqrt{ \EE \|z_m - z_{m-1} \|^2 } + \sqrt{ \EE \|z_{m-1} - z_{m-2} \|^2}  + \tfrac{2 \sqrt{s}}{K_1 \rho} \sqrt{ \EE \Upsilon_{m-1} } + a K_3 ( \EE [\Psi_m - \Psi^\star])^{1-\theta}.
     \end{aligned}
    \end{equation*}
    Because $\Psi_m = \Phi(z_m) + \mathcal{O}(\|z_m - z_{m-1}\|^2 + \Upsilon_m)$, we can rewrite the final term as $\Phi(z_m) - \Phi_m^\star$.
    \begin{equation*}
    \begin{aligned}
        (\EE [\Psi_m - \Phi^\star_m])^{1-\theta}
        &=  (\EE [ \Phi(z_m) - \Phi_m^\star + \tfrac{1}{2 \bar{L} \lambda \rho} \Upsilon_m + \tfrac{V_1 + V_\Upsilon / \rho}{2 \bar{L} \lambda} \|z_m - z_{m-1}\|^2 ])^{1-\theta} \\
        \symnum{1}{\le} & (\EE [\Phi(z_m) - \Phi_m^\star])^{1-\theta} + \big( \tfrac{1}{2 \bar{L} \lambda \rho} \EE \Upsilon_m \big)^{1-\theta} + \big( \tfrac{V_1 + V_\Upsilon / \rho}{2 \bar{L} \lambda} \EE \|z_m - z_{m-1}\|^2 \big)^{1-\theta}.
    \end{aligned}
    \end{equation*}
    Inequality \numcirc{1} is due to the fact that $(a+b)^{1-\theta} \le a^{1-\theta} + b^{1-\theta}$. This yields the inequality
    \begin{equation*}
    \begin{aligned}
         & \sum\nolimits_{k=m}^\infty \sqrt{ \EE \|z_{k+1} - z_k\|^2} + \sqrt{\EE\|z_k - z_{k-1}\|^2} \\
         & \le \sqrt{ \EE \|z_m - z_{m-1} \|^2 } + \sqrt{ \EE \|z_{m-1} - z_{m-2} \|^2} + \tfrac{2 \sqrt{s}}{K_1 \rho} \sqrt{ \EE \Upsilon_{m-1} } + a K_3 ( \EE [\Phi(z_m) - \Phi_m^\star])^{1-\theta} \\
         & \qquad + a K_3 \big( \tfrac{1}{2 \bar{L} \lambda \rho} \EE \Upsilon_m \big)^{1-\theta} + a K_3 \big( \tfrac{V_1 + V_\Upsilon / \rho}{2 \bar{L} \lambda} \EE \|z_m - z_{m-1}\|^2 \big)^{1-\theta}.
     \end{aligned}
    \end{equation*}
    Applying the \KL inequality \eqref{eq:kl},
    \begin{equation}
    \label{eq:main}
        a K_3 ( \EE \left[\Phi(z_m) - \Phi_m^\star \right] )^{1-\theta} \le a K_3 ( \EE \|\zeta_m\| )^{\frac{1-\theta}{\theta}},
    \end{equation}
    where $\zeta_m \in \partial \Phi(z_m)$ and we have absorbed the constant $C$ into $a$. Equation \eqref{eq:subnormbound} provides a bound on the norm of the subgradient:
    \begin{equation*}
    \begin{aligned}
        (\EE \|\zeta_m\|)^{\frac{1-\theta}{\theta}} & \le \big( p ( \sqrt{\EE \| z_m - z_{m-1} \|^2} + \sqrt{\EE \|z_{m-1} - z_{m-2}\|^2} ) + \sqrt{s \EE \Upsilon_{m-1}} \big)^{\frac{1 - \theta}{\theta}}.
        \end{aligned}
    \end{equation*}
    Denote the right side of this inequality $\Theta_m^{\frac{1-\theta}{\theta}}$. Therefore,
    \begin{equation}
    \label{eq:dom}
    \begin{aligned}
        & \sum\nolimits_{k=m}^\infty \sqrt{ \EE \|z_{k+1} - z_k\|^2} + \sqrt{\EE\|z_k - z_{k-1}\|^2} \\
        &\le \sqrt{ \EE \|z_m - z_{m-1} \|^2 } + \sqrt{ \EE \|z_{m-1} - z_{m-2} \|^2} + \tfrac{2 \sqrt{s}}{K_1 \rho} \sqrt{ \EE \Upsilon_{m-1} } + a K_3 \Theta_m^{\frac{1-\theta}{\theta}} \\
        & + a K_3 \big( \tfrac{1}{2 \bar{L} \lambda \rho} \EE \Upsilon_m \big)^{1-\theta} + a K_3 \big( \tfrac{V_1 + V_\Upsilon / \rho}{2 \bar{L} \lambda} \EE \|z_m - z_{m-1}\|^2 \big)^{1-\theta}.
    \end{aligned}
    \end{equation}
    Suppose $\theta \in (1/2, 1)$. Each of the terms on the right side of this inequality are converging to zero, but at different rates. Because $\Theta_m = \mathcal{O}(\sqrt{ \EE \|z_m - z_{m-1} \|^2 } + \sqrt{\EE \|z_{m-1} - z_{m-2} \|^2} + \sqrt{\EE \Upsilon_{m-1}})$, and $\theta$ satisfies $\frac{1-\theta}{\theta} < 1$, the term $\Theta_m^{\frac{1-\theta}{\theta}}$ dominates the first three terms on the right side of this inequality for large $m$. Also, because $\frac{1-\theta}{2 \theta} \le 1 - \theta$, $\Theta_m^{\frac{1-\theta}{\theta}}$ dominates the final two terms as well. Combining these facts, there exists a natural number $M_1$ such that for all $m \ge M_1$,
    \begin{equation}
    \label{eq:simple}
        \big( \sum\nolimits_{k=m}^\infty \sqrt{ \EE \|z_{k+1} - z_k\|^2} + \sqrt{\EE\|z_k - z_{k-1}\|^2} \big)^{\frac{\theta}{1 - \theta}} \le P \Theta_m,
    \end{equation}
    for some constant $P > (a K_3)^{\frac{\theta}{1 - \theta}}$. The bound of \eqref{eq:sqrtgeo} implies
    \begin{equation*}
    \begin{aligned}
        2 \sqrt{s \EE \Upsilon_{m-1}} & \le \tfrac{4 \sqrt{s}}{\rho} \big( \sqrt{\EE \Upsilon_{m-1}} - \sqrt{\EE \Upsilon_m} + \sqrt{V_\Upsilon} (\sqrt{\EE \|z_m - z_{m-1}\|^2} + \sqrt{\EE \|z_{m-1} - z_{m-2}\|^2}) \big).
    \end{aligned}
    \end{equation*}
    Therefore,
    \begin{equation}
    \begin{aligned}
    \label{eq:new}
        \Theta_m & = p ( \sqrt{ \EE \|z_m - z_{m-1} \|^2 } + \sqrt{\EE \|z_{m-1} - z_{m-2} \|^2} ) + (2 \sqrt{ s \EE \Upsilon_{m-1} } - \sqrt{ s \EE \Upsilon_{m-1} }) \\
        & \le \big( p + \tfrac{4 \sqrt{s V_\Upsilon} }{\rho} \big) ( \sqrt{ \EE \|z_m - z_{m-1} \|^2 } + \sqrt{\EE \|z_{m-1} - z_{m-2} \|^2} ) \\
        & \qquad + \tfrac{4 \sqrt{s} }{\rho} ( \sqrt{ \EE \Upsilon_{m-1} } - \sqrt{\EE \Upsilon_m} ) - \sqrt{ s \EE \Upsilon_{m-1}}.
    \end{aligned}
    \end{equation}
    Furthermore, because $\frac{\theta}{1-\theta} > 1$ and $\EE \Upsilon_m \to 0$, for large enough $m$, we have $(\sqrt{\EE \Upsilon_m})^{\tfrac{\theta}{1 - \theta}} \ll \sqrt{\EE \Upsilon_m}$. This ensures that there exists a natural number $M_2$ such that for every $m \ge M_2$,
    \begin{equation}
    \label{eq:upbound}
        \big(\tfrac{4 \sqrt{s} (1 - \rho/4)}{\rho (p + 4 \sqrt{s V_\Upsilon} / \rho ) } \sqrt{\EE \Upsilon_m} \big)^{\frac{\theta}{1 - \theta}} \le P \sqrt{ s \EE \Upsilon_m }.
    \end{equation}
    (The constant appearing on the left was chosen to simplify later arguments.) Therefore, \eqref{eq:simple} implies
    \begin{equation*}
    \begin{aligned}
        & \big( \sum\nolimits_{k=m}^\infty \sqrt{ \EE \|z_{k+1} - z_k\|^2} + \sqrt{\EE\|z_k - z_{k-1}\|^2} + \tfrac{4 \sqrt{s} (1 - \rho/4)}{\rho (p + 4 \sqrt{s V_\Upsilon} / \rho ) } \sqrt{ \EE \Upsilon_m } \big)^{\frac{\theta}{1 - \theta}} \\
        & \symnum{1}{\le} \tfrac{2^{\frac{\theta}{1-\theta}}}{2} \big( \sum\nolimits_{k=m}^\infty \sqrt{ \EE \|z_{k+1} - z_k\|^2} + \sqrt{\EE\|z_k - z_{k-1}\|^2} \big)^{\frac{\theta}{1 - \theta}} \\
        & \qquad + \tfrac{2^{\frac{\theta}{1-\theta}}}{2} \big( \tfrac{4 \sqrt{s} (1 - \rho/4)}{\rho (p + 4 \sqrt{s V_\Upsilon} / \rho ) } \sqrt{ \EE \Upsilon_m } \big)^{\frac{\theta}{1 - \theta}} \\
        & \symnum{2}{\le} \tfrac{2^{\frac{\theta}{1-\theta}}}{2} \big( \sum\nolimits_{k=m}^\infty \sqrt{ \EE \|z_{k+1} - z_k\|^2} + \sqrt{\EE\|z_k - z_{k-1}\|^2} \big)^{\frac{\theta}{1 - \theta}} + \tfrac{2^{\frac{\theta}{1-\theta}}}{2} \big(P \sqrt{ s \EE \Upsilon_m } \big) \\
        & \symnum{3}{\le} \tfrac{2^{\frac{\theta}{1-\theta}}}{2} \Big(P (p + 4 \sqrt{s V_\Upsilon} / \rho) \big( \sqrt{\EE \left\| z_m - z_{m-1} \right\|^2} +
        \sqrt{ \left\| z_{m-1} - z_{m-2} \right\|^2} \big) \\
        & \qquad + \tfrac{4 \sqrt{s} P (1 - \rho / 4)}{\rho} \big( \sqrt{ \EE \Upsilon_{m-1} } - \sqrt{ \EE \Upsilon_m} \big) \Big).
    \end{aligned}
    \end{equation*}
    Here, \numcirc{1} follows by convexity of the function $x^{\frac{\theta}{1-\theta}}$ for $\theta \in [1/2,1)$ and $x \ge 0$, \numcirc{2} is \eqref{eq:upbound}, and \numcirc{3} is \eqref{eq:simple} combined with \eqref{eq:new}. We absorb the constant $\frac{2^{\frac{\theta}{1-\theta}}}{2}$ into $P$. With
    \begin{equation*}
        S_m \defeq \sum\nolimits_{k=m}^\infty \sqrt{ \EE \|z_{k+1} - z_k\|^2 } + \sqrt{ \EE \|z_k - z_{k-1}\|^2} + \tfrac{4 \sqrt{s} P (1 - \rho / 4)}{\rho (p + 4 \sqrt{s V_\Upsilon} / \rho)} \sqrt{\EE \Upsilon_m},
    \end{equation*}
    we have shown
    \begin{equation}
    \label{eq:simplere}
        S_m^{\frac{\theta}{1 - \theta}} \le P (p + 4 \sqrt{s V_\Upsilon} / \rho) (S_{m-1} - S_m),
    \end{equation}
    The rest of the proof follows the proof of \cite[Theorem 5]{attouchbolte}. Let $h(r) \defeq r^{-\frac{\theta}{1 - \theta}}$. First, suppose that $h(S_m) \le R h(S_{m-1})$ for some $R \in (1, \infty)$. Then \eqref{eq:simplere} ensures that
    \begin{equation*}
        \begin{aligned}
        1  \le P (p + 4 \sqrt{s V_\Upsilon} / \rho)(S_{m-1} - S_m) h(S_m)
        & \le R P (p + 4 \sqrt{s V_\Upsilon} / \rho)(S_{m-1} - S_m) h(S_{m-1}) \\
        & \le R P (p + 4 \sqrt{s V_\Upsilon} / \rho)\int_{S_m}^{S_{m-1}} h(r) dr \\
        & = \tfrac{R P (p + 4 \sqrt{s V_\Upsilon} / \rho)(1 - \theta)}{1 - 2 \theta} \Big[ S_{m-1}^{\frac{1 - 2 \theta}{1 - \theta}} - S_m^{\frac{1 - 2 \theta}{1 - \theta}} \Big].
        \end{aligned}
    \end{equation*}
    Hence,
    \begin{equation*}
        0 < - \tfrac{1 - 2 \theta}{R P (p + 4 \sqrt{s V_\Upsilon} / \rho)(1 - \theta)} \le S_{m}^{\frac{1 - 2 \theta}{1 - \theta}} - S_{m-1}^{\frac{1 - 2 \theta}{1 - \theta}}.
    \end{equation*}
    Now suppose $h(S_m) > R h(S_{m-1})$, so that $S_m < R^{-\frac{1 - \theta}{\theta}} S_{m-1}$ and $S_m^{\frac{1 - 2 \theta}{1 - \theta}} > q^{\frac{1 - 2 \theta}{1 - \theta}} S_{m-1}^{\frac{1 - 2 \theta}{1 - \theta}}$ where $q = R^{-\frac{1 - \theta}{\theta}}$. This implies that
    \begin{equation*}
        \big( q^{\frac{1 - 2 \theta}{1 - \theta}} - 1 \big) S_{m-1}^{\frac{1 - 2 \theta}{1 - \theta}} \le S_m^{\frac{1 - 2 \theta}{1 - \theta}} - S_{m-1}^{\frac{1 - 2 \theta}{1 - \theta}},
    \end{equation*}
    and the quantity on the left is clearly bounded away from zero because $q < 1$, $\frac{1 - 2 \theta}{1 - \theta} < 0$, and $S_{m-1} \to 0$. This shows that in either case, there exists a $\mu > 0$ such that
    \begin{equation*}
        \mu \le S_m^{\frac{1 - 2 \theta}{1 - \theta}} - S_{m-1}^{\frac{1 - 2 \theta}{1 - \theta}}.
    \end{equation*}
    Summing this inequality from $m = M_2$ to $m = M$, we obtain $(M - M_2) \mu \le S_M^{\frac{1 - 2 \theta}{1 - \theta}} - S_{M_2-1}^{\frac{1 - 2 \theta}{1 - \theta}}$, and because the function $x \mapsto x^{\frac{1 - \theta}{1 - 2 \theta}}$ is decreasing, this implies
    \begin{equation*}
        S_M \le \big( S_{M_2 - 1}^{\frac{1 - 2 \theta}{1 - \theta}} + (M - M_2) \mu \big)^{\frac{1 - \theta}{1 - 2 \theta}} \le d M^{\frac{1 - \theta}{1 - 2 \theta}},
    \end{equation*}
    for some constant $d$. By Jensen's inequality, we can say $\sum_{k = m}^\infty \EE \|z_k - z_{k-1}\| \le S_M \le d M_2^{- \frac{1 - \theta}{2 \theta - 1}}$. Using the fact that $\EE \|z_k - z^\star\| = \EE \|\sum_{k=m}^\infty z_k - z_{k-1}\| \le \EE \sum_{k=m}^\infty \|z_k - z_{k-1}\|$ proves Claim 1.

    If $\theta = 1/2$, then $\|\zeta_m\|^{\frac{1-\theta}{\theta}} = \|\zeta_m\|$. Equation \eqref{eq:dom} then gives
    \begin{equation}
    \begin{aligned}
    \label{eq:lin}
        & \sum\nolimits_{i=m}^\infty \sqrt{ \EE \|z_{k+1} - z_k\|^2} + \sqrt{\EE\|z_k - z_{k-1}\|^2} \\
        & \le \Big( 1 + a K_3 \big( p + \sqrt{\tfrac{V_1 + V_\Upsilon / \rho}{2 \bar{L} \lambda} } \big) \Big) \big( \sqrt{ \EE \|z_m - z_{m-1} \|^2 } + \sqrt{ \EE \|z_{m-1} - z_{m-2} \|^2} \big) \\
        & \qquad + \left( \tfrac{2 \sqrt{s}}{K_1 \rho} + a K_3 \sqrt{s} \right) \sqrt{ \EE \Upsilon_{m-1} } + a K_3 \sqrt{\tfrac{1}{2 \bar{L} \lambda \rho}} \sqrt{\EE \Upsilon_m},
    \end{aligned}
    \end{equation}
    where we have added the non-negative term $a K_3 \sqrt{\tfrac{V_1 + V_\Upsilon / \rho}{2 \bar{L} \lambda}} \sqrt{\EE \|z_{m-1} - z_{m-2} \|^2}$ to the right to simplify the presentation.
    Using equation \eqref{eq:sqrtgeo}, we have that, for any constant $c > 0$,
    \begin{equation*}
    0 \le - c \sqrt{ \EE \Upsilon_m } + c \big( 1 - \tfrac{\rho}{2} \big) \sqrt{ \EE \Upsilon_{m-1} } + c \sqrt{V_\Upsilon} ( \sqrt{ \EE \|z_m - z_{m-1}\|^2 } + \sqrt{ \EE \|z_{m-1} - z_{m-2}\|^2} ).
    \end{equation*}
    Combining this inequality with \eqref{eq:lin},
    \begin{equation*}
    \begin{aligned}
        & \sum\nolimits_{i=m}^\infty \sqrt{ \EE \|z_{k+1} - z_k\|^2} + \sqrt{\EE\|z_k - z_{k-1}\|^2} \\
        & \le \big( 1 + a K_3 \big( p + \sqrt{\tfrac{V_1 + V_\Upsilon / \rho}{2 \bar{L} \lambda}} \big) + c \sqrt{V_\Upsilon} \big) \big( \sqrt{ \EE \|z_m - z_{m-1} \|^2 } + \sqrt{ \EE \|z_{m-1} - z_{m-2} \|^2} \big) \\
        & \qquad + c \big( 1 - \tfrac{\rho}{2} + \tfrac{2 \sqrt{s}}{c K_1 \rho} + \tfrac{a K_3 \sqrt{s}}{c} \big) \sqrt{ \EE \Upsilon_{m-1} } - c \big(1 - a K_3 c^{-1} \sqrt{\tfrac{1}{2 \bar{L} \lambda \rho}} \big) \sqrt{\EE \Upsilon_m}.
    \end{aligned}
    \end{equation*}
    Defining
    \begin{equation*}
        T_m \defeq \sum\nolimits_{i=m}^\infty \sqrt{ \EE \|z_{i+1} - z_i\|^2 } + \sqrt{ \EE \|z_i - z_{i-1}\|^2 },
    \end{equation*}
    and $P_2 = 1 + a K_3 \big( p + 4 \sqrt{s V_\Upsilon} / \rho + \sqrt{ \frac{V_1 + V_\Upsilon / \rho}{2 \bar{L} \lambda} } \big) + c \sqrt{V_\Upsilon}$, we have shown
    \begin{equation*}
    \begin{aligned}
        & T_m + c \big(1 - a K_3 c^{-1} \sqrt{\tfrac{1}{2 \bar{L} \lambda \rho}} \big) \sqrt{\EE \Upsilon_m} \\
        & \le P_2 ( T_{m-1} - T_m ) + c \big( 1 - \tfrac{\rho}{2} + \tfrac{2 \sqrt{s}}{c K_1 \rho} + \tfrac{a K_3 \sqrt{s}}{c} \big) \sqrt{ \EE \Upsilon_{m-1} }.
    \end{aligned}
    \end{equation*}
    Rearranging,
    \begin{equation*}
        (1 + P_2) T_m + c \big(1 - a K_3 c^{-1} \sqrt{\tfrac{1}{2 \bar{L} \lambda \rho}} \big) \sqrt{\EE \Upsilon_m}
        \le P_2 T_{m-1} + c \big( 1 - \tfrac{\rho}{2} + \tfrac{2 \sqrt{s}}{c K_1 \rho} + \tfrac{a K_3 \sqrt{s}}{c} \big) \sqrt{ \EE \Upsilon_{m-1} }.
    \end{equation*}
    This implies
    \begin{equation*}
    \begin{aligned}
        &T_m + \sqrt{\EE \Upsilon_m} \\
        &\le \max\Big\{ \sfrac{P_2}{1 + P_2}, \big( 1 - \tfrac{\rho}{2} + \tfrac{2 \sqrt{s}}{c K_1 \rho} + \tfrac{a K_3 \sqrt{s}}{c} \big) \big(1 - a K_3 c^{-1} \sqrt{\tfrac{1}{2 \bar{L} \lambda \rho}} \big)^{-1} \Big\} (T_{m - 1} + \sqrt{\EE \Upsilon_{m-1}}).
    \end{aligned}
    \end{equation*}
    For large $c$, the second coefficient in the above expression approaches $1-\rho/2$. This proves the linear rate of Claim 2.

    When $\theta = 0$, the KL property \eqref{eq:kl} implies that exactly one of the following two scenarios holds: either $\EE \Phi(z_k) \not = \Phi_k^\star$ and
    \begin{equation}
    \label{eq:lobound}
        0 < C \le \EE \|\zeta_k\| \qquad \forall \zeta_k \in \EE \partial \Phi(z_k),
    \end{equation}
    or $\Phi(z_k) = \Phi_k^\star$. We show that the above inequality can only hold for a finite number of iterations.

    Using the subgradient bound, the first scenario implies
    \begin{equation*}
    \begin{aligned}
        C^2
        \le (\EE \|\zeta_k\|)^2
        & \le (p \EE \|z_k - z_{k-1}\| + p \EE \|z_{k-1} - z_{k-2}\| + \EE \Gamma_{k-1})^2, \\
        & \le 3 p^2 ( \EE \|z_k - z_{k-1}\| )^2 + 3 p^2 (\EE \|z_{k-1} - z_{k-2}\|)^2 + 3 (\EE \Gamma_{k-1})^2, \\
        & \le 3 p^2 \EE \|z_k - z_{k-1}\|^2 + 3 p^2 \EE \|z_{k-1} - z_{k-2}\|^2 + 3 s \EE \Upsilon_{k-1}.
    \end{aligned}
    \end{equation*}
    where we have used the inequality $(a_1+a_2+\cdots+a_s)^2 \le s (a_1^2 + \cdots + a_s^2)$ and Jensen's inequality. Applying this inequality to the decrease of $\Psi_k$ \eqref{eq:psidec}, we obtain
    \begin{equation*}
    \begin{aligned}
        \EE \Psi_k & \le \EE \Psi_{k-1} + \big( \tfrac{\bar{L}(\lambda+1)}{2} + \tfrac{V_1 + V_\Gamma / \rho}{2 \bar{L} \lambda} + Z - \tfrac{1}{2 \eta} \big) \EE \|z_k - z_{k-1}\|^2 - Z \EE \|z_{k-1} - z_{k-2}\|^2 \\
        & \le \EE \Psi_{k-1} - C^2 + \mathcal{O}(\EE \|z_k - z_{k-1}\|^2) + \mathcal{O}(\EE \|z_{k-1} - z_{k-2}\|^2) + \mathcal{O}(\EE \Upsilon_{k-1}),
        \end{aligned}
    \end{equation*}
    for some constant $C^2$.\footnote{We have ignored extraneous constants in the final three terms for clarity.} Because the final three terms go to zero as $k \to \infty$, there exists an index $M_3$ so that the sum of these three terms is bounded above by $C^2 / 2$ for all $k \ge M_3$. Therefore,
    \begin{equation*}
        \EE \Psi_k \le \EE \Psi_{k-1} - \tfrac{C^2}{2}, \qquad \forall k \ge M_3.
    \end{equation*}
    Because $\Psi_k$ is bounded below for all $k$, this inequality can only hold for $N < \infty$ steps. After $N$ steps, it is no longer possible for the bound \eqref{eq:lobound} to hold, so it must be that $\Phi(z_k) = \Phi_k^\star$. Because $\Phi_k^\star < \Phi(z^\star)$, $\Phi_k^\star < \EE \Phi(z_k)$, and both $\EE \Phi(z_k), \Phi_k^\star$ converge to $\EE \Phi(z^\star)$, we must have $\Phi_k^\star = \EE \Phi(z_k) = \EE \Phi(z^\star)$.
\end{proof}

The main difference between these convergence rates and those of PALM occurs when $\theta \in (0,1/2]$. In this case, the linear convergence rate cannot be faster than the geometric decay of the MSE of the gradient estimator, which is of order $(1 - \rho)^k$ after $k$ iterations. Without mini-batching (i.e. $b = 1$), this rate is approximately $(1-1/n)^k$ for SAGA estimator and $(1 - 1/p)^k$ for SARAH estimator.

\section{Numerical Experiments}
\label{sec:ex}

To demonstrate the advantages of SPRING, we compare SPRING using SAGA and SARAH gradient estimators to PALM \cite{bolte2014proximal} and inertial PALM \cite{pock2016inertial}. We also present results for SPRING using the (non-variance-reduced) SGD estimator (a case studied by Xu and Yin \cite{xu2015}). We refer to SPRING using the SGD, SAGA, and SARAH gradient estimators as SPRING-SGD, SPRING-SAGA, and SPRING-SARAH, respectively.
Three applications are considered for comparison: sparse non-negative matrix factorization (Sparse-NMF), sparse principal component analysis (Sparse-PCA), and blind image-deblurring (BID).

{\bf Sparse-NMF:} Given a data-matrix $A$, we seek a factorization $A \approx XY$ where $X\in \mathbb{R}^{n \times r}, Y \in \mathbb{R}^{r \times d}$ are non-negative with $r \le d$ and $X$ sparse. Sparse-NMF has the following formulation:
\begin{equation}\label{S_NMF0}
\min_{X, Y}  \|A - XY\|_F^2, \qquad \mathrm{s.t.}\ \  X, Y \ge 0,\ \  \|X_i\|_0 \leq s,\ i = 1,..., r.
\end{equation}
Here, $X_i$ denotes the $i$'th column of $X$. In dictionary learning and sparse coding, $X$ is called the learned dictionary with coefficients $Y$. In this formulation, the sparsity on $X$ is strictly enforced using the non-convex $\ell_0$ constraint, but one can also use $\ell_1$ regularization to preserve convexity.

{\bf Sparse-PCA:} The problem of Sparse-PCA with $r$ principal components can be written as:
\begin{equation}\label{S_PCA}
	\min_{X, Y}\|A - XY\|_F^2 + \lambda_1 \|X\|_1 + \lambda_2 \|Y\|_1,
\end{equation}
where $ X\in \mathbb{R}^{n \times r}, Y \in \mathbb{R}^{r \times d}$. We use $\ell_1$ regularization on both $X$ and $Y$ to promote sparsity.

{\bf Blind Image-Deblurring:} Let $Z$ be a blurred image. The problem of blind deconvolution reads:
\begin{equation}
 \min_{X, Y} \|Z - X \odot Y\|_F^2 + \lambda \sum\nolimits_{r = 1}^{2d}\Phi([D(X)]_r) \quad \mathrm{s.t.}\quad  0 \le X \le 1,\ 0 \le Y \le 1,\ \|Y\|_1 \leq 1,
\end{equation}
where $\odot$ is the 2D convolution operator, $X$ is the image to recover, and $Y$ is the blur-kernel to estimate. We choose a classic smooth edge-preserving regularizer in the image domain, with $D(\cdot)$ being the 2D differential operator computing the horizontal and vertical gradients for each pixel. For the potential function $\Phi(\cdot)$, we choose $\Phi(v) := \log(1 + \theta v^2)$ as in \cite{pock2016inertial}. This potential function encourages sparsity in the image gradients and hence promotes sharp images. We choose $\theta = 10^3$ in our experiments.

One of the benefits of SPRING and PALM is that the two step-sizes, $\gamma_{X,k}$ and $\gamma_{Y,k}$, depend separately on the Lipschitz constants $\hat{L}_X(Y_k)$ and $\hat{L}_Y(X_k)$. The practical performance of these algorithms depends significantly on the step-size choices. The following section describes how we use adaptive step-sizes in our experiments.

\subsection{Parameter choices and on-the-fly estimation of Lipschitz constants}

The global Lipschitz constants of the partial gradients of $F$ are usually unknown and difficult to estimate. In practice, adaptive step-size choices based on estimating local Lipschitz constants are needed for PALM and inertial PALM \cite{pock2016inertial}. In our experiments, we use the power method to estimate the Lipschitz constants on-the-fly in every iteration of the compared algorithms. For SPRING-SGD, SPRING-SAGA, and SPRING-SARAH, we find that it is sufficient to randomly sub-sample a mini-batch and run 5 iterations of the power method to get an estimate of the Lipschitz constants of the stochastic gradients. For PALM, we run 5 iterations of the power method in each iteration on the full batch to get an estimate of the Lipschitz constants of the full partial gradients.

For example, consider estimating the Lipschitz constants of the gradients corresponding to the objective function of Sparse-NMF (\ref{S_NMF0}). Let $X_k$ and $Y_k$ be the updates of $k$-th iteration, then $L_Y(X_k) = \|X_k\|^2$, which is the largest squared singular value of $X_k$, and can be computed via power iteration:
\begin{equation*}
    v_i = \sfrac{X_k^T(X_kv_{i-1})}{\|X_k^T(X_kv_{i-1})\|_2},
\end{equation*}
with a random initialization $\|v_0\|_2 = 1$. We find that using $5$ iterations is sufficient to provide good estimates, so we approximate $L_Y(X_k)$ by $\| X_k^T(X_k v_5)\|_2$. We use the same strategy for $L_X(Y_k)$.

Denote the estimated Lipschitz constants of the full gradients as $\hat{L}_X(Y_k)$ and $\hat{L}_Y(X_k)$, and denote the estimated Lipschitz constants of the stochastic estimates as $\tilde{L}_X(Y_k)$ and $\tilde{L}_Y(X_k)$. We set the step-sizes of the compared algorithms as follows:

\begin{itemize}
    \item {\bf PALM:} $\gamma_{X,k} = \frac{1}{\hat{L}_X(Y_k)}$ and $\gamma_{Y,k} = \frac{1}{\hat{L}_Y(X_k)}$ (these are the standard step-sizes \cite{bolte2014proximal}).
    \item {\bf Inertial PALM:} $\gamma_{X,k} = \frac{0.9}{\hat{L}_X(Y_k)}$, $\gamma_{Y,k} = \frac{0.9}{\hat{L}_Y(X_k)}$, and we set the momentum parameter to $\frac{k - 1}{k + 2}$, where $k$ denotes the number of iterations. Pock and Sabach \cite{pock2016inertial} assert that this dynamic momentum parameter achieves the best practical performance.\footnote{The dynamic choice of momentum parameter is not theoretically analyzed by Pock and Sabach \cite{pock2016inertial}, but it appears to be superior to the constant inertial parameter choice. Pock and Sabach suggest the aggressive step-sizes $\gamma_{X,k} = \frac{1}{\hat{L}_X(Y_k)}$ and $\gamma_{Y,k} = \frac{1}{\hat{L}_Y(X_k)}$ for the dynamic scheme, but we find these choices sometimes lead to unstable/divergent behavior in the late iterations. Hence, we use the slightly smaller step-sizes $\gamma_{X,k} = \frac{0.9}{\hat{L}_X(Y_k)}$ and $\gamma_{Y,k} = \frac{0.9}{\hat{L}_Y(X_k)}$ instead. These choices ensure the algorithm is stable, and we observe that they do not compromise the convergence rate in practice.}
    \item {\bf SPRING-SGD:} $\gamma_{X,k} = \frac{1}{\sqrt{\lceil k b/n \rceil}\tilde{L}_X(Y_k)}$ and $\gamma_{Y,k} = \frac{1}{\sqrt{\lceil kb/n \rceil}\tilde{L}_Y(X_k)}$. It is well-known in the literature that a shrinking step-size is necessary for SGD to converge to a critical point \cite{bottou2010large,konevcny2015mini,moulines2011non,xu2015}.
    \item {\bf SPRING-SAGA:} $\gamma_{X,k} = \frac{1}{3\tilde{L}_X(Y_k)}$ and $\gamma_{Y,k} = \frac{1}{3\tilde{L}_Y(X_k)}$.
    \item {\bf SPRING-SARAH:} $\gamma_{X,k} = \frac{1}{2\tilde{L}_X(Y_k)}$ and $\gamma_{Y,k} = \frac{1}{2\tilde{L}_y(X_k)}$.
\end{itemize}

\begin{remark}[Practical step-sizes for SPRING-SAGA and SPRING-SARAH]
While the step-sizes suggested in Sections \ref{sec:conv} and \ref{sec:kl} lead to state-of-the-art convergence rates for \eqref{eq:J-F-R}, we observe that those step-size choices are conservative for SPRING-SAGA and SPRING-SARAH in practice. Hence, we adopt the suggested step-size choices in the original works with scale factors $1/3$ for SAGA \cite[Section 2]{SAGA} and $1/2$ for SARAH \cite[Corollary 3]{sarah}. For all tested methods, the step-sizes we use are optimal in practice while ensuring convergence in all experiments with extensive tests.
\end{remark}

The same random initialization is used for all of the compared algorithms in our Sparse-NMF and Sparse-PCA experiments, while for BID we initialize the image estimate with the blurred image and the kernel estimate with all ones. We observe that SPRING with variance-reduced gradients can be sensitive to poor initialization, and this may initially compromise convergence. However, this initialization issue can be effectively resolved if we use plain stochastic gradient without variance-reduction in the first epoch of SPRING-SARAH/SPRING-SAGA as a warm-start, which is suggested in \cite{konevcny2017semi}.

\subsection{Sparse-NMF}

We consider the extended Yale-B dataset and the ORL dataset, which are standard facial recognition benchmarks consisting of human face images.\footnote{Preprocessed versions \cite{cai2007spectral,cai2007learning} can be found in: http://www.cad.zju.edu.cn/home/dengcai/Data/FaceData.html} The ORL datasets contain 400 images of size $64 \times 64$, and the extended Yale-B dataset contains 2414 cropped images of size $32 \times 32$. In this experiment, we extract 49 sparse basis-images for both datasets.
In each iteration of the stochastic algorithms, we randomly sub-sample $2.5\%$ of the full batch as a mini-batch.

\begin{figure}[!ht]
	\centering
	\subfloat[{\tt ORL dataset}]{ \includegraphics[width=0.425\linewidth]{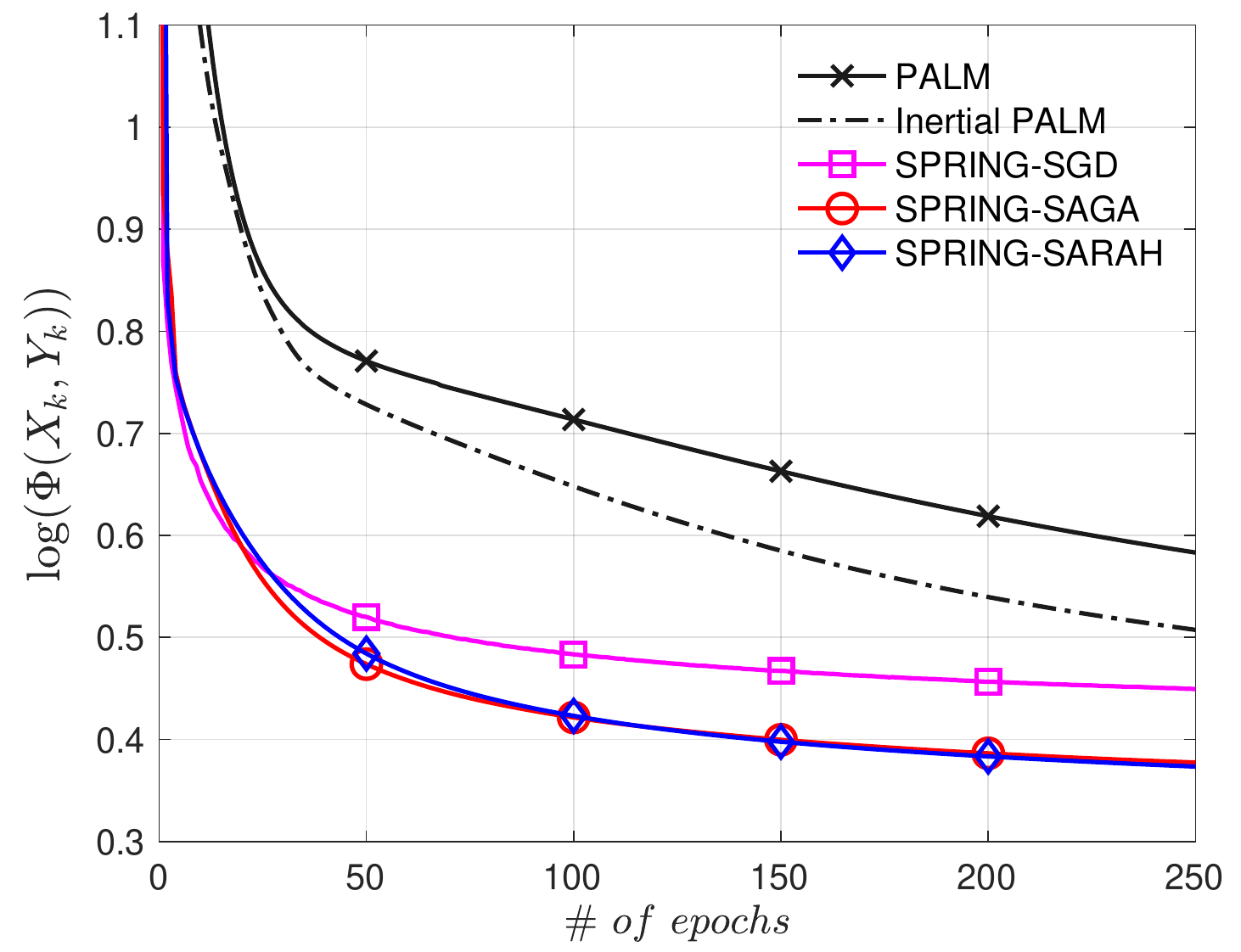} } 	\hspace{3mm}
	\subfloat[{\tt Yale dataset}]{ \includegraphics[width=0.425\linewidth]{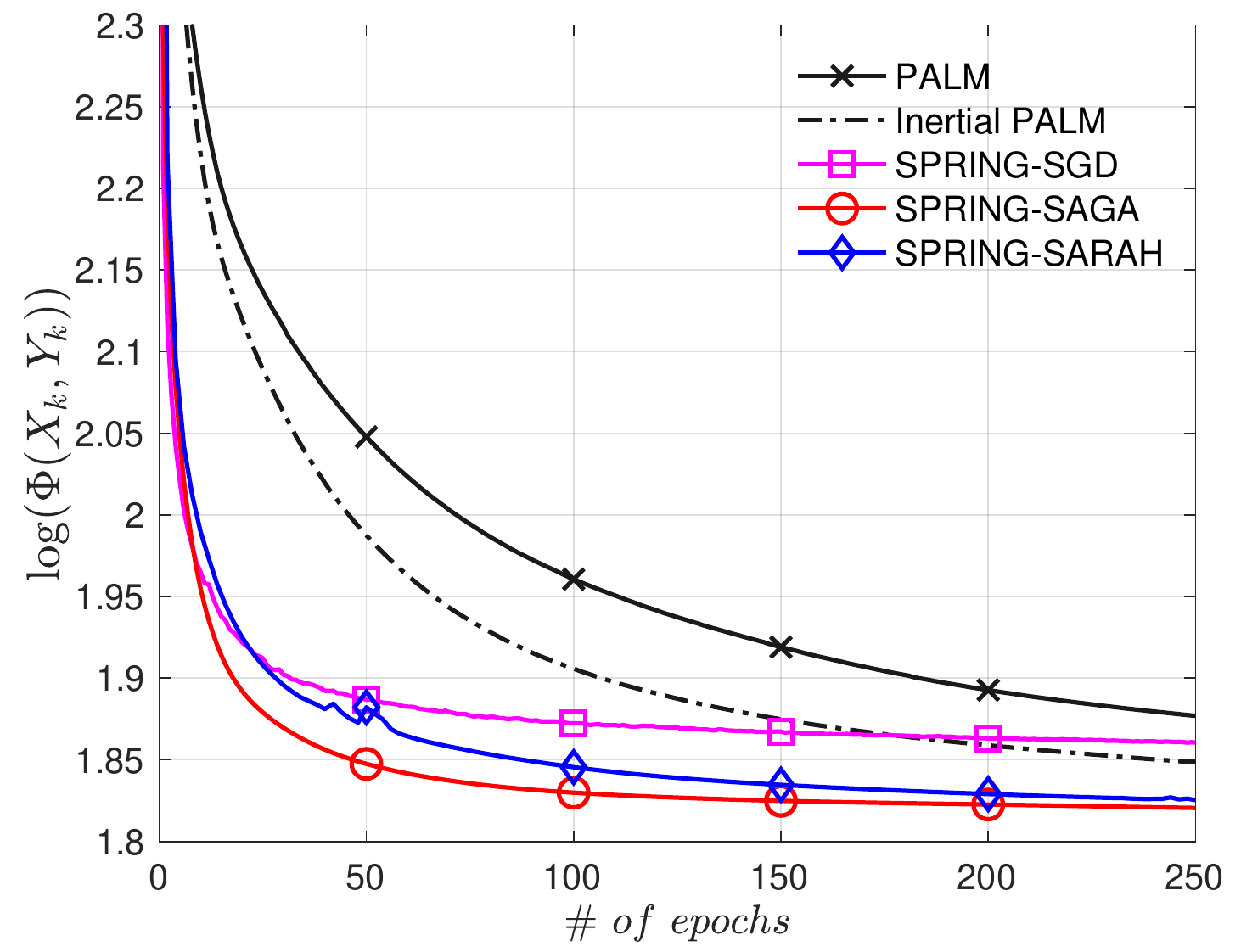} } \\[-2mm]
	\caption{Objective decrease comparison of Sparse-NMF: ORL dataset (left) and Yale dataset (right).}
	\label{snmf_plot}
\end{figure}

The obtained results are shown in Figure \ref{snmf_plot}, from which we observe:

\begin{itemize}
\item Overall, SPRING using SAGA and SARAH estimators achieves superior performance compared to PALM, inertial PALM, and SPRING using the vanilla SGD gradient estimator.
\item PALM has the worst performance in the considered Sparse-NMF tasks, which is not surprising since PALM is the baseline method in this comparison. Incorporating inertia can offer considerable acceleration for PALM.
\item SPRING using the vanilla SGD gradient estimator achieves fast convergence initially, but gradually slows its convergence due to the shrinking step-size that is necessary to combat the non-reducing variance.
However, using variance-reduced gradient estimators SAGA and SARAH, SPRING is able to overcome this issue and achieve the best overall convergence rates.
\end{itemize}

Figure \ref{snmf_orl_ima} presents the basis images generated by SPRING-SAGA and PALM for the ORL dataset (we also present the basis images for the Yale dataset in the supplemental material). It is clear that the basis images generated by SPRING-SAGA appear natural and smooth, while PALM's results at the same epoch appear noisy and distorted.

\begin{figure}[!ht]
	\centering
	\subfloat[{SPRING-SAGA}]{ \includegraphics[width=0.425\linewidth]{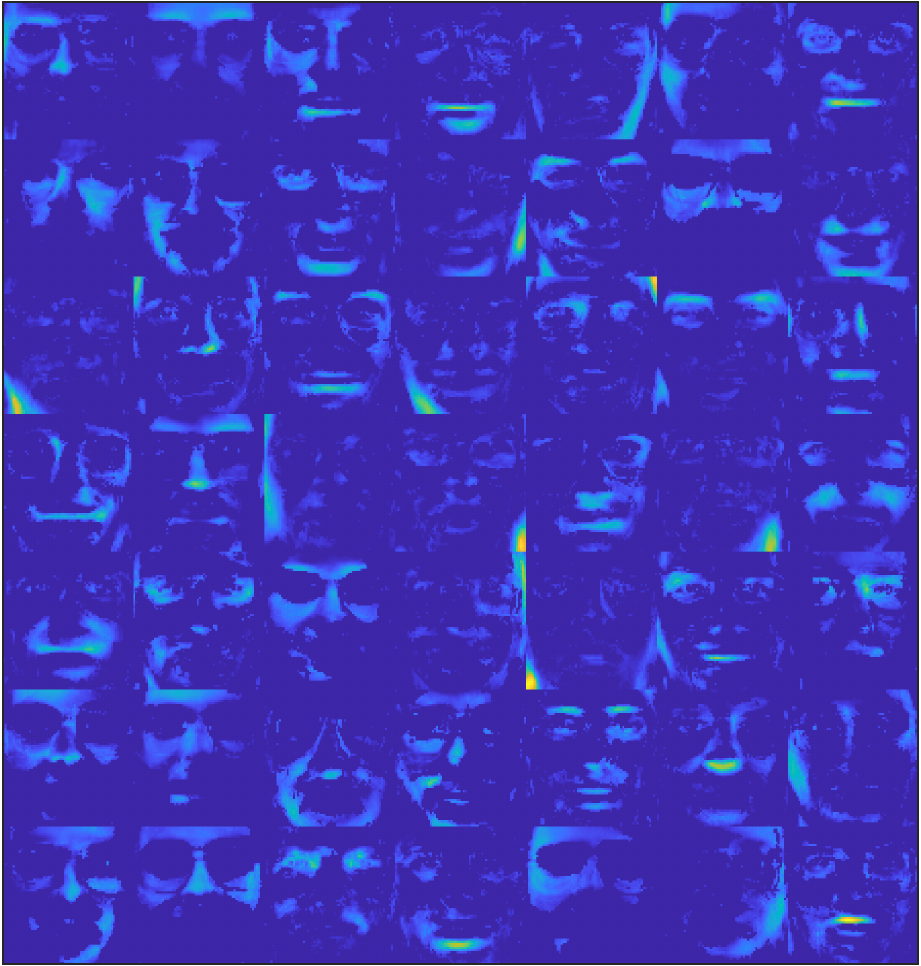} } \hspace{3mm}
	\subfloat[{PALM}]{ \includegraphics[width=0.425\linewidth]{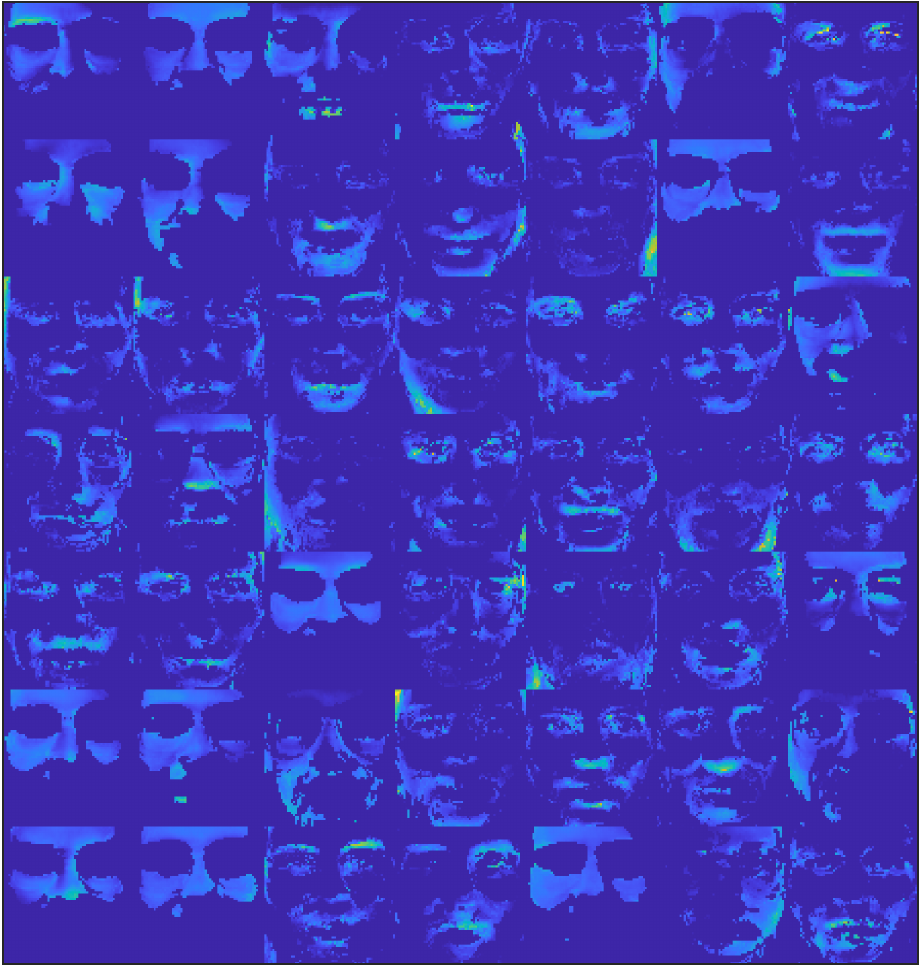} } \\[-2mm]
	\caption{Basis images from the Sparse-NMF experiment generated by SPRING-SAGA and PALM on the $250^{th}$ epoch for the ORL dataset.}
	\label{snmf_orl_ima}
	\vspace{-2ex}
\end{figure}

\subsection{Sparse-PCA}

For our Sparse-PCA experiments, we compare SPRING-SAGA, SPRING-SARAH, SPRING-SGD and PALM. Similar to what we observe in the Sparse-NMF experiments, our results in Figure \ref{spca_plot} show that SPRING with stochastic variance-reduced gradient estimators achieves the fastest convergence. We also observe that inertia provides significant acceleration to PALM in both the Sparse-NMF and Sparse-PCA tasks. We believe that such inertial schemes can also be extended to accelerate SPRING and leave it as an important direction of future research (see \cite{ispring} for some work in this direction).

\begin{figure}[!ht]
	\centering
	\subfloat[{\tt ORL dataset}]{ \includegraphics[width=0.425\linewidth]{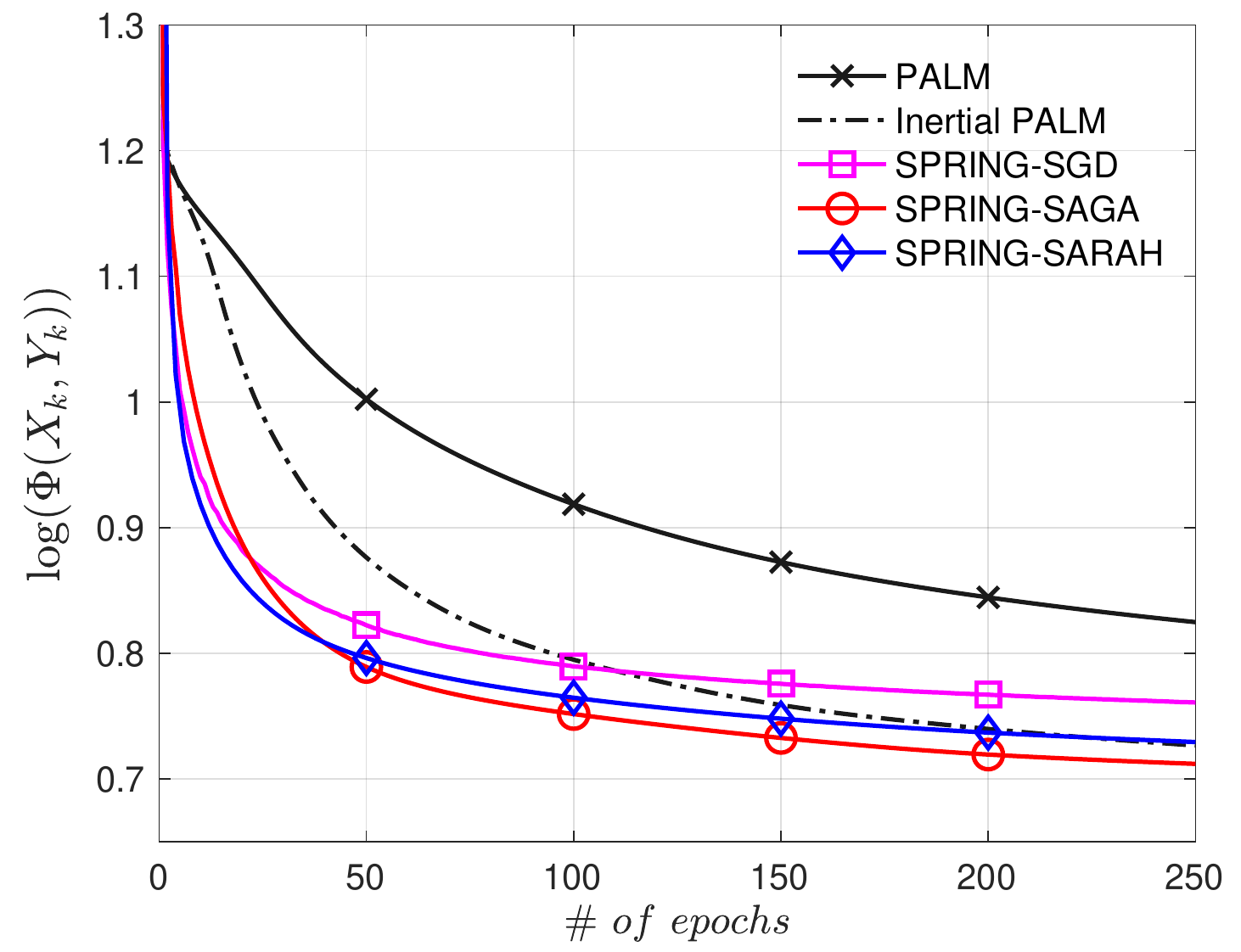} } 	\hspace{3mm}
	\subfloat[{\tt Yale dataset}]{ \includegraphics[width=0.425\linewidth]{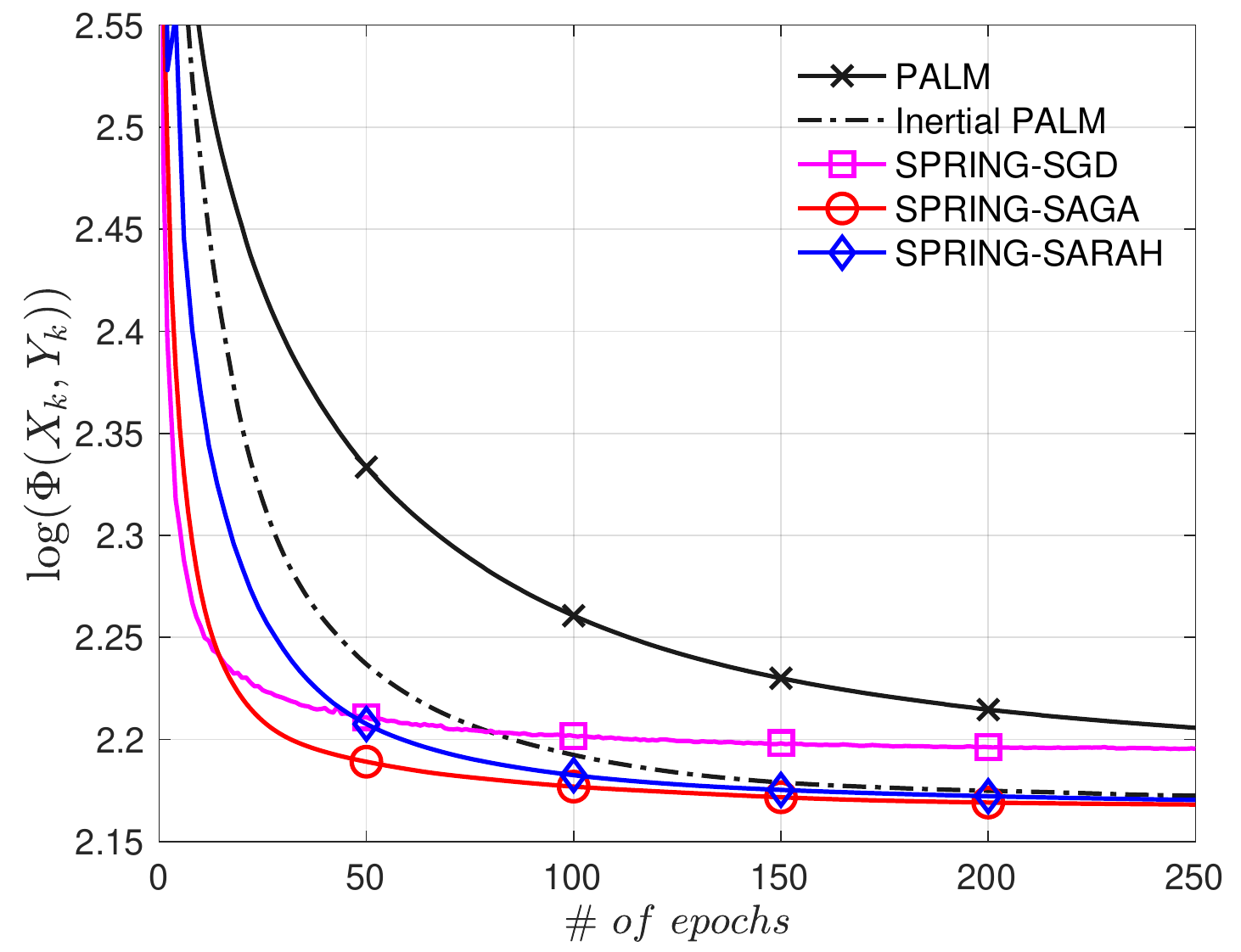} } \\[-2mm]
	\caption{Objective decrease comparison of Sparse-PCA:  ORL dataset (left) and Yale dataset (right).}
	\label{spca_plot}
	\vspace{-3ex}
\end{figure}

\begin{figure}[!ht]
	\centering
	\subfloat[{\tt Kodim08}]{ \includegraphics[width=0.425\linewidth]{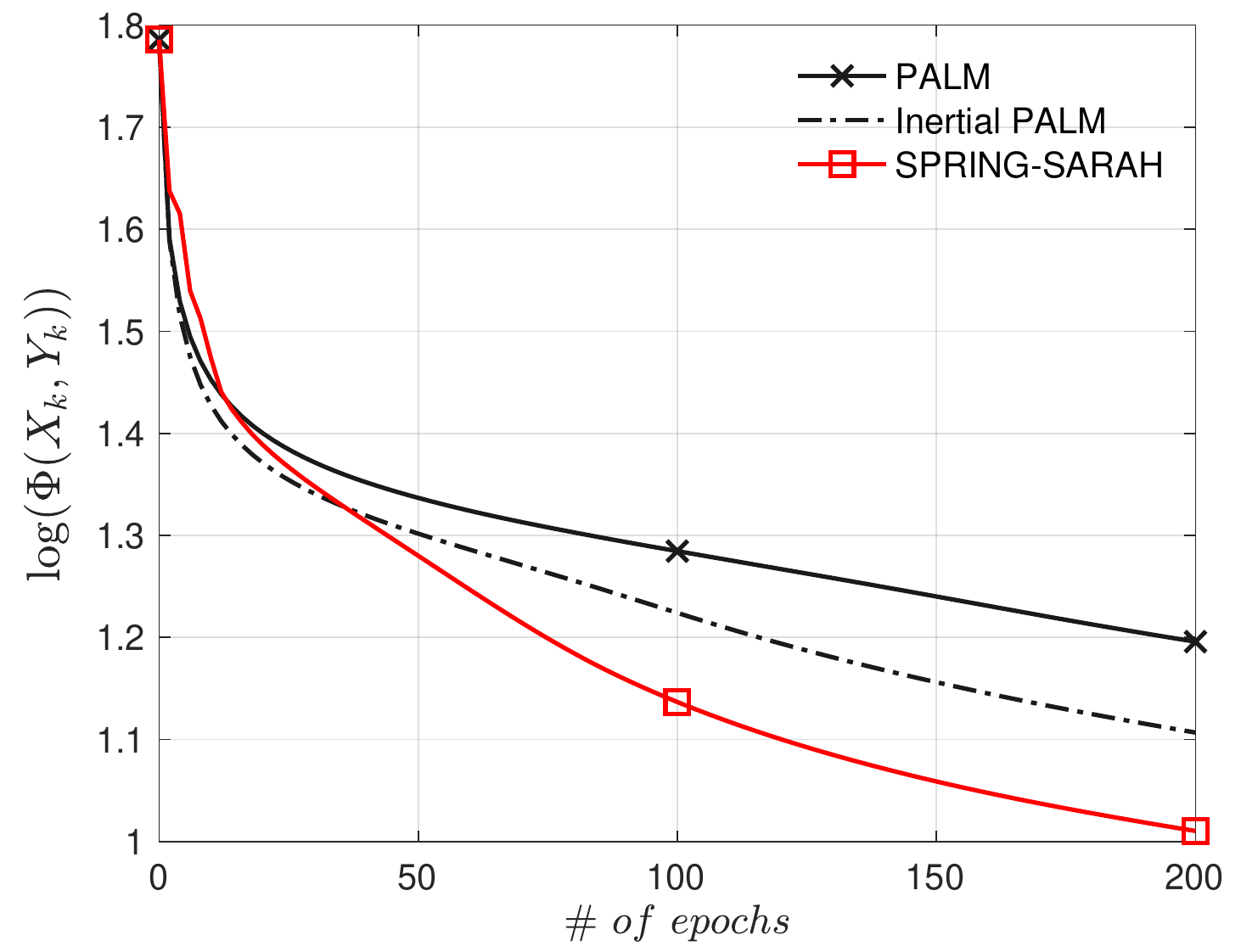} } 	\hspace{3mm}
	\subfloat[{\tt Kodim15}]{ \includegraphics[width=0.425\linewidth]{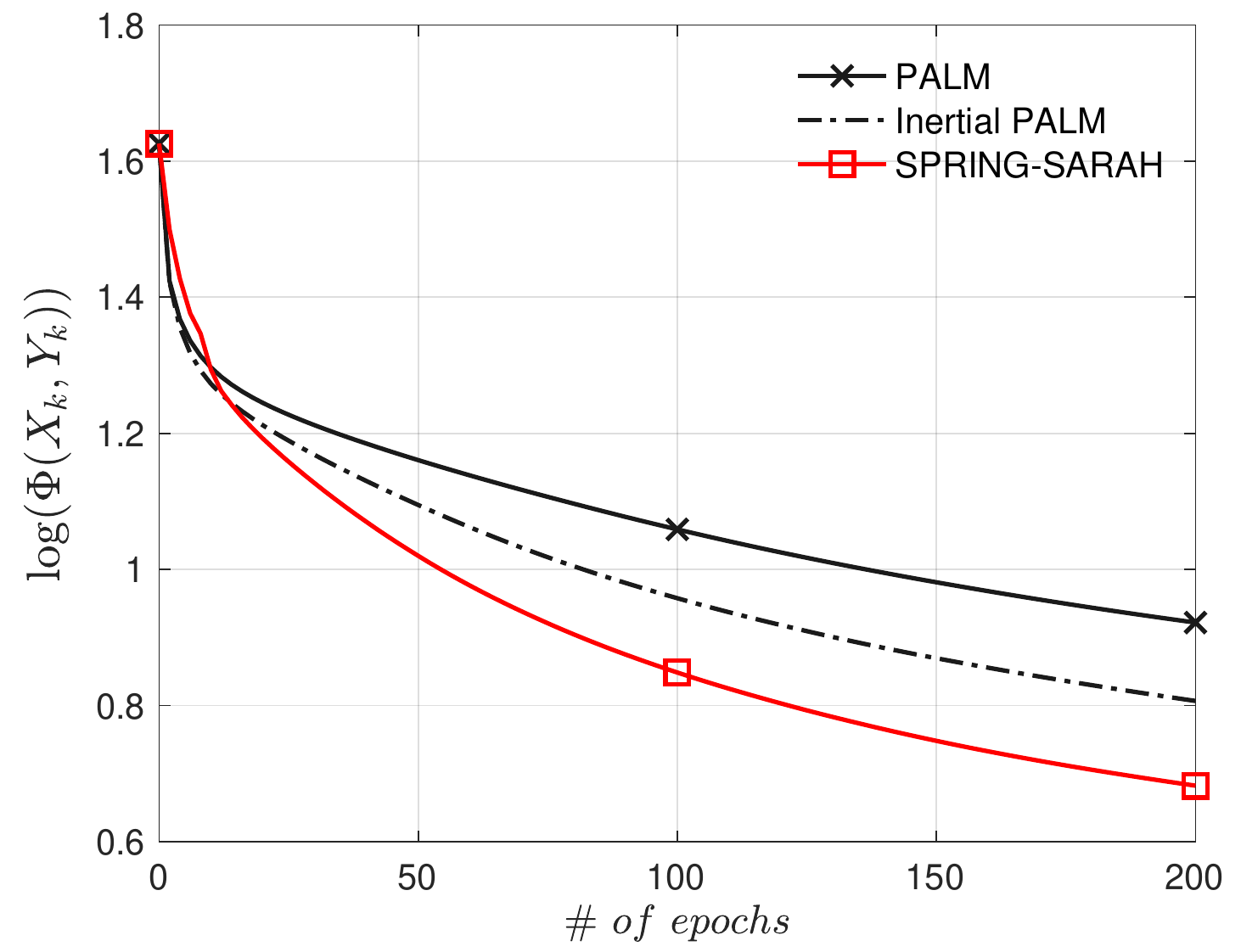} } \\[-2mm]
	\caption{Objective decrease comparison of blind image-deconvolution experiment on Kodim08 (left), and Kodim15 (right) images using an $11 \times 11$ motion-blur kernel.}
	\label{bid_m}
\end{figure}

\subsection{Blind Image-Deblurring}

For blind image-deconvolution, we choose to compare SPRING-SARAH, PALM and inertial PALM. We use two images, {\it Kodim08} and {\it Kodim15}, of size $256 \times 256$ for testing. For each image, two blur kernels---linear motion blur and out-of-focus blur are considered with additional additive Gaussian noise.
For SPRING, the mini-batch size is $1/16$ of the full batch.

For both images with motion blur, the convergence comparisons of the algorithms are provided in Figure \ref{bid_m}, from which we observe SPRING-SARAH is faster than the other two methods in both cases. Figures \ref{kodak08_motion} and \ref{kodak15_motion} provide comparisons of the recovered image and blur kernel. We observe superior performance of SPRING-SARAH over PALM in these figures as well. In particular, we compare the estimated blur kernel of the two algorithms at every $20^{th}$ epoch, and find out that SPRING-SARAH is also faster than PALM.
It is worth noting that, although stochastic gradient methods have been shown to be inherently inefficient for non-blind and non-uniform deblurring task where the blur kernels are known or estimated beforehand \cite{tang2019}, SPRING still offers significant acceleration over PALM in blind-deblurring tasks.
Additional experiments using motion blur kernels are provided in the appendix.

\begin{figure}[!ht]
	\centering
	\subfloat[Original image and kernel]{ \includegraphics[width=0.225\linewidth]{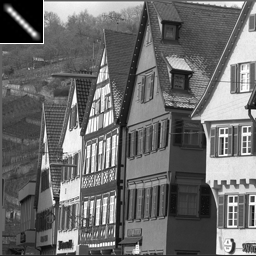} } 	\hspace{1mm}
	\subfloat[Blurred image]{ \includegraphics[width=0.225\linewidth]{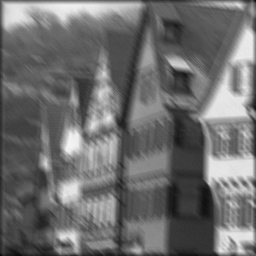} } \hspace{1mm}
	\subfloat[Recovered by PALM]{ \includegraphics[width=0.225\linewidth]{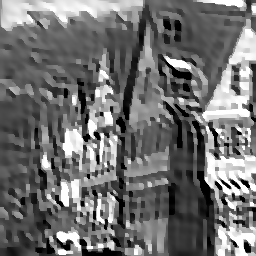} } 	\hspace{1mm}
	\subfloat[Recovered by SPRING]{ \includegraphics[width=0.225\linewidth]{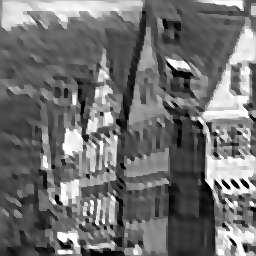} } \\[-2mm]
	\subfloat[Estimated kernel by PALM]{ \includegraphics[width=0.475\linewidth]{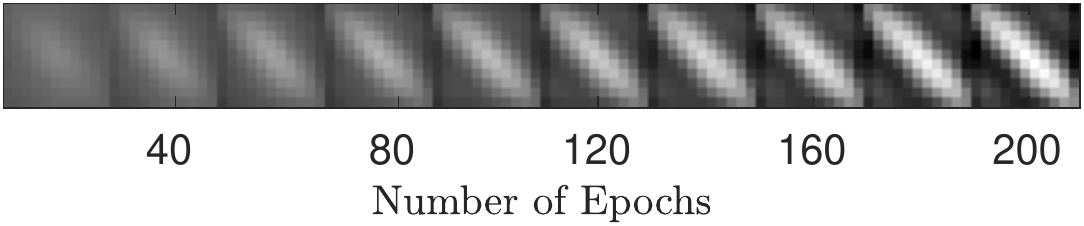} } 	\hspace{1mm}
	\subfloat[Estimated kernel by SPRING]{ \includegraphics[width=0.475\linewidth]{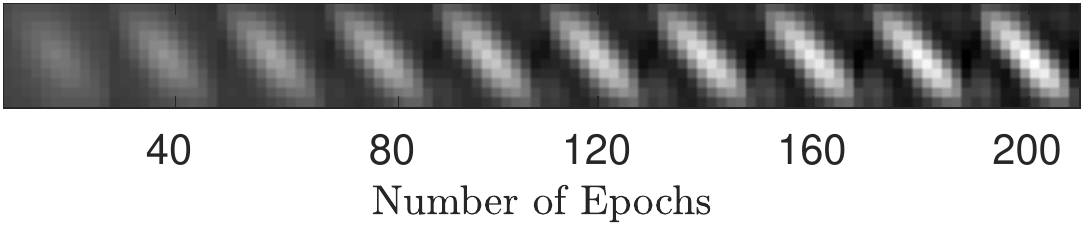} } \\[-2mm]
	\caption{Image and kernel reconstructions from the blind image-deconvolution experiment on the Kodim08 image using an $11 \times 11$ motion blur kernel.}
	\label{kodak08_motion}
	\vspace{-3ex}
\end{figure}

\begin{figure}[!ht]
	\centering
	\subfloat[Original image and kernel]{ \includegraphics[width=0.225\linewidth]{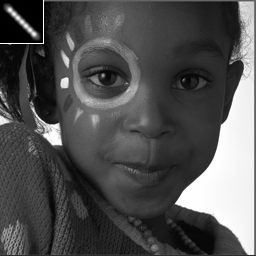} } 	\hspace{1mm}
	\subfloat[Blurred image]{ \includegraphics[width=0.225\linewidth]{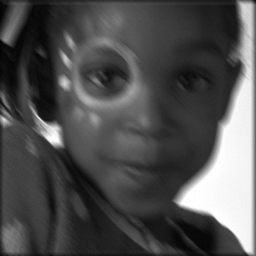} } \hspace{1mm}
	\subfloat[Recovered by PALM]{ \includegraphics[width=0.225\linewidth]{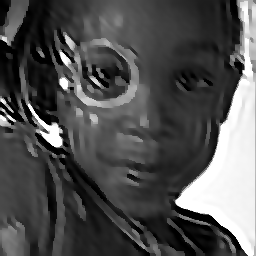} } 	\hspace{1mm}
	\subfloat[Recovered by SPRING]{ \includegraphics[width=0.225\linewidth]{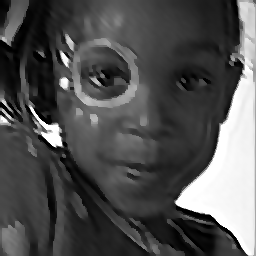} } \\[-2mm]
	\subfloat[Estimated kernel by PALM]{ \includegraphics[width=0.475\linewidth]{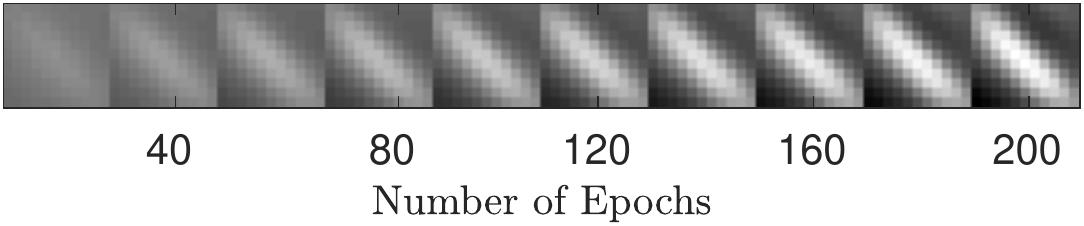} } 	\hspace{1mm}
	\subfloat[Estimated kernel by SPRING]{ \includegraphics[width=0.475\linewidth]{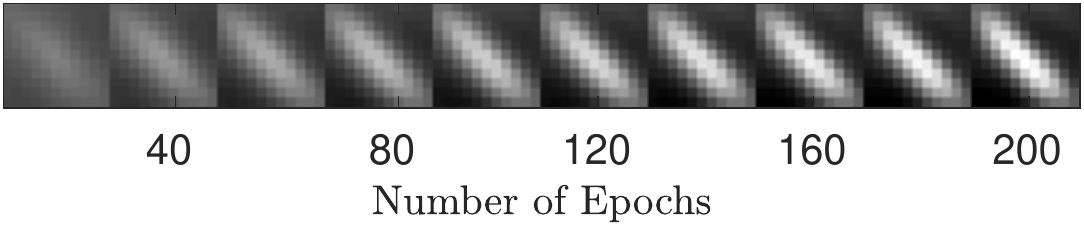} } \\[-2mm]
	\caption{Image and kernel reconstructions from the blind image-deconvolution experiment on the Kodim15 image using an $11 \times 11$ motion blur kernel}.
	\label{kodak15_motion}
	\vspace{-3ex}
\end{figure}

\section{Conclusion}

We propose stochastic extensions of the PALM algorithm of for solving a class of structured non-smooth and non-convex optimization problems. We analyse the convergence properties of our stochastic PALM with two typical variance-reduced stochastic gradient estimators, SAGA and SARAH. For generic optimization problems of the form \eqref{eq:J-F-R}, we show that SPRING-SAGA (with $b \le \mathcal{O}(n^{2/3})$) and SPRING-SARAH return an $\epsilon$-approximate critical point in expectation in no more than $O(\frac{n^2 L}{b^3 \epsilon^2})$ and $O(\frac{\sqrt{n} L}{\epsilon^2})$ SFO calls, respectively, showing that SPRING-SARAH achieves the complexity lower bound for stochastic non-convex optimization.
For objectives satisfying an error bound, we further demonstrate that {our methods converge linearly to the global optimum}. Because of the generality of our results, they contain almost all existing results for stochastic non-convex optimization as special cases, and they improve on them in many settings.

\section*{Acknowledgements}
JT and MD acknowledge support from the ERC Advanced grant, project 694888, C-SENSE.
JL acknowledges support from the Leverhulme Trust.
CBS acknowledges support from the Leverhulme Trust project on Breaking the Non-Convexity Barrier, and on Unveiling the Invisible, the Philip Leverhulme Prize, the EPSRC grant No. EP/S026045/1, EPSRC grant No. EP/M00483X/1, and EPSRC Centre No. EP/N014588/1, the European Union Horizon 2020 research and innovation programmes under the Marie Skłodowska-Curie grant agreement No. 691070 CHiPS and the Marie Skłodowska-Curie grant agreement No 777826, the Cantab Capital Institute for the Mathematics of Information, and the Alan Turing Institute.

\clearpage

\begin{small}
\bibliography{spring.bib}
\bibliographystyle{acm}

\clearpage

\appendix

\section{Additional numerical experiments}\label{app-C}

\begin{figure}[h!]
	\centering
	\subfloat[{SPRING-SAGA}]{ \includegraphics[width=0.425\linewidth]{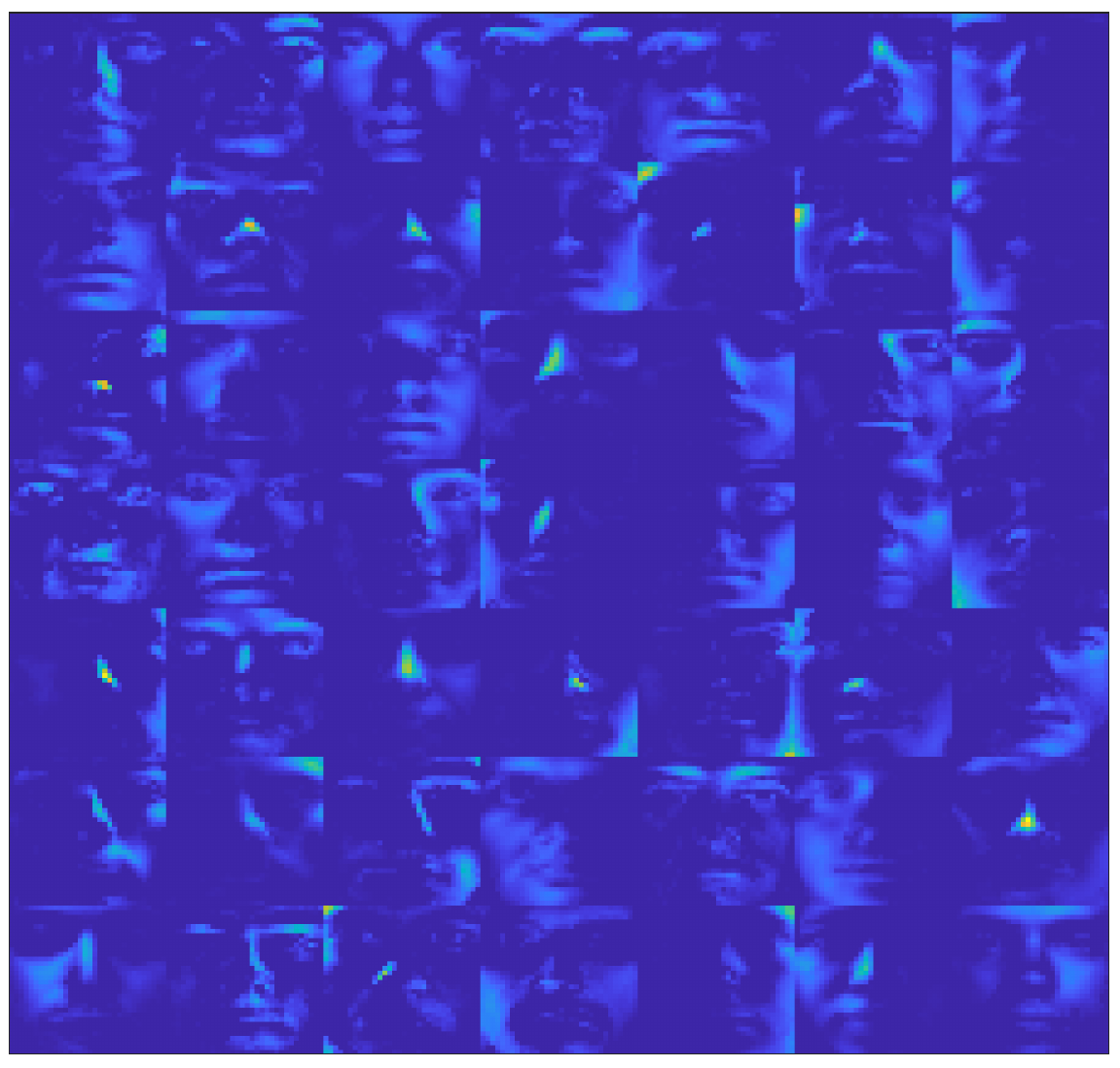} } \hspace{3mm}
	\subfloat[{PALM}]{ \includegraphics[width=0.425\linewidth]{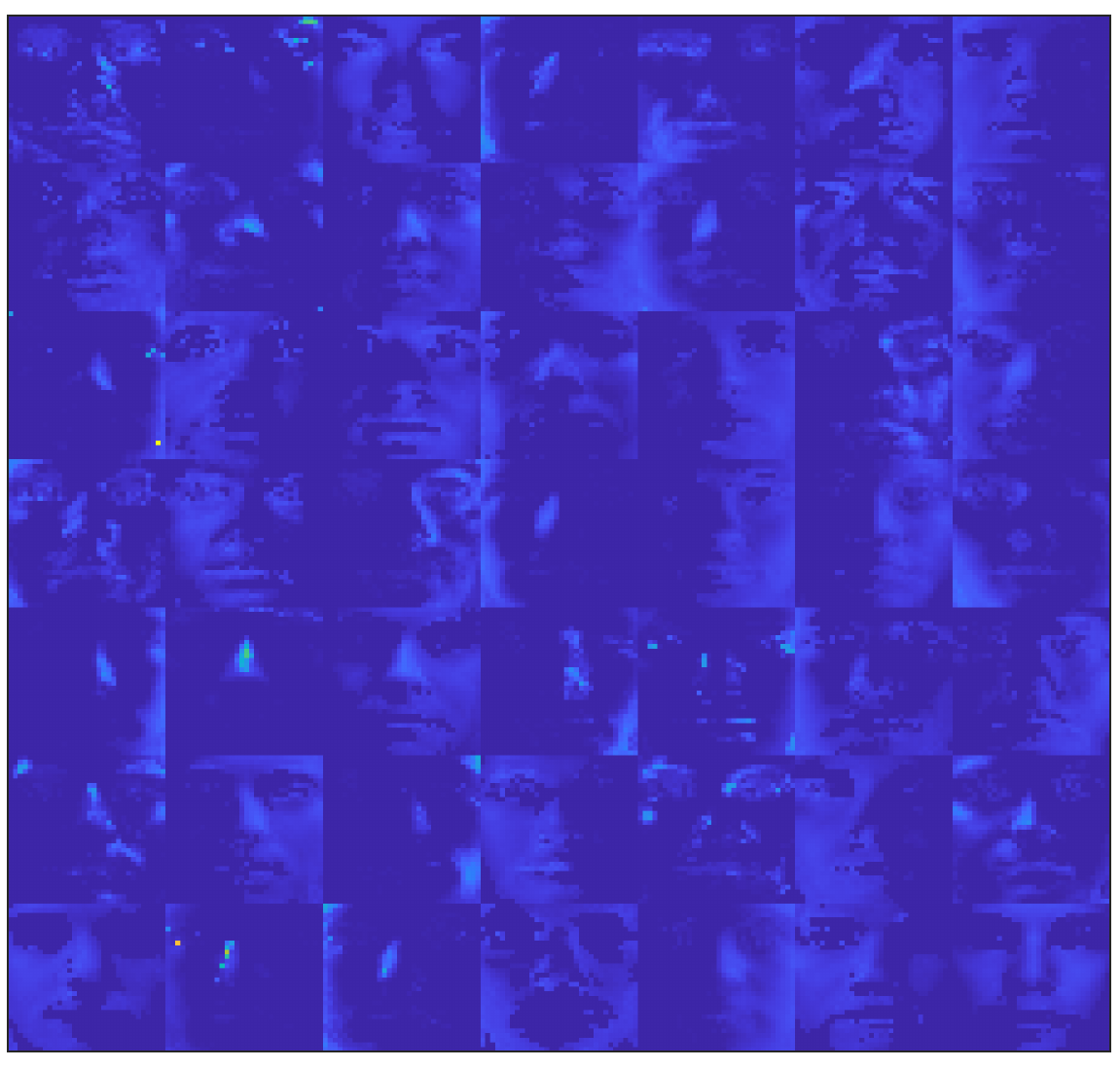} } \\[-2mm]
	\caption{Basis images from the Sparse-NMF experiment generated by SPRING-SAGA and PALM on the $10^{th}$ epoch for the Yale dataset.}
	\label{snmf_yale_ima}
\end{figure}

This section contains additional numerical experiments demonstrating the superiority of SPRING over PALM. Figure \ref{snmf_yale_ima} displays the results of our Sparse-NMF experiment on the Yale dataset. As with the ORL dataset, we clearly observe that SPRING-SAGA converges to clean basis images faster than PALM.

Figures \ref{bid_o} and \ref{kodak08_outfocus} show additional comparisons for blind image-deblurring where the images are blurred with an out-of-focus kernel. The settings here are the same for the BID experiments presented in the main text. Again, our SPRING-SARAH algorithm outperforms PALM and inertial-PALM.

\begin{figure}[h!]
	\centering
	\subfloat[{\tt Kodim08}]{ \includegraphics[width=0.425\linewidth]{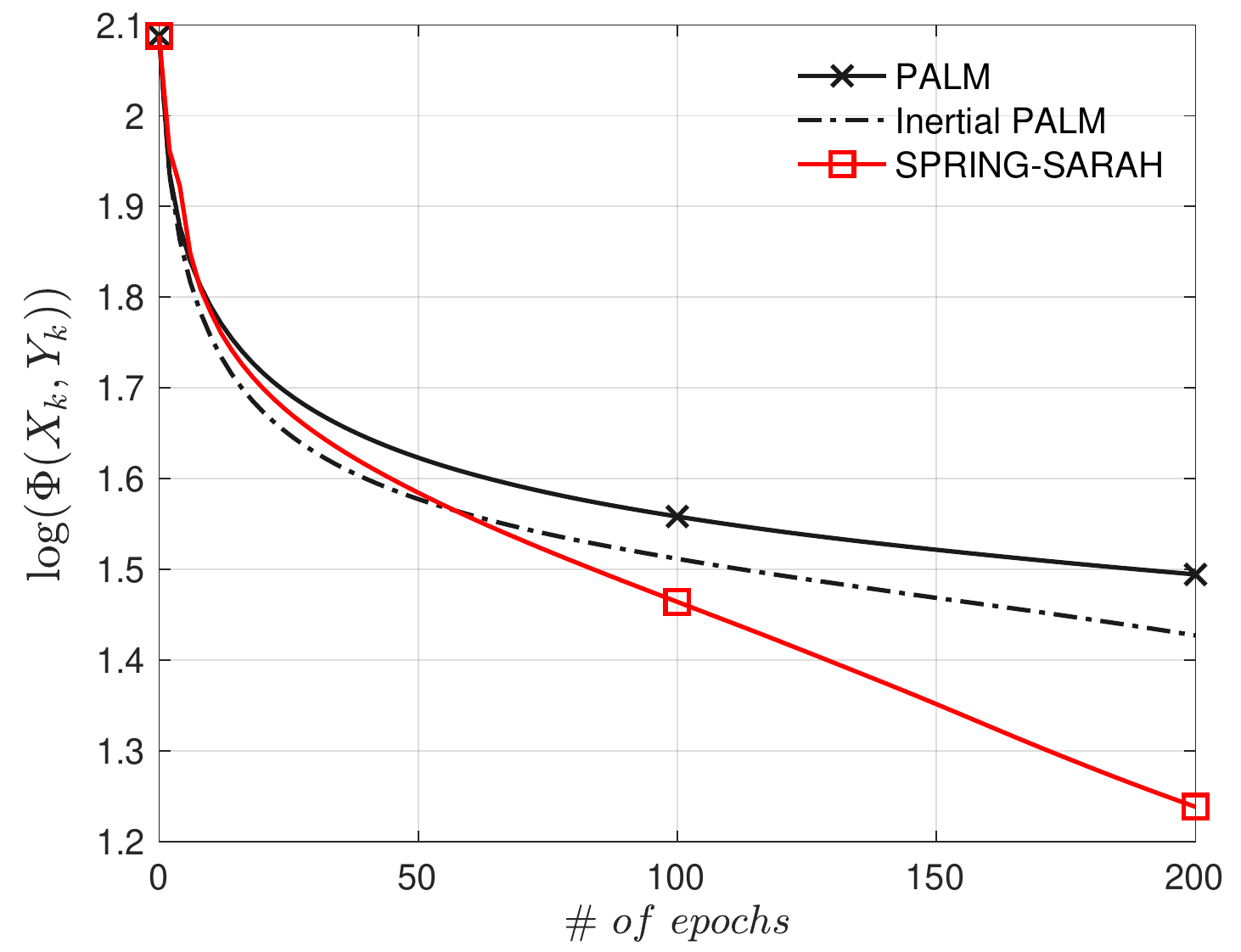} } 	\hspace{3mm}
	\subfloat[{\tt Kodim15}]{ \includegraphics[width=0.425\linewidth]{bid_kodim15.pdf} } \\[-2mm]
	\caption{Objective decrease comparison of blind image-deconvolution experiment on Kodim08 (left) and Kodim15 (right) images using an out-of-focus blur kernel.}
	\label{bid_o}
\end{figure}

\begin{figure}[h!]
	\centering
	\subfloat[Original image]{ \includegraphics[width=0.225\linewidth]{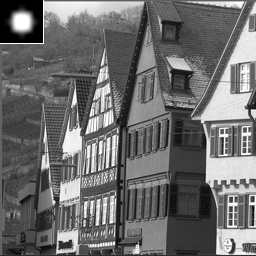} } 	\hspace{1mm}
	\subfloat[Blurred image]{ \includegraphics[width=0.225\linewidth]{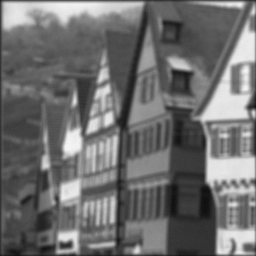} } \hspace{1mm}
	\subfloat[Recovered by PALM]{ \includegraphics[width=0.225\linewidth]{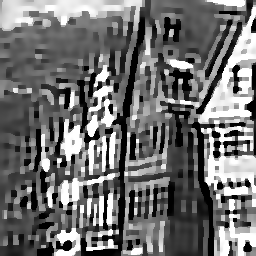} } 	\hspace{1mm}
	\subfloat[Recovered by SPRING]{ \includegraphics[width=0.225\linewidth]{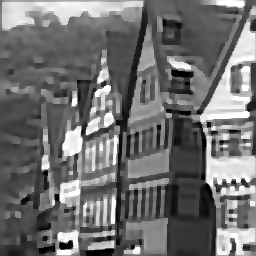} } \\[-2mm]
	\subfloat[Estimated kernel by PALM]{ \includegraphics[width=0.475\linewidth]{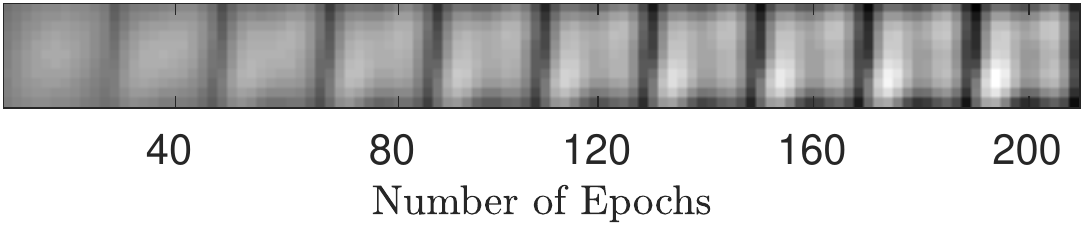} } 	\hspace{1mm}
	\subfloat[Estimated kernel by SPRING]{ \includegraphics[width=0.475\linewidth]{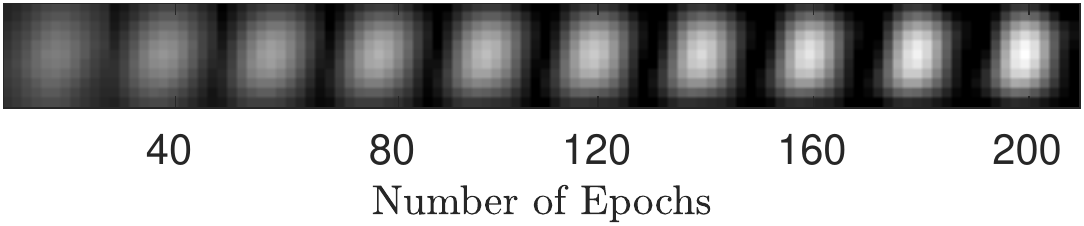} } \\[-2mm]
	\caption{Image and kernel reconstructions from the blind image-deconvolution experiment on the Kodim08 image using an out-of-focus blur kernel.}
	\label{kodak08_outfocus}
\end{figure}

\begin{figure}[h!]
	\centering
	\subfloat[Original image and kernel]{ \includegraphics[width=0.225\linewidth]{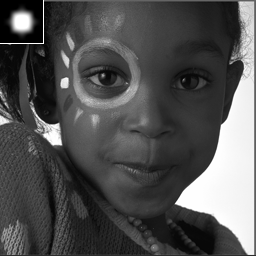} } 	\hspace{1mm}
	\subfloat[Blurred image]{ \includegraphics[width=0.225\linewidth]{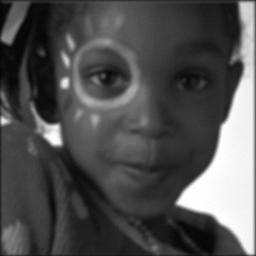} } \hspace{1mm}
	\subfloat[Recovered by PALM]{ \includegraphics[width=0.225\linewidth]{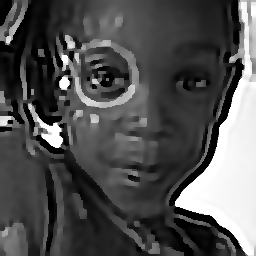} } 	\hspace{1mm}
	\subfloat[Recovered by SPRING]{ \includegraphics[width=0.225\linewidth]{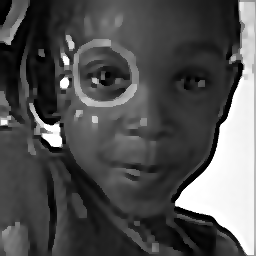} } \\[-2mm]
	\subfloat[Estimated kernel by PALM]{ \includegraphics[width=0.475\linewidth]{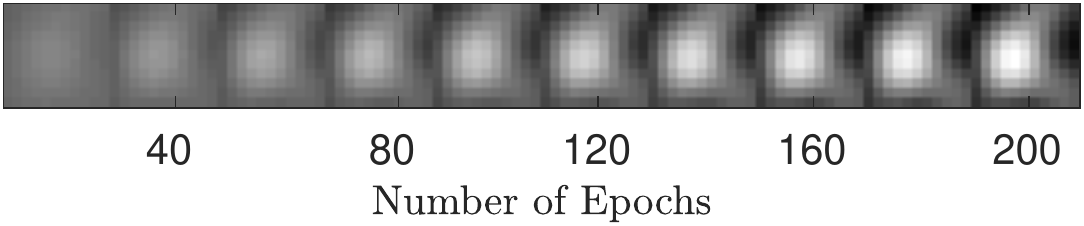} } 		\hspace{1mm}
	\subfloat[Estimated kernel by SPRING]{ \includegraphics[width=0.475\linewidth]{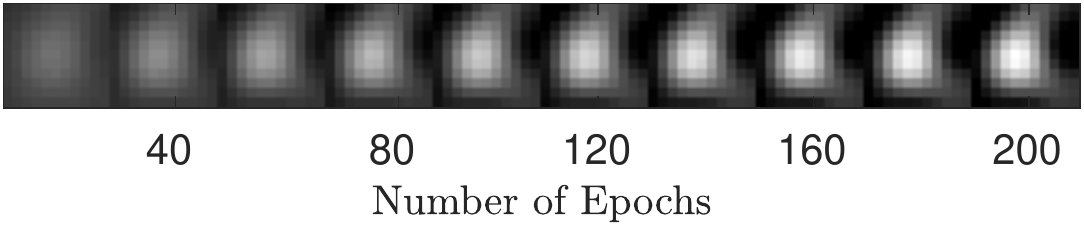} } \\[-2mm]
	\caption{Image and kernel reconstructions from the blind image-deconvolution experiment on the Kodim15 image using an out-of-focus blur kernel.}
	\label{kodak15_outfocus}
\end{figure}

\section{SAGA Variance Bound}
\label{app:sagaconst}

We define the SAGA gradient estimators $\tnablaxsaga$ and $\tnablaysaga$ as follows:
\begin{equation}
\begin{aligned}
    \tnablaxsaga(x_k,y_k) & = \frac{1}{b} \Big( \sum_{j \in J^x_k} \nablax F_j(x_k,y_k) - \nablax F_j(\varphi^j_k,y_k) \Big) + \frac{1}{n} \sum_{i=1}^n \nablax F_i(\varphi_k^i,y_k) \\
    \tnablaysaga(x_{k+1},y_k) & = \frac{1}{b} \Big( \sum_{j \in J^y_k} \nablay F_j(x_{k+1},y_k) - \nablax F_j(x_{k+1},\xi_k^j) \Big) + \frac{1}{n} \sum_{i=1}^n \nablax F_i(x_{k+1},\xi_k^i),
\end{aligned}
\end{equation}
where $J^x_k$ and $J^y_k$ are mini-batches containing $b$ indices. The variables $\varphi_k^i$ and $\xi_k^i$ follow the update rules $\varphi_{k+1}^i = x_k$ if $i \in J^x_k$ and $\varphi_{k+1}^i = \varphi_k^i$ otherwise, and $\xi_{k+1}^i = y_k$ if $i \in J^y_k$ and $\xi_{k+1}^i = \xi_k^i$ otherwise.

To prove our variance bounds, we require the following lemma.
\begin{lemma}
\label{lem:orth}
    Suppose $X_1, \cdots, X_t$ are independent random variables satisfying $\E X_i = 0$ for all $i$. Then
    \begin{equation}
        \E \|X_1 + \cdots + X_t\|^2 = \E [ \|X_1\|^2 + \cdots + \|X_t\|^2 ].
    \end{equation}
\end{lemma}

\begin{proof}
    Our hypotheses on these random variables imply $\E \langle X_i, X_j \rangle = 0$ for $i \not = j$. Therefore,
    \begin{equation}
        \E \|X_1 + \cdots + X_t\|^2 = \sum_{i,j = 1}^t \E \langle X_i, X_j \rangle = \E [ \|X_1\|^2 + \cdots + \|X_t\|^2 ].
    \end{equation}
\end{proof}

We are now prepared to prove that the SAGA gradient estimator is variance-reduced.

\begin{lemma}
    The SAGA gradient estimator satisfies
    \begin{equation}
    \begin{aligned}
        \E \|\tnablaxsaga(x_k,y_k) - \nablax F(x_k,y_k) \|^2 & \le \frac{1}{b n} \sum_{i=1}^n \Big\|\nablax F_i(x_k,y_k) - \nablax F_i(\varphi_k^i,y_k) \Big\|^2, \\
        \E \|\tnablaysaga(x_{k+1},y_k) - \nablay F(x_{k+1},y_k) \|^2 & \le \frac{4}{b n} \sum_{i=1}^n \Big\|\nablay F_i(x_k,y_k) - \nablay F_i(x_k, \xi_k^i) \Big\|^2 \\
        & \quad + \frac{6 M^2}{b} \E \|x_{k+1} - x_k\|^2,
    \end{aligned}
    \end{equation}
    as well as
    \begin{equation}
    \begin{aligned}
        \E \|\tnablaxsaga(x_k,y_k) - \nablax F(x_k,y_k) \| & \le \frac{1}{\sqrt{b n}} \sum_{i=1}^n \Big\|\nablax F_i(x_k,y_k) - \nablax F_i(\varphi_k^i,y_k) \Big\|, \\
        \E \|\tnablaysaga(x_{k+1},y_k) - \nablay F(x_{k+1},y_k) \| & \le \frac{2}{\sqrt{b n}} \sum_{i=1}^n \Big\|\nablay F_i(x_k,y_k) - \nablay F_i(x_k,\xi_k^i) \Big\| \\
        & \quad + \frac{\sqrt{6} M}{\sqrt{b}} \E \|x_{k+1} - x_k\|.
    \end{aligned}
    \end{equation}
\end{lemma}

\begin{proof}
    The proof amounts to computing expectations and applying the Lipschitz continuity of $\nablax F_i$.
\begin{equation}
\begin{aligned}
\label{eq:var}
    & \E \|\tnablaxsaga(x_k,y_k) - \nablax F(x_k,y_k) \|^2 \\
    = & \E \Big\| \frac{1}{b} \sum_{j \in J_k^x} \Big( \nablax F_{j}(x_k,y_k) - \nablax F_{j}(\varphi_k^{j},y_k) \Big) - \nablax F(x_k,y_k) + \frac{1}{n} \sum_{i=1}^n \nablax F_i(\varphi_k^i,y_k) \Big\|^2 \\
    \symnum{1}{\le} & \frac{1}{b^2} \E \sum_{j \in J_k^x} \Big\| \nablax F_{j}(x_k,y_k) - \nablax F_{j}(\varphi_k^{j},y_k) \Big\|^2 \\
    = & \frac{1}{b n} \sum_{i=1}^n \Big\|\nablax F_i(x_k,y_k) - \nablax F_i(\varphi_k^i,y_k) \Big\|^2.
\end{aligned}
\end{equation}
Inequality \numcirc{1} follows from Lemma \ref{lem:orth}.
We can also say that
\begin{equation}
    \begin{aligned}
    \label{eq:vargamma}
    \E \|\tnablaxsaga(x_k,y_k) - \nablax F(x_k,y_k) \| \symnum{1}{\le} & \sqrt{\E \|\tnablaxsaga(x_k,y_k) - \nablax F(x_k,y_k) \|^2} \\
    &\le \frac{1}{\sqrt{b n}} \sqrt{\sum_{i=1}^n \Big\|\nablax F_i(x_k,y_k) - \nablax F_i(\varphi_k^i,y_k) \Big\|^2} \\
    &\le \frac{1}{\sqrt{b n}} \sum_{i=1}^n \Big\|\nablax F_i(x_k,y_k) - \nablax F_i(\varphi_k^i,y_k) \Big\|.
    \end{aligned}
\end{equation}
Inequality \numcirc{1} is Jensen's.

We use an analogous argument for $\tnablaysaga$. Let $\EE_{k,x}$ denote the expectation conditional on the first $k$ iterations and $J^x_k$. By the same reasoning as in \eqref{eq:var},
\begin{equation}
    \EE_{k,x} \|\tnablaysaga(x_{k+1},y_k) - \nablay F(x_{k+1},y_k) \|^2 \le \frac{1}{b n} \sum_{i=1}^n \Big\|\nablay F_i(x_{k+1},y_k) - \nablay F_i(x_{k+1}, \xi_k^i) \Big\|^2.
\end{equation}
Applying the Lipschitz continuity of $\nablay F_i$,
\begin{equation}
\begin{aligned}
    & \frac{1}{b n} \sum_{i=1}^n \Big\|\nablay F_i(x_{k+1},y_k) - \nablay F_i(x_{k+1}, \xi_k^i) \Big\|^2 \\
    & \le \frac{2}{b n} \sum_{i=1}^n \Big\|\nablay F_i(x_{k+1},y_k) - \nablay F_i(x_k, y_k) \Big\|^2 + \frac{2}{b n} \sum_{i=1}^n \Big\|\nablay F_i(x_k,y_k) - \nablay F_i(x_{k+1}, \xi_k^i) \Big\|^2 \\
    & \le \frac{2 M^2}{b} \Big\|x_{k+1} - x_k \Big\|^2 + \frac{4}{b n} \sum_{i=1}^n \Big\|\nablay F_i(x_k,\xi_k^i) - \nablay F_i(x_{k+1}, \xi_k^i) \Big\|^2 + \frac{4}{b n} \sum_{i=1}^n \Big\|\nablay F_i(x_k,y_k) - \nablay F_i(x_k, \xi_k^i) \Big\|^2 \\
    & \le \frac{2 M^2}{b} \Big\|x_{k+1} - x_k \Big\|^2 + \frac{4 M^2}{b} \Big\|x_k - x_{k+1}\Big\|^2 + \frac{4}{b n} \sum_{i=1}^n \Big\|\nablay F_i(x_k,y_k) - \nablay F_i(x_k, \xi_k^i) \Big\|^2.
\end{aligned}
\end{equation}
Also, by the same reasoning as in \eqref{eq:vargamma},
\begin{equation}
    \begin{aligned}
    \EE_{k,x} \|\tnablaysaga(x_{k+1},y_k) - \nablay F(x_{k+1},y_k) \| \symnum{1}{\le} & \sqrt{\EE_{k,x} \|\tnablaysaga(x_{k+1},y_k) - \nablax F(x_{k+1},y_k) \|^2} \\
    &\le \sqrt{ \frac{4}{b n} \sum_{i=1}^n \Big\|\nablay F_i(x_k,y_k) - \nablay F_i(x_k,\xi_k^i) \Big\|^2 + \frac{6 M^2}{b} \|x_{k+1} - x_k\|^2} \\
    &\le \frac{2}{\sqrt{b n}} \sum_{i=1}^n \Big\|\nablay F_i(x_k,y_k) - \nablay F_i(x_k,\xi_k^i) \Big\| + \frac{\sqrt{6} M}{\sqrt{b}} \|x_{k+1} - x_k\|.
    \end{aligned}
\end{equation}
Applying the operator $\E$ to these two inequalities gives the desired result.
\end{proof}

\begin{lemma}
\label{lem:var}
    The SAGA gradient estimator is variance-reduced with
    \begin{equation}
    \begin{aligned}
        \Upsilon_{k+1} &= \frac{1}{b n} \Big( \sum_{i = 1}^n \|\nablax F_i(x_{k+1},y_{k+1}) - \nablax F_i(\varphi_{k+1}^i,y_{k+1})\|^2 + 4 \| \nablay F_i(x_{k+1},y_{k+1}) - \nablay F_i(x_{k+1}, \xi_{k+1}^i) \|^2 \Big), \\
        \Gamma_{k+1} & = \frac{1}{\sqrt{b n}} \Big( \sum_{i=1}^n \| \nablax F_i(x_{k+1},y_{k+1}) - \nablax F_i(\varphi_{k+1}^i,y_{k+1}) \| + 2 \| \nablay F_i(x_{k+1},y_{k+1}) - \nablay F_i(x_{k+1}, \xi_{k+1}^i) \| \Big),
    \end{aligned}
    \end{equation}
    and constants $V_1 = 6 M^2/b$, $V_2 = \sqrt{6} M / \sqrt{b}$, $V_\Upsilon = \frac{134 n L^2}{b^2}$, and $\rho = \frac{b}{2 n}$.
\end{lemma}

\begin{proof}
We must show that $\E \Upsilon_{k+1}$ decreases at a geometric rate. We first bound the MSE of the estimator $\tnablaxsaga$. Applying the inequality $\|a-c\|^2 \le (1+\delta)\|a-b\|^2 + (1+\delta^{-1})\|b-c\|^2$ twice,
\begin{equation}
\begin{aligned}
    & \sfrac{1}{b n} \sum_{i=1}^n \E \Big\|\nablax F_i(x_{k+1},y_{k+1}) - \nablax F_i(\varphi_{k+1}^i,y_{k+1}) \Big\|^2 \\
    & \le \sfrac{1 + \delta}{b n} \E \sum_{i=1}^n \Big\|\nablax F_i(x_k,y_k) - \nablax F_i(\varphi_{k+1}^i,y_{k+1}) \Big\|^2 + \sfrac{1 + \delta^{-1}}{b n} \sum_{i=1}^n \|\nablax F_i(x_{k+1},y_{k+1}) - \nablax F_i(x_k,y_k)\|^2 \\
    & \le \tfrac{(1 + \delta)^2}{b n} \E \sum_{i=1}^n \Big\|\nablax F_i(x_k,y_k) - \nablax F_i(\varphi_{k+1}^i,y_k) \Big\|^2 + \tfrac{(1 + \delta^{-1})(1+\delta)}{b n} \E \sum_{i=1}^n \Big\|\nablax F_i(\varphi_{k+1},y_{k+1}) - \nablax F_i(\varphi_{k+1}^i,y_k) \Big\|^2 \\
    & \quad + \tfrac{1 + \delta^{-1}}{b n} \sum_{i=1}^n \|\nablax F_i(x_{k+1},y_{k+1}) - \nablax F_i(x_k,y_k)\|^2.
    \end{aligned}
\end{equation}
Next, we compute the expectation of the first term.
\begin{equation}
\begin{aligned}
    & \le \tfrac{(1 + \delta)^2 (1 - b / n)}{b n} \sum_{i=1}^n \Big\|\nablax F_i(x_k,y_k) - \nablax F_i(\varphi_k^i,y_k) \Big\|^2 \\
    & + \tfrac{(1 + \delta^{-1})(1+\delta)}{b n} \E \sum_{i=1}^n \Big\|\nablax F_i(\varphi_{k+1}^i,y_{k+1}) - \nablax F_i(\varphi_{k+1}^i,y_k) \Big\|^2  + \tfrac{1 + \delta^{-1}}{b n} \sum_{i=1}^n \|\nablax F_i(x_{k+1},y_{k+1}) - \nablax F_i(x_k,y_k)\|^2 \\
    & \le \tfrac{(1 + \delta)^2 (1 - b / n)}{b n} \sum_{i=1}^n \Big\|\nablax F_i(x_k,y_k) - \nablax F_i(\varphi_k^i,y_k) \Big\|^2 + \tfrac{(1 + \delta^{-1})(1+\delta) M^2}{b} \E \Big\|y_{k+1} - y_k \Big\|^2 \\
    & \quad + \tfrac{(1 + \delta^{-1}) M^2}{b} \E \|z_{k+1} - z_k\|^2.
\end{aligned}
\end{equation}

We bound the MSE of the estimator $\tnablaysaga$ similarly.
\begin{equation}
\begin{aligned}
    & \tfrac{1}{b n} \sum_{i=1}^n \E \Big\|\nablay F_i(x_{k+1},y_{k+1}) - \nablay F_i(x_{k+1}, \xi_{k+1}^i) \Big\|^2 \\
    & \le \tfrac{1 + \delta}{b n} \E \sum_{i=1}^n \Big\|\nablay F_i(x_{k+1},y_k) - \nablay F_i(x_{k+1}, \xi_{k+1}^i) \Big\|^2 + \tfrac{1 + \delta^{-1}}{b n} \E \sum_{i=1}^n \|\nablay F_i(x_{k+1},y_{k+1}) - \nablay F_i(x_{k+1},y_k)\|^2 \\
    & = \tfrac{(1 + \delta)(1 - b / n)}{b n} \E \sum_{i=1}^n \Big\|\nablay F_i(x_{k+1},y_k) - \nablay F_i(x_{k+1}, \xi_k^i) \Big\|^2 + \tfrac{1 + \delta^{-1}}{b n} \E \sum_{i=1}^n \|\nablay F_i(x_{k+1},y_{k+1}) - \nablay F_i(x_{k+1},y_k)\|^2 \\
    & \le \tfrac{(1 + \delta)^2 (1 - b / n)}{b n} \E \sum_{i=1}^n \Big\|\nablay F_i(x_k,y_k) - \nablay F_i(x_{k+1}, \xi_k^i) \Big\|^2 + \tfrac{1 + \delta^{-1}}{b n} \E \sum_{i=1}^n \|\nablay F_i(x_{k+1},y_{k+1}) - \nablay F_i(x_{k+1},y_k)\|^2 \\
    & \quad + \tfrac{(1+\delta)(1 + \delta^{-1})(1 - b / n)}{b n} \E \sum_{i=1}^n \|\nablay F_i(x_{k+1},y_k) - \nablay F_i(x_k,y_k)\|^2 \\
    & \le \tfrac{(1 + \delta)^3 (1 - b / n)}{b n} \E \sum_{i=1}^n \Big\|\nablay F_i(x_k,y_k) - \nablay F_i(x_k, \xi_k^i) \Big\|^2 + \tfrac{1 + \delta^{-1}}{b n} \E \sum_{i=1}^n \|\nablay F_i(x_{k+1},y_{k+1}) - \nablay F_i(x_{k+1},y_k)\|^2 \\
    & \quad + \tfrac{(1+\delta)(1 + \delta^{-1})(1 - b / n)}{b n} \E \sum_{i=1}^n \|\nablay F_i(x_{k+1},y_k) - \nablay F_i(x_k,y_k)\|^2 \\
    & \quad + \tfrac{(1+\delta)^2(1 + \delta^{-1})(1 - b / n)}{b n} \E \sum_{i=1}^n \|\nablay F_i(x_{k+1},\xi_k^i) - \nablay F_i(x_k,\xi_k^i)\|^2,
    \end{aligned}
\end{equation}

and, by the Lipschitz continuity of $\nablay F_i$,

\begin{equation}
\begin{aligned}
    & \tfrac{1}{b n} \sum_{i=1}^n \E \Big\|\nablay F_i(x_{k+1},y_{k+1}) - \nablay F_i(x_{k+1}, \xi_{k+1}^i) \Big\|^2 \\
    & \le \frac{(1 + \delta)^3 (1 - b / n)}{b n} \E \sum_{i=1}^n \Big\|\nablay F_i(x_k,y_k) - \nablay F_i(x_k, \xi_k^i) \Big\|^2 + \frac{(1 + \delta^{-1}) L^2_y}{b} \E \|y_{k+1} - y_k\|^2 \\
    & \quad + \frac{(1+\delta)(1 + \delta^{-1})(1 - b / n) M^2}{b} \E \|x_{k+1} - x_k\|^2 \\
    & \quad + \frac{(1+\delta)^2 (1 + \delta^{-1})(1 - b / n) M^2}{b} \E \|x_{k+1} - x_k\|^2.
\end{aligned}
\end{equation}
With
\begin{equation}
    \Upsilon_{k+1} = \frac{1}{b n} \Big( \sum_{i = 1}^n \|\nablax F_i(x_{k+1},y_{k+1}) - \nablax F_i(\varphi_{k+1}^i,y_{k+1})\|^2 + 4 \| \nablay F_i(x_{k+1},y_{k+1}) - \nablay F_i(x_{k+1}, \xi_{k+1}^i) \|^2 \Big),
\end{equation}
we can now say
\begin{equation}
    \begin{aligned}
        \E \Upsilon_{k+1} &\le (1 + \delta)^3 (1 - b / n) \Upsilon_k + \frac{4 (1 + \delta^{-1}) L^2_y}{b} \E \|y_{k+1} - y_k\|^2 \\
        & \quad + \frac{8 (1+\delta)^2(1 + \delta^{-1})(1 - b / n) M^2}{b} \E \|x_{k+1} - x_k\|^2 \\
        & \quad + \frac{(1+\delta)(1 + \delta^{-1}) M^2}{b} \E \|y_{k+1} - y_k\|^2 + \frac{(1 + \delta^{-1}) M^2}{b} \E \|z_{k+1} - z_k\|^2 \\
        &\le (1 + \delta)^3 (1 - b / n) \Upsilon_k + \frac{14 (1 + \delta)^2 (1 + \delta^{-1}) L^2}{b} \E [\|z_{k+1} - z_k\|^2],
    \end{aligned}
    \end{equation}
    where $L \defeq \max\{L_x,L_y,M\}$. Choosing $\delta = \frac{b}{6 n}$, we are ensured that $(1 + \delta)^3 (1 - b / n) \le 1 - \frac{b}{2 n}$, producing the inequality
    \begin{equation}
    \begin{aligned}
        \E \Upsilon_{k+1} & \le (1 - \sfrac{b}{2 n}) \Upsilon_k + \frac{14 (1 + \frac{b}{6 n})^2 (6 n / b + 1) L^2}{b} \E [\|z_{k+1} - z_k\|^2] \\
        & \le (1 - \sfrac{b}{2 n}) \Upsilon_k + \frac{134 n L^2}{b^2} \E [\|z_{k+1} - z_k\|^2].
    \end{aligned}
    \end{equation}
    This proves the geometric decay of $\Upsilon_k$ in expectation.

    All that is left is to show that if $\EE \|z_k - z_{k-1}\|^2 \to 0$, then so do $\Upsilon_k$ and $\Gamma_k$. We begin by showing that $\sum_{i=1}^n \EE \|\nablax F_i(x_k,y_k) - \nablax F_i(\varphi_k^i,y_k)\|^2 \to 0$.
    \begin{equation}
        \begin{aligned}
            \sum_{i=1}^n \EE \|\nablax F_i(x_k,y_k) - \nablax F_i(\varphi_k^i,y_k)\|^2 & \le L_x^2 \sum_{i=1}^n \EE \|x_k - \varphi_k^i\|^2 \\
            & \le L_x^2 n \Big(1 + \tfrac{2 n}{b} \Big) \EE \|x_k - x_{k-1}\|^2 + \Big( 1 + \tfrac{b}{2 n} \Big) \sum_{i=1}^n \EE \|x_{k-1} - \varphi_k^i\|^2 \\
            & \le L_x^2 n \Big(1 + \tfrac{2 n}{b} \Big) \EE \|x_k - x_{k-1}\|^2 + \Big( 1 + \tfrac{b}{2 n} \Big) \Big( 1 - \tfrac{b}{n} \Big) \sum_{i=1}^n \EE \|x_{k-1} - \varphi_{k-1}^i\|^2 \\
            & \le L_x^2 n \Big(1 + \tfrac{2 n}{b} \Big) \EE \|x_k - x_{k-1}\|^2 + \Big( 1 - \tfrac{b}{2 n} \Big) \sum_{i=1}^n \EE \|x_{k-1} - \varphi_{k-1}^i\|^2 \\
            & \le L_x^2 n \Big( 1 + \tfrac{2 n}{b} \Big) \sum_{\ell=1}^k \Big(1 - \tfrac{b}{2 n} \Big)^{k-\ell} \EE \|x_\ell - x_{\ell-1}\|^2.
        \end{aligned}
    \end{equation}
    Because $\mathbb{E} \|x_k - x_{k-1}\|^2 \to 0$, it is clear that the bound on the right goes to zero as $k \to \infty$. An analogous argument shows that $\sum_{i=1}^n \EE \|\nablax F_i(x_k,y_k) - \nablax F_i(x_k, \xi_k^i)\|^2 \to 0$ as well. The fact that $\EE \Gamma_k \to 0$ follows similarly:
    \begin{equation}
    \begin{aligned}
        \sum_{i=1}^n \EE \|\nablax F_i(x_k,y_k) - \nablax F_i(\varphi_k^i,y_k)\| & \le L_x \sum_{i=1}^n \EE \|x_k - \varphi_k^i\| \\
        & \le n L_x \|x_k - x_{k-1}\| + \sum_{i=1}^n \EE \|x_{k-1} - \varphi_k^i \| \\
        & \le n L_x \|x_k - x_{k-1}\| + \Big( 1- \frac{b}{n} \Big) \sum_{i=1}^n \EE \|x_{k-1} - \varphi_{k-1}^i \| \\
        & \le n L_x \sum_{\ell = 1}^k \Big( 1 - \frac{b}{n} \Big)^{k - \ell} \EE \|x_\ell - x_{\ell-1}\|.
    \end{aligned}
    \end{equation}
    As $\|x_k - x_{k-1}\|^2 \to 0$, it follows that $\|x_k - x_{k-1}\| \to 0$ (because Jensen's inequality implies $\EE \|x_k - x_{k-1}\| \le \sqrt{\EE \|x_k - x_{k-1}\|^2} \to 0$), so the bound above implies $\EE \Gamma_k \to 0$ as well.

\end{proof}

\section{SARAH Variance Bound}
\label{app:sarahconst}

As in the previous section, we use $J_k^x$ to denote the mini-batches used to approximate $\nablax F(x_k,y_k)$, and we use $J_k^y$ to denote the mini-batches used to approximate $\nablay F(x_{k+1},y_k)$.

\begin{lemma}
    The SARAH gradient estimator is variance reduced with
    \begin{equation}
    \begin{aligned}
        \Upsilon_{k+1} & = \|\tnablax^{\textnormal{\tiny SARAH}} (x_k, y_k) - \nablax F(x_k, y_k)\|^2 + \| \tnablay^{\textnormal{\tiny SARAH}} (x_{k+1}, y_k) - \nablay F(x_{k+1}, y_k) \|^2, \\
        \Gamma_{k+1} & = \|\tnablax^{\textnormal{\tiny SARAH}} (x_k, y_k) - \nablax F(x_k, y_k)\| + \| \tnablay^{\textnormal{\tiny SARAH}} (x_{k+1}, y_k) - \nablay F(x_{k+1}, y_k) \|,
    \end{aligned}
    \end{equation}
    and constants $\rho = 1 / p$, $V_1 = V_\Upsilon = 2 L^2$, and $V_2 = 2 L$.
\end{lemma}

\begin{proof}

Let $\mathbb{E}_{k,p}$ denote the expectation conditional on the first $k$ iterations and the event that we do not compute the full gradient at iteration $k$. The conditional expectation of the SARAH gradient estimator in this case is
\begin{equation}
\begin{aligned}
    \mathbb{E}_{k,p} \tnablaxsarah(x_k,y_k) & = \frac{1}{b} \mathbb{E}_{k,p} \Big(\sum_{j \in J_k^x} \nablax F_j(x_k,y_k) - F_j(x_{k-1}, y_{k-1}) \Big) + \tnablaxsarah (x_{k-1}, y_{k-1}) \\
    & = \nablax F(x_k,y_k) - \nablax F(x_{k-1}, y_{k-1}) + \tnablaxsarah (x_{k-1}, y_{k-1}).
\end{aligned}
\end{equation}
We begin with a bound on $\EE_{k,p} \|\tnablax^{\textnormal{\tiny SARAH}} (x_k, y_k) - \nablax F(x_k, y_k)\|^2$.
\begin{equation}
\begin{aligned}
    & \mathbb{E}_{k,p} \|\tnablax^{\textnormal{\tiny SARAH}} (x_k, y_k) - \nablax F(x_k, y_k)\|^2 \notag \\
    = & \mathbb{E}_{k,p} \| \tnablaxsarah (x_{k-1}, y_{k-1}) - \nablax F(x_{k-1}, y_{k-1}) + \nablax F(x_{k-1}, y_{k-1}) - \nablax F(x_k, y_k) \\
    & \quad \quad \quad \quad + \tnablaxsarah (x_k, y_k) - \tnablaxsarah (x_{k-1}, y_{k-1}) \|^2 \\
    = & \Big\| \tnablaxsarah (x_{k-1}, y_{k-1}) - \nablax F(x_{k-1}, y_{k-1}) \Big\|^2 + \Big\|\nablax F(x_{k-1},y_{k-1}) - \nablax F(x_k, y_k) \Big\|^2 \\
    & \quad + \mathbb{E}_{k,p} \Big\| \tnablaxsarah (x_k, y_k) - \tnablaxsarah (x_{k-1}, y_{k-1}) \Big\|^2 \notag \\
    & \quad + 2 \langle \nablax F(x_{k-1}, y_{k-1}) - \tnablaxsarah (x_{k-1}, y_{k-1}), \nablax F(x_k, y_k) - \nablax F(x_{k-1}, y_{k-1}) \rangle \\
    & \quad - 2 \Big\langle \nablax F(x_{k-1}, y_{k-1}) - \tnablaxsarah (x_{k-1}, y_{k-1}), \mathbb{E}_{k,p} \Big[ \tnablaxsarah (x_k,y_k) - \tnablaxsarah (x_{k-1}, y_{k-1}) \Big] \Big\rangle \\
    & \quad - 2 \Big\langle \nablax F(x_k, y_k) - \nablax F(x_{k-1}, y_{k-1}), \mathbb{E}_{k,p} \Big[ \tnablaxsarah (x_k,y_k) - \tnablaxsarah (x_{k-1}, y_{k-1}) \Big] \Big\rangle.
\end{aligned}
\end{equation}
To simplify the inner-product terms, we use the fact that
\begin{equation}
\begin{aligned}
    & \mathbb{E}_{k,p}[\tnablaxsarah (x_k, y_k) - \tnablaxsarah (x_{k-1}, y_{k-1})] = \nablax F(x_k, y_k) - \nablax F(x_{k-1}, y_{k-1}).
\end{aligned}
\end{equation}
With this equality established, we see that the second inner product is equal to
\begin{equation}
\begin{aligned}
    & - 2 \Big\langle \nablax F(x_{k-1}, y_{k-1}) - \tnablaxsarah (x_{k-1}, y_{k-1}), \mathbb{E}_{k,p} \Big[ \tnablaxsarah (x_k, y_k) - \tnablaxsarah (x_{k-1}, y_{k-1}) \Big] \Big\rangle \notag \\
    = & - 2 \langle \nablax F(x_{k-1}, y_{k-1}) - \tnablaxsarah (x_{k-1}, y_{k-1}), \nablax F(x_k, y_k) - \nablax F(x_{k-1}, y_{k-1}) \rangle,
\end{aligned}
\end{equation}
so the first two inner-products sum to zero. The third inner product is equal to
\begin{equation}
\begin{aligned}
    & - 2 \Big\langle \nablax F(x_k, y_k) - \nablax F(x_{k-1}, y_{k-1}), \mathbb{E}_{k,p} \Big[ \tnablaxsarah (x_k, y_k) - \tnablaxsarah (x_{k-1}, y_{k-1}) \Big] \Big\rangle \notag \\
    = & - 2 \langle \nablax F(x_k, y_k) - \nablax F(x_{k-1}, y_{k-1}), \nablax F(x_k, y_k) - \nablax F(x_{k-1}, y_{k-1}) \rangle \\
    = & - 2 \| \nablax F(x_k, y_k) - \nablax F(x_{k-1}, y_{k-1}) \|^2.
\end{aligned}
\end{equation}
Altogether, we have
\begin{equation}
\begin{aligned}
    & \mathbb{E}_{k,p} \|\tnablaxsarah (x_k, y_k) - \nablax F(x_k, y_k)\|^2 \\
    &\le \Big\| \tnablaxsarah (x_{k-1}, y_{k-1}) - \nablax F(x_{k-1}, y_{k-1}) \Big\|^2 - \|\nablax F(x_k, y_k) - \nablax F(x_{k-1}, y_{k-1})\|^2 \notag \\
    & + \mathbb{E}_{k,p} \| \tnablaxsarah (x_k, y_k) - \tnablaxsarah (x_{k-1}, y_{k-1})\|^2 \\
    &\le \Big\| \tnablaxsarah (x_{k-1}, y_{k-1}) - \nablax F(x_{k-1}, y_{k-1}) \Big\|^2 + \mathbb{E}_{k,p} \| \tnablaxsarah (x_k, y_k) - \tnablaxsarah (x_{k-1}, y_{k-1})\|^2.
\end{aligned}
\end{equation}
We can bound the second term by computing the expectation.
\begin{equation}
\begin{aligned}
    \mathbb{E}_{k,p} \| \tnablaxsarah (x_k, y_k) - \tnablaxsarah (x_{k-1}, y_{k-1})\|^2 = & \mathbb{E}_{k,p} \Big\| \frac{1}{b} \Big(\sum_{j \in J_k^x} \nablax F_j(x_k, y_k) - \nablax F_j(x_{k-1}, y_{k-1}) \Big) \Big\|^2 \\
    &\le \frac{1}{b} \mathbb{E}_{k,p} \Big[\sum_{j \in J_k^x} \| \nablax F_j(x_k, y_k) - \nablax F_j(x_{k-1}, y_{k-1})\|^2 \Big] \\
    = & \frac{1}{n} \sum_{i=1}^n \| \nablax F_i(x_k, y_k) - \nablax F_i(x_{k-1}, y_{k-1})\|^2.
\end{aligned}
\end{equation}
The inequality is due to the convexity of the function $x \mapsto \|x\|^2$. This results in the recursive inequality
\begin{equation}
\begin{aligned}
    & \mathbb{E}_{k,p} \|\tnablaxsarah(x_k,y_k) - \nablax F(x_k, y_k)\|^2 \notag \\
    &\le \Big\| \tnablaxsarah (x_{k-1}, y_{k-1}) - \nablax F(x_{k-1}, y_{k-1}) \Big\|^2 + \frac{1}{n} \sum_{i=1}^n  \| \nablax F_i(x_k, y_k) - \nablax F_i(x_{k-1}, y_{k-1})\|^2.
\end{aligned}
\end{equation}
This bounds the MSE under the condition that the full gradient is not computed. When the full gradient is computed, the MSE is equal to zero, so
\begin{equation}
\begin{aligned}
    & \E \|\tnablaxsarah (x_k, y_k) - \nablax F(x_k, y_k)\|^2 \notag \\
    &\le \Big(1 - \frac{1}{p} \Big) \Big( \Big\| \tnablaxsarah (x_{k-1}, y_{k-1}) - \nablax F(x_{k-1}, y_{k-1}) \Big\|^2 + \frac{1}{n} \sum_{i=1}^n  \| \nablax F_i(x_k, y_k) - \nablax F_i(x_{k-1}, y_{k-1})\|^2 \Big) \\
    &\le \Big(1 - \frac{1}{p} \Big) \Big\| \tnablaxsarah (x_{k-1}, y_{k-1}) - \nablax F(x_{k-1}, y_{k-1}) \Big\|^2 + M^2 \|z_k - z_{k-1}\|^2.
\end{aligned}
\end{equation}
By symmetric arguments, analogous results hold for $\E \|\tnablaysarah (x_{k+1}, y_k) - \nablay F(x_{k+1}, y_k)\|^2$:
\begin{equation}
\begin{aligned}
    & \E \|\tnablaysarah (x_{k+1}, y_k) - \nablay F(x_{k+1}, y_k)\|^2 \notag \\
    &\le \Big(1 - \frac{1}{p} \Big) \Big\| \tnablaysarah (x_k, y_{k-1}) - \nablay F(x_k, y_{k-1}) \Big\|^2 + M^2 (\E \| x_{k+1} - x_k \|^2 + \|y_k - y_{k-1}\|^2).
\end{aligned}
\end{equation}
Combining the two inequalities above, we have shown
\begin{equation}
    \begin{aligned}
    & \E [ \|\tnablaxsarah (x_k, y_k) - \nablax F(x_k, y_k)\|^2 + \|\tnablaysarah (x_{k+1}, y_k) - \nablay F(x_{k+1}, y_k)\|^2 ] \\
    &\le \Big( 1 - \frac{1}{p} \Big) \Big( \|\tnablaxsarah (x_{k-1}, y_{k-1}) - \nablax F(x_{k-1}, y_{k-1})\|^2 + \|\tnablaysarah (x_k, y_{k-1}) - \nablay F(x_k, y_{k-1})\|^2 \Big) \\
    & + 2 L^2 \E [ \|z_{k+1} - z_k\|^2 + \|z_k - z_{k-1}\|^2 ]
    \end{aligned}
\end{equation}
We have also established the geometric decay property:
\begin{equation}
    \begin{aligned}
        \E \Upsilon_{k+1} \le \Big( 1 - \frac{1}{p} \Big) \Upsilon_k + 2 L^2 \E [ \|z_{k+1} - z_k\|^2 + \|z_k - z_{k-1}\|^2 ],
    \end{aligned}
\end{equation}
justifying the choice of constants $\rho = 1 / p$ and $V_1 = V_\Upsilon = 2 L^2$. Similar bounds hold for $\Gamma_k$ due to Jensen's inequality:
\begin{equation}
\begin{aligned}
    & \E \|\tnablaxsarah (x_k, y_k) - \nablax F(x_k, y_k)\| \notag \\
    &\le \sqrt{\E \|\tnablaxsarah (x_k, y_k) - \nablax F(x_k, y_k)\|^2} \notag \\
    &\le \sqrt{\Big(1 - \frac{1}{p} \Big) \Big\| \tnablaxsarah (x_{k-1}, y_{k-1}) - \nablax F(x_{k-1}, y_{k-1}) \Big\|^2 + M^2 \| z_k - z_{k-1} \|^2} \\
    &\le \sqrt{\Big(1 - \frac{1}{p} \Big)} \Big\| \tnablaxsarah (x_{k-1}, y_{k-1}) - \nablax F(x_{k-1}, y_{k-1}) \Big\| + M \|z_k - z_{k-1}\|.
\end{aligned}
\end{equation}
Applying an analogous result for $\tnablay$ gives the desired bound on $\Gamma_k$.

It is also easy to see that $\EE \|z_k - z_{k-1}\|^2 \to 0$ implies $\EE \Upsilon_k \to 0$:
\begin{equation}
    \begin{aligned}
        \EE \Upsilon_k & \le \Big( 1 - \frac{1}{p} \Big) \EE \Upsilon_{k-1} + 2 L^2 \EE [ \|z_{k+1} - z_k\|^2 + \|z_k - z_{k-1}\|^2 ] \\
        & \le 2 L^2 \sum_{\ell=1}^k \Big(1 - \frac{1}{p} \Big)^{k - \ell} \EE [ \|z_{\ell+1} - z_\ell\|^2 + \|z_\ell - z_{\ell-1}\|^2.
    \end{aligned}
\end{equation}
As $\EE \|\tnablaxsarah (x_k, y_k) - \nablax F(x_k, y_k)\|^2 \to 0$, so does $\EE \|\tnablaxsarah (x_k, y_k) - \nablax F(x_k, y_k)\| \to 0$ by Jensen's inequality, so it is clear that $\EE \Gamma_k \to 0$ as well.

\end{proof}

\end{small}

\end{document}